\documentclass{article}
\usepackage {amsmath, amsfonts, amssymb, mathrsfs, amsthm,stmaryrd, sectsty,hyphenat,calc}
\usepackage[usenames]{color}
\usepackage{palatino}
\usepackage[all]{xy}
\def\mf#1{\mathfrak{#1}}
\def\mc#1{\mathcal{#1}}
\def\mb#1{\mathbb{#1}}
\def\tx#1{\textrm{#1}}
\def\tb#1{\textbf{#1}}

\def\R{\mathbb{R}}

\def\C{\mathbb{C}}
\def\Q{\mathbb{Q}}

\def\Z{\mathbb{Z}}
\def\N{\mathbb{N}}

\def\ol#1{\overline{#1}}
\def\ul#1{\underline{#1}}

\def\hat{\widehat}

\def\rw{\rightarrow}

\def\lrw{\longrightarrow}

\def\irw{\hookrightarrow}

\def\<{\langle}
\def\>{\rangle}

\newenvironment{mytitle}
{\begin{center}\large\sc}
{\end{center}}

\newtheorem{thm}{Theorem}[section]
\newtheorem{lem}[thm]{Lemma}
\newtheorem{pro}[thm]{Proposition}
\newtheorem{cor}[thm]{Corollary}

\newtheorem{fct}[thm]{Fact}

\sectionfont{\center\sc\normalsize}
\subsectionfont{\bf\normalsize}

\numberwithin{equation}{section}

\newlength{\sumcorr}
\def\sprod#1{\setlength{\sumcorr}{(\widthof{$\displaystyle\prod_{#1}$}-\widthof{$\displaystyle\prod$})/2} \hspace{-\sumcorr}\prod_{#1}\hspace{-\sumcorr} }

\begin{document}

\begin{mytitle} Rigid inner forms of real and $p$-adic groups \end{mytitle}
\begin{center} Tasho Kaletha \end{center}
\begin{abstract}
We define a new cohomology set $H^1(u \rw W,Z \rw G)$ for an affine algebraic group $G$ and a finite central subgroup $Z$, both defined over a local field of characteristic zero, which is an enlargement of the usual first Galois cohomology set of $G$. We show how this set can be used to give a precise conjectural description of the internal structure and endoscopic transfer of tempered $L$-packets for arbitrary connected reductive groups that extends the well-known conjectural description for quasi-split groups. In the case of real groups, we show that this description is correct using Shelstad's work.
\end{abstract}
{\let\thefootnote\relax\footnotetext{This research is supported in part by NSF grant DMS-1161489.}}

\section{Introduction}
The principal goal of this paper is to give a precise conjectural description of the internal structure of tempered $L$-packets and the character identities satisfied by them for an arbitrary connected reductive group defined over a local field $F$ of characteristic zero, and then to prove that this description is correct when $F=\R$. For a quasi-split group $G$, such a description has been available for some time. Indeed, let $\Gamma$, $W_F$, and $W_F'$ be the absolute Galois, Weil, and Weil-Deligne groups of $F$, let $\hat G$ be the connected complex Langlands dual group of $G$, and let $^LG = \hat G \rtimes W_F$ be the Weil-form of its $L$-group. Given a tempered Langlands parameter $\varphi : W_F' \rw {^LG}$, one expects the existence of a finite set $\Pi_\varphi^G$ of irreducible admissible representations of the topological group $G(F)$. These finite sets, called $L$-packets, are supposed to satisfy a number of properties, some of which are listed in \cite[\S10]{Bo77}. Among the most important properties are their internal parameterization and the endoscopic character identities. These tie the $L$-packets to the stabilization of the spectral side of the Arthur-Selberg trace formula and lead to the multiplicity formula for discrete automorphic representations. The conjectural internal parameterization is the following. First, Shahidi's tempered $L$-packet conjecture \cite[\S9]{Sh90} states that for a fixed Whittaker datum $\mf{w}$ of $G$ the set $\Pi_\varphi^G$ should contain a unique $\mf{w}$-generic representation. Second, if we let $S_\varphi$ denote the centralizer in $\hat G$ of the image of $\varphi$, it is expected that there exists an injection (bijection if $F$ is $p$-adic)
\begin{equation} \label{eq:qb} \iota_\mf{w}: \Pi_\varphi^G \rw \tx{Irr}(\pi_0(S_\varphi/Z(\hat G)^\Gamma)), \end{equation}
where the right hand side is the set of isomorphism classes of irreducible representations of the finite group $\pi_0(S_\varphi/Z(\hat G)^\Gamma)$. This map should however not be arbitrary. It should send the unique $\mf{w}$-generic constituent of $\Pi_\varphi^G$ to the trivial representation and should, moreover, provide the correct relationship between the Harish-Chandra characters of the constituents of $\Pi_\varphi^G$ and the characters of the representations of $\pi_0(S_\varphi/Z(\hat G)^\Gamma)$ so that the endoscopic character identities hold. This conjecture has been established for $F=\R$ and general quasi-split connected reductive groups by Langlands and Shelstad \cite{Lan73}, \cite{She79a}, \cite{She79b}, \cite{She81}, \cite{She82}, \cite{SheTE3}
and for a finite extension $F/\Q_p$ and quasi\-split symplectic and orthogonal groups by Arthur \cite{Art11}. Important for both the statement and the proof of the conjecture is the fact that the same datum that is used to fix the bijection \eqref{eq:qb}, namely the Whittaker datum, also leads to a normalization of the endoscopic transfer factors that enter the formulation of the character identities.

The situation for groups $G'$ which are not quasi-split is more subtle. Let $G$ be the (unique up to isomorphism) quasi-split inner form of $G'$. If we fix an isomorphism $\xi : G \rw G'$ defined over $\ol{F}$ such that $\xi^{-1}\sigma(\xi)$ is an inner automorphism for all $\sigma \in \Gamma$, then $\xi$ can be used to identify the $L$-groups of $G$ and $G'$. The parameter $\varphi$ now becomes a parameter for $G'$ and we may ask for an analog of \eqref{eq:qb}. Such an analog has to depend not just on $G'$, but also on $\xi$. The tuple $(G',\xi)$ is called an inner twist of $G$ and the set of isomorphisms of inner twists is parameterized by $H^1(\Gamma,G/Z(G))$. It was shown by Kottwitz \cite{Kot86} that there is a canonical map from this set to the Pontryagin dual of the finite abelian group $Z(\hat G_\tx{sc})^\Gamma$ -- the $\Gamma$-fixed points of the center of the simply-connected cover of the derived subgroup of $\hat G$. One can now try to formulate a conjectural injection similar to \eqref{eq:qb} in terms of a variant of $S_\varphi$ involving $\hat G_\tx{sc}$, making a reference to the character of $Z(\hat G_\tx{sc})^\Gamma$ to which $\xi$ corresponds. However Vogan \cite{Vog93} and Arthur \cite{Art06} observe that such an attempt cannot be successful. Arthur's point of departure is the fact that on a non-quasi-split group $G'$ the endoscopic transfer factors have no natural normalization and this makes it impossible to state the endoscopic character identities. Since those are intimately tied with the internal structure of $L$-packets, one also cannot hope to parameterize that structure. He suggests \cite[\S3]{Art06} that to resolve this problem, one can conjecture the existence of two sets of functions -- the spectral transfer factors $\Delta(\varphi,\pi)$ and the mediating functions $\rho(\Delta,\tilde s)$. These functions take away the problem arising from the lack of a natural normalization of the endoscopic transfer factors by incorporating all possible normalizations. Shelstad \cite{SheTE3} has been able to show that such functions indeed exist when the ground field is $\R$. The existence of these functions for $p$-adic fields has thus far remained unknown. This was a serious problem that prevented the establishment of the endoscopic classification of representations of non-quasi-split symplectic and orthogonal groups. We refer the reader to \cite[Ch. 9]{Art11} for a discussion.

Studying this problem from a different perspective, Vogan \cite{Vog93} points out that the object $(G',\xi)$ has too many automorphisms, and these automorphisms can permute $\Pi_\varphi^{G'}$ without being detected in any way by $^LG$. This behavior is not at all pathological and already occurs for groups as simple as $\tx{SL}_2(\R)$. This indicates that the datum $(G',\xi)$ is by itself not sufficient to specify an injection as in \eqref{eq:qb} and that one needs to further enrich it by additional data. Vogan then proposes one such enrichment, which consists of adding to $(G',\xi)$ an element $z \in Z^1(\Gamma,G)$ with the property that $\xi^{-1}\sigma(\xi) = \tx{Ad}(z(\sigma))$. The triples $(G',\xi,z)$ are called pure inner twists and their isomorphism classes are parameterized by the set $H^1(\Gamma,G)$, which according to Kottwitz's result is related to the Pontryagin dual of $\pi_0(Z(\hat G)^\Gamma)$. The work of Adams, Barbasch, Kottwitz, Vogan, and others suggests the following variant of \eqref{eq:qb}: There should exist a finite set $\Pi_\varphi^\tx{pure}$ of isomorphism classes of quadruples $(G',\xi,z,\pi')$, where $(G',\xi,z)$ is a pure inner twist and $\pi'$ is an irreducible tempered representation of $G'(F)$, together with a commutative diagram
\begin{equation} \label{eq:pb} \xymatrix{
\Pi_\varphi^\tx{pure}\ar[r]^-{\iota_\mf{w}}\ar[d]&\tx{Irr}(\pi_0(S_\varphi))\ar[d]\\
H^1(\Gamma,G)\ar[r]&\pi_0(Z(\hat G)^\Gamma)^*
} \end{equation}
in which the bottom arrow is Kottwitz's map, the right arrow sends an irreducible representation to its central character, the left arrow sends a quadruple $(G',\xi,z,\pi')$ to the class of $z$, and the top arrow (the only conjectural arrow in the diagram) is a bijection that identifies the trivial representation on the right to the quadruple $(G,\tx{id},1,\pi)$ on the left, where $\pi$ is the unique $\mf{w}$-generic constituent of $\Pi_\varphi^G$, and furthermore provides the correct virtual characters necessary for the endoscopic character identities. Note that in this formulation the top arrow is expected to be bijective both in the real and in the $p$-adic case, making this formulation more uniform than that of \eqref{eq:qb}. The diagram \eqref{eq:pb} was constructed by DeBacker and Reeder \cite{DR09} for any unramified $p$-adic group $G$ and a class of depth-zero supercuspidal Langlands parameters $\varphi$. It was then shown by the author \cite{Kal11} that this construction satisfies the expected endoscopic character identities. Just as in the quasi-split case, it was important for both the statement and the proof of the conjecture that given a pure inner twist $(\xi,z) : G \rw G'$, the data $\mf{w}$ and $z$ lead to a natural normalization $\Delta[\mf{w},z]$ of the endoscopic transfer factors. The relationship of this point of view with that of Arthur is straightforward: The spectral transfer factor is given by the expression $\Delta(\varphi,\iota_{\mf{w}}(\rho))=\tx{tr}(\rho(s))$, where $s \in S_\varphi$ is part of the endoscopic datum to which $\Delta[\mf{w},z]$ is associated, and the mediating function $\rho(\Delta,\tilde s)$ is specified by $\rho(\Delta[\mf{w},z],\tilde s)=1$.

While pure inner twists seem to very elegantly resolve the problem, they have an essential drawback: The map $H^1(\Gamma,G) \rw H^1(\Gamma,G/Z(G))$ is usually not surjective, which means that not every inner twist $\xi : G \rw G'$ can be equipped with an element $z$. The worst case is when $F$ is $p$-adic and $G$ is simply\-connected. A theorem of Kneser \cite{Kne65} states that then $H^1(\Gamma,G)=\{1\}$ and the diagram \eqref{eq:pb} collapses to the quasi-split case \eqref{eq:qb}. Over the real numbers, the notion of a strong real form  was introduced in the work of Adams-Barbasch-Vogan, which resolves this problem and builds one of the foundations of the treatment of the local Langlands correspondence for real groups in \cite{ABV92}. Over $p$-adic fields however, the problem of finding a suitable analog of strong real forms has thus far remained open \cite[Problem 9.3]{Vog93}.

One attempt to alleviate this problem can be made using Kottwitz's theory of isocrystals with additional structure \cite{Kot85}, \cite{Kot97}. In this theory a different cohomology set for $G$ is studied -- if we let $L$ be the completion of the maximal unramified extension of $F$ and $\ol{L}$ an algebraic closure of $L$, Kottwitz studies the set $H^1(W_F,G(\ol{L}))$ and shows that the elements of a certain subset $B(G)_\tx{bas}$ of this cohomology set give inner forms of $G$ and have an interpretation in terms of $\hat G$ similar to that of $H^1(\Gamma,G)$. Using this theory, Kottwitz has suggested  a diagram similar to \eqref{eq:pb}, and a precise formulation of the resulting conjecture is presented in \cite[\S2.4]{Kal14}. It was then shown by the author \cite{Kal14} that the work \cite{DR09} of DeBacker and Reeder extends to this setting and that the endoscopic character identities hold. Furthermore, the same is true \cite{KalEpi} for a different class of supercuspidal Langlands parameters, for which the corresponding $L$-packets consist of epipelagic representations \cite{RY14}. The cohomology set $B(G)_\tx{bas}$ has a map to $H^1(\Gamma,G/Z(G))$ which is surjective when the center of $G$ is connected. In this case, the set $B(G)_\tx{bas}$ resolves the problem completely. The opposite case is that of a simply connected group $G$, where again one has $B(G)_\tx{bas}=\{1\}$. One also encounters other problems when using $B(G)_\tx{bas}$ to study local $L$-packets. For one, the relationship between $B(G)_\tx{bas}$ and the strong real forms of \cite{ABV92} is unclear. Furthermore, it is not clear how to relate local $L$-packets and the stable Arthur-Selberg trace formula when $G$ fails the Hasse principle, because this would entail passing from $G$ to its simply connected cover $G_\tx{sc}$, a step that is problematic due to $B(G_\tx{sc})_\tx{bas}=\{1\}$.

In the present paper we introduce a new cohomology set for affine algebraic groups by replacing the cohomology of the Galois group with the cohomology of a certain Galois gerb \cite{LR87} that is canonically associated to any local field of characteristic zero. This new cohomology set resolves the problems described above pertaining to the statement of the local Langlands correspondence and endoscopic character identities for $p$-adic groups. It provides a solution to \cite[Problem 9.3]{Vog93}. It also provides explicit formulas for Arthur's conjectural spectral transfer factors and mediating functions. In fact, since our goal is not just to give a good description of $L$-packets for $p$-adic groups, but to also make sure that this description interfaces well with the stabilized Arthur-Selberg trace formula so as to be suitable for global applications, we provide a construction that works uniformly for real and $p$-adic groups. The interplay between this construction and the trace formula is studied in \cite{KalGRI}, where it is shown that the local normalizations establish here fit perfectly with the stabilized trace formula. The new cohomology set is associated to any affine algebraic group $G$ and any finite central subgroup $Z \subset G$, both defined over a local field $F$ of characteristic zero, and is denoted by $H^1(u \rw W,Z \rw G)$. For applications to the local Langlands correspondence, $G$ will be connected and reductive, and $Z$ will be a finite central subgroup, which can often times be taken to be the center of the derived subgroup of $G$. The cohomology set $H^1(u \rw W,Z \rw G)$ has the following properties: There exists an injective map $H^1(\Gamma,G) \rw H^1(u \rw W,Z \rw G)$ and a surjective map $H^1(u \rw W,Z \rw G) \rw H^1(\Gamma,G/Z)$. Both of these maps are functorial in $G$. We show that, when $G$ is connected and reductive and $Z$ contains the center of the derived subgroup of $G$, the induced map $H^1(u \rw W,Z \rw G) \rw H^1(\Gamma,G/Z(G))$ is surjective, and thus every inner twist $\xi : G \rw G'$ can be equipped with an element of $H^1(u \rw W,Z \rw G)$. The set $H^1(u \rw W,Z \rw G)$ is efficient in the following sense: It is always finite, and when $G$ is split and $F$ is $p$-adic, the map $H^1(u \rw W,Z(G_\tx{der}) \rw G) \rw H^1(\Gamma,G/Z(G))$ is bijective, which means that for every inner twist $\xi : G \rw G'$ there is a unique element of $H^1(u \rw W,Z(G_\tx{der}) \rw G)$ belonging to that twist. A similar efficiency holds over the real numbers, but is slightly more complicated to state. For a general connected reductive $G$ and finite central $Z$, the set $H^1(u \rw W,Z \rw G)$ admits a functorial map to a certain finite abelian group which is constructed from $^LG$. This analog of the result of Kottwitz discussed above allows us to construct a normalization of the endoscopic transfer factors from an element of $H^1(u \rw W,Z \rw G)$  and this in turn allows to state a version of \eqref{eq:pb} and of the endoscopic character identities for all inner forms of a given quasi-split connected reductive group $G$, thus for any connected reductive group.

To elaborate on the last sentence, let $G$ be a connected reductive group defined and quasi-split over $F$, and $Z \subset G$ a finite central subgroup defined over $F$. Set $\bar G=G/Z$. The isogeny $G \rw \bar G$ dualizes to an isogeny $\hat{\bar G} \rw \hat G$ of the complex Langlands dual groups. We let $Z(\hat{\bar G})^+$ be the preimage under this isogeny of the diagonalizable group $Z(\hat G)^\Gamma$. Then we show that the set $H^1(u \rw W,Z \rw G)$ is equipped with a functorial map to the Pontryagin dual of $\pi_0(Z(\hat{\bar G})^+)$. This map is a bijection when $F$ is $p$-adic or when $G$ is a torus, and is compatible with Kottwitz's map for any $F$ and $G$. We define a rigid inner twist of $G$ to be an inner twist $\xi : G \rw G'$ equipped with an element $z \in Z^1(u \rw W,Z \rw G)$ (for some $Z$) that lifts the element $\xi^{-1}\sigma(\xi) \in Z^1(\Gamma,G/Z(G))$.  Given a tempered Langlands parameter $\varphi$, we let $S_\varphi^+$ denote the preimage in $\hat{\bar G}$ of $S_\varphi$. When $F=\R$, the finite group $\pi_0(S_\varphi^+)$ is always abelian, but may fail to be a 2-group. When $F$ is $p$-adic, $\pi_0(S_\varphi^+)$ can be non-abelian -- this already happens for $\tx{SL}_2$ and we discuss an example in Section \ref{sec:lpack}. We expect to have a finite set $\Pi_\varphi$ of isomorphism classes of quadruples $(G',\xi,z,\pi')$ and a commutative diagram
\begin{equation} \label{eq:rb} \xymatrix{
\Pi_\varphi\ar[d]\ar[rr]^-{\iota_\mf{w}}&&\tx{Irr}(\pi_0(S_\varphi^+))\ar[d]\\
H^1(u \rw W,Z \rw G)\ar[rr]&&\pi_0(Z(\hat{\bar G})^+)^*
} \end{equation}
with the same properties as \eqref{eq:pb}. In fact, each term in diagram \eqref{eq:pb} is a subset of the corresponding term here, and we expect that this diagram is an enlargement of \eqref{eq:pb} in the obvious sense. Furthermore, we show that the data $\mf{w}$ and $z$ lead to a normalization $\Delta[\mf{w},z]$ for the endoscopic transfer factor and this allows us to state the conjectural endoscopic character identities. In order to state these identities, we must work with a slight refinement of the notion of endoscopic datum. This refinement resolves another problem observed by Arthur in \cite{Art06} which pertains to the invariance of the transfer factor under automorphisms of endoscopic data. We refer the reader to Section \ref{sec:lpack} for more details. The relationship between our statement of the local Langlands conjecture and the endoscopic character identities and that of Arthur is again straightforward. When $G$ is simply connected, the conjectural spectral transfer factor of Arthur is given by the expression $\Delta(\varphi,\iota_{\mf{w}}(\rho))=\tx{tr}(\rho(\tilde s))$, where $\tilde s$ is part of the refined endoscopic datum to which $\Delta[\mf{w},z]$ is associated, and the mediating function $\rho(\Delta,\tilde s)$ is specified by $\rho(\Delta[\mf{w},z],\tilde s)=1$. When $G$ is not simply connected, the situation is almost as simple but requires a bit more notation. We refer the reader to \cite[\S4.6]{KalGRI}. For any $G$, our results provide a construction of Arthur's spectral transfer factors and mediating functions. As a result, these objects are now known for $p$-adic groups.

Given the mature state of the local Langlands correspondence and endoscopy for real groups, two natural questions arise: Given a real reductive group $G$, how does our set $H^1(u \rw W,Z \rw G)$ relate to the set of strong real forms of $G$ constructed in \cite{ABV92}, and how does our Diagram \eqref{eq:rb} and our statement of endoscopic character identities relate to the work of Langlands and Shelstad. In this paper we answer both questions completely. With regards to the first question, we were pleasantly surprised to find out that, while the notion of strong real forms and the notion of rigid inner twists (for real groups) are defined in very different ways, they are in fact equivalent. By this we mean that the category of strong real forms of a given real group is equivalent to the category of rigid inner twists. We alert the reader that \cite{Vog93} introduces the similar sounding notion of ``rigid rational form'', of which ``strong real form'' is a special case. Despite the similarity in names, our ``rigid inner twists'' in the case of real groups are equivalent to ``strong real forms'', and not to ``rigid rational forms''. Regarding the second question, it turns out that the constructions and arguments of Langlands and Shelstad can be put into our framework without much effort, after which their work implies the existence of Diagram \eqref{eq:rb} for any tempered Langlands parameter, as well as the validity of our statement of endoscopic character identities. A further natural question would be to compare the construction of Diagram \eqref{eq:rb} given in this paper, which is based on the cohomology sets $H^1(u \rw W,Z \rw G)$, with the analogous construction in \cite{ABV92}, which is based on the geometry of the dual group. We leave this more subtle question for a separate paper.

There is by now substantial evidence that the formulation of the refined local Langlands correspondence presented in this paper is the correct one. First, it is uniform for real and $p$-adic groups and is true for real groups. Second, we show in \cite{KalGRI} that this formulation fits seamlessly into the spectral side of the stabilized Arthur-Selberg trace formula. In particular, the canonical adelic transfer factor admits a decomposition as the product of the normalized local transfer factors introduced here, and the normalized local bijections of Diagram \eqref{eq:rb} fit together to a canonical pairing between the adelic $L$-packet and its global $S$-group, which in turn leads to the multiplicity formula for discrete automorphic representations of arbitrary connected reductive groups. With these facts at hand, a proof of our formulation of the local Langlands conjecture for classical groups is well within reach by the methods of \cite{Art11}. Third, it is shown in \cite[\S4.6]{KalGRI} that our formulation implies, and is in fact equivalent to, a stronger version of Arthur's local conjecture \cite[\S3]{Art06}. The strengthening comes from the fact that the results of this paper give explicit formulas for Arthur's conjectural spectral transfer factors and mediating functions. Fourth, it is shown in \cite[\S4,\S6]{KalRIBG} that for groups for which the formulation of the local conjecture based on $B(G)_\tx{bas}$, as presented in \cite[\S2.4]{Kal14}, is available, its validity is equivalent to the validity of the formulation presented here. In particular, this implies that the work of \cite{Kal14} and \cite{KalEpi} provides a proof of the validity of our formulation in the special case of depth-zero and epipelagic supercuspidal representations. And fifth, it is shown in \cite[\S5,\S6]{KalRIBG} that the validity of our formulation for all connected reductive groups is equivalent to the validity of the formulation of \cite[\S2.4]{Kal14} for all connected reductive groups with connected center. The latter formulation, besides having been proved in special cases in \cite{Kal14} and \cite{KalEpi}, is supported by a conjecture of Kottwitz that describes the cohomology of Rapoport-Zink spaces. Fargues has recently announced a conjectural program that aims at a resolution of Kottwitz's conjecture and at a proof of the formulation of \cite[\S2.4]{Kal14} for all connected reductive $p$-adic groups with connected center.

We will now describe the contents of this paper and sketch the construction of the set $H^1(u \rw W,Z \rw G)$ and its map to the dual of $\pi_0(Z(\hat{\bar G})^+)$. The construction and study of $H^1(u \rw W,Z \rw G)$ is the main topic of Section \ref{sec:defcoh}. It is based on Kottwitz's notion of \emph{algebraic 1-cocycles} introduced in \cite[\S8]{Kot97}. If $W$ is a topological group, which is an extension of $\Gamma$ by an algebraic group $X$ (a Galois gerb in the terminology of \cite{LR87}), then Kottwitz defines an algebraic 1-cocycle of $W$ into the $\ol{F}$-points of an algebraic group $Y$ to be a continuous 1-cocycle $W \rw Y(\ol{F})$ whose restriction to $X(\ol{F})$ is given by an algebraic homomorphism $X \rw Y$. In Section \ref{sec:u} we construct a certain pro-finite multiplicative algebraic group $u$ as a limit of certain finite multiplicative algebraic groups $u_{E/F,n}$. We then show that $H^1(\Gamma,u)$ vanishes and $H^2(\Gamma,u)$ is topologically cyclic. This implies that there is an (up to isomorphism) canonical Galois gerb $W$ bound by $u$ and the only automorphisms it has come from conjugation by elements of $u$. In other words, it is as canonical as the relative Weil group of a finite Galois extension of $F$. In Section \ref{sec:defcohsub} we then define $H^1(u \rw W,Z \rw G)$ to be the set of cohomology classes of those algebraic 1-cocycles of $W$ valued in $G$ whose restriction to $u$ has image contained in $Z$. An important feature of the group $u$ is that for any finite (necessarily multiplicative) algebraic group $Z$ defined over $F$ there is a natural surjection $\tx{Hom}(u,Z)^\Gamma \rw H^2(\Gamma,Z)$. This eventually leads to the surjectivity of $H^1(u \rw W,Z \rw G) \rw H^1(\Gamma,G/Z)$. The latter, together with the injective map $H^1(\Gamma,G) \rw H^1(u \rw W,Z \rw G)$, the finiteness of $H^1(u \rw W,Z \rw G)$, its functoriality in $Z \rw G$, an inflation-restriction sequence, as well as further properties, are discussed in Section \ref{sec:basic}. These properties make the set $H^1(u \rw W,Z \rw G)$ easily computable by reducing the computation to that of classical Galois cohomology groups. In Section \ref{sec:h1sc} we construct a quotient of $H^1(u \rw W,Z \rw G)$ by a certain equivalence relation. This quotient is called $H^1_\tx{sc}(u \rw W,Z \rw G)$ and is analogous to the first Galois-cohomology set of the crossed module $G_\tx{sc} \rw G$. When the ground field $F$ is $p$-adic, the equivalence relation is trivial and we obtain nothing new, but over the real numbers $H^1_\tx{sc}(u \rw W,Z \rw G)$ is usually a proper quotient of $H^1(u \rw W,Z \rw G)$.

Section \ref{sec:tn+iso} is concerned with the construction of the functorial map from $H^1(u \rw W,Z \rw G)$ to the Pontryagin dual of $\pi_0(Z(\hat{\bar G})^+)$ in the case where $G$ is a connected reductive group defined over $F$. This map is among the most important properties of the cohomology set $H^1(u \rw W,Z \rw G)$ and is crucial for its application to the local Langlands conjectures. Instead of using the language of the dual group, in this section we construct a finite abelian group $\bar Y_{+,\tx{tor}}(Z \rw G)$ from the root datum of $G$ which will later turn out to admit a functorial map to $\pi_0(Z(\hat{\bar G})^+)^*$ that is bijective when $F$ is $p$-adic or when $G$ is a torus and injective in general. The abelian group $\bar Y_{+,\tx{tor}}(Z \rw G)$ is constructed in Section \ref{sec:yfunc}. It is functorial in $Z \rw G$. In Section \ref{sec:uni} we show that there can exist at most one functorial isomorphism $\bar Y_{+,\tx{tor}}(Z \rw G) \rw H^1_\tx{sc}(u \rw W,Z \rw G)$ subject to two conditions, one of them being that it coincides with the Tate-Nakayama isomorphism when $Z=\{1\}$ and $G$ is a torus. The task then becomes to construct this isomorphism. For this, we introduce in Section \ref{sec:ucp} a device similar to the cup-product between Tate-cochains of positive and negative degrees, which we call an unbalanced cup-product. In section \ref{sec:arith} we review some arithmetic material from \cite{Lan83} concerning specific representatives of fundamental classes of finite Galois extensions. With these preparations in place we can give an explicit realization of the Galois gerb $W$ in Section \ref{sec:expw} and then use it to construct the isomorphism $\bar Y_{+,\tx{tor}}(Z \rw G) \rw H^1_\tx{sc}(u \rw W,Z \rw G)$, first when $G$ is a torus in Section \ref{sec:tn+s}, and then when $G$ is a connected reductive group in Section \ref{sec:tn+g}.

Section \ref{sec:llc} describes how the set $H^1(u \rw W,Z \rw G)$ can be applied to the study of the local Langlands correspondence and endoscopy. We introduce the concept of rigid inner twists in Section \ref{sec:rigid} and show how for a given maximal torus $S \subset G$ in a connected reductive group, the set $H^1(u \rw W,Z \rw S)$ parameterizes the rational classes inside the stable class of any given strongly-regular element of $S(F)$. This leads to a cohomological invariant $\tx{inv}(\delta,\delta')$ that will allow us to normalize the transfer factors later. This discussion follows the ideas already used in \cite[\S2.1]{Kal11}, but now adapted to the set $H^1(u \rw W,Z \rw G)$. In Section \ref{sec:strong} we establish the equivalence between the notion of rigid inner twists of a given real reductive group and that of strong real forms of it. In Section \ref{sec:refined} we establish the functorial injective map from the abelian group $\bar Y_{+,\tx{tor}}(Z \rw G)$ defined in Section \ref{sec:yfunc} to the Pontryagin dual of $\pi_0(Z(\hat{\bar G})^+)$. This map can be phrased (Corollary \ref{cor:tn+pair}) as a pairing between $H^1(u \rw W,Z \rw G)$ and $\pi_0(Z(\hat{\bar G})^+)$, and, besides $\tx{inv}(\delta,\delta')$, this pairing is the second ingredient in the normalization of the transfer factor. We then proceed to define the notions of a refined endoscopic datum and of an isomorphism between such data. This notion is the third ingredient in the normalization of the transfer factor, which we then are able to establish. We also show that this normalization is invariant under all automorphisms of the refined endoscopic datum, thereby resolving the issue noted by Arthur \cite[p. 208]{Art06} that an absolute transfer factor for a non-quasi-split group need not be invariant under all automorphisms of a (usual) endoscopic datum. In Section \ref{sec:lpack} we spell out the conjectural diagram \eqref{eq:rb} and the statements of the conjectural endoscopic character identities. In Section \ref{sec:real} we turn to the setting of real groups and show that the work of Langlands and Shelstad implies that the conjectures stated in Section \ref{sec:lpack} hold for real groups.

The author is grateful to Robert Kottwitz for introducing him to this problem and for sharing his intuition that algebraic cocycles of Galois gerbs could hold the key to its resolution. He further thanks Jeffrey Adams, Stephen DeBacker, Diana Shelstad, and Olivier Ta\"ibi, for their helpful comments, suggestions, and corrections. The support of the National Science Foundation via grant DMS-161489 is gratefully acknowledged.

\tableofcontents

\section{Some notation}

Let $F$ be a local field of characteristic zero. Fix an algebraic closure $\ol{F}$ of $F$ and write $\Gamma$ or $\Gamma_F$ for the Galois group of $\ol{F}/F$, and $W_F$ for the Weil group of $\ol{F}/F$. Finite extensions of $F$ will be taken to be subfields of $\ol{F}$. For a fixed finite Galois extension $E/F$ we will write $\Gamma_{E/F}$ and $W_{E/F}$ for the relative Galois and Weil groups and $N_{E/F}$ for the norm endomorphism of any $\Gamma_{E/F}$-module. We will reserve the letter $W$ for a different purpose. Given $\sigma \in \Gamma$ and $x \in \ol{F}$, we will denote the image of $x$ under $\sigma$ by $\sigma x$.

We will use the symbol $\N^\times$ to denote the set of positive integers with the partial order given by divisibility. By a co-final sequence in $\N^\times$ we mean a totally ordered subset $\{n_k\} \subset \N^\times$ so that every element of $\N^\times$ is dominated by some $n_k$.

If $D$ is a diagonalizable group, we will write $X^*(D)$ and $X_*(D)$ for its character and cocharacter modules. These will be written additively, and the canonical pairing between $\chi \in X^*(D)$ and $\lambda \in X_*(D)$ will be denoted by $\<\chi,\lambda\>$. For $x \in \ol{F}$, we will sometimes write $x^\lambda$ instead of $\lambda(x)$ for the image of $x$ under the map $\lambda : \mb{G}_m \rw D$.

If $G$ is a connected reductive group, we will write $G_\tx{der}$ for its derived subgroup, and $G_\tx{sc}$ and $G_\tx{ad}$ for the simply-connected cover and the adjoint quotient of $G_\tx{der}$. If $S \subset G$ is a maximal torus, we will write $S_\tx{der}$, $S_\tx{sc}$ and $S_\tx{ad}$ for the corresponding maximal tori of $G_\tx{der}$, $G_\tx{sc}$ and $G_\tx{ad}$ respectively. The subset of strongly regular semi-simple elements of $G$ will be denoted by $G_\tx{sr}$. When $G$ is defined over $F$, we will write $G(R)$ for the set of points of $G$ with values in an $F$-algebra $R$. The notation $g \in G$ will be shorthand for $g \in G(\ol{F})$. The action of $g$ on $G$ by conjugation will be denoted by $\tx{Ad}(g)$. Two elements $g_1,g_2 \in G_\tx{sr}(F)$ are called stably conjugate if they are conjugate under $G(\ol{F})$. For all elements $h \in G(\ol{F})$ with $\tx{Ad}(h)g_1=g_2$, conjugation by $h$ provides the same isomorphism from the centralizer of $g_1$ to the centralizer of $g_2$, and we will call this isomorphism $\phi_{g_1,g_2}$. In the slightly more general setting where an isomorphism $\xi : G \rw G'$ has been fixed and $g_1 \in G_\tx{sr}$, $g_2 \in G'_\tx{sr}$ are such that there exists $h \in G$ with $g_2=\xi(hg_1h^{-1})$, we will write $\phi_{g_1,g_2}$ for the isomorphism $\xi\circ\tx{Ad}(h)$.

We will write $\hat G$ for the (connected) complex Langlands dual group of $G$ and $^LG=\hat G \rtimes W_F$ for the (Weil-form) of its $L$-group.

Given a finite group $\Delta$ and a $\Delta$-module $M$, we will use besides the usual group cohomology $H^i(\Delta,M)$ also the modified, or Tate-cohomology, which we will denote by $H^i_\tx{Tate}(\Delta,M)$.

\section{The cohomology set $H^1(u \rw W,Z \rw G)$} \label{sec:defcoh}

\subsection{The multiplicative pro-algebraic group $u$} \label{sec:u}

For a finite Galois extension $E/F$ and a natural number $n$, we consider the algebraic group $R_{E/F}[n] := \tx{Res}_{E/F}\mu_n$. This is a multiplicative group with $X^*(R_{E/F}[n])=\Z/n\Z[\Gamma_{E/F}]$ with $\Gamma$ acting by multiplication on the left, and for any Galois extension $K/F$ we have $R_{E/F}[n](K)=\tx{Maps}(\Gamma_{E/F},\mu_n(\ol{F}))^{\Gamma_K}$, where $(\sigma f)(\tau) = \sigma(f(\sigma^{-1}\tau))$ for $f : \Gamma_{E/F} \rw \mu_n(\ol{F})$, $\sigma \in\Gamma_K$, and $\tau \in \Gamma_{E/F}$.
We have the diagonal embedding $\mu_n \rw R_{E/F}[n]$ which sends each $x \in \mu_n$ to the constant map with value $x$. We define the multiplicative group $u_{E/F,n}$ to be the cokernel of this embedding, so that we have the exact sequence
\begin{equation} \label{eq:defu} 1 \rw \mu_n \rw R_{E/F}[n] \rw u_{E/F,n} \rw 1. \end{equation}
If $K/F$ is a Galois extension containing $E$ and $m$ is a multiple of $n$, we define the map
\begin{equation} \label{eq:defp} p : R_{K/F}[m] \rw R_{E/F}[n],\qquad (pf)(a) = \prod_{\substack{b \in \Gamma_{K/F}\\ b \mapsto a}} f(b)^\frac{m}{n}. \end{equation}
This map is an epimorphism and descends to an epimorphism $u_{K/F,m} \rw u_{E/F,n}$. We define the pro-algebraic multiplicative group $u$ as the limit
\[ u := \varprojlim u_{E/F,n} \]
taken over the index category $\mc{I}_0$ whose objects are tuples $(E/F,n)$ and where there is at most one morphism $(K/F,m) \rw (E/F,n)$ and it exists if and only if $E \subset K$ and $n|m$. For every $(E/F,n)$ the canonical map $u \rw u_{E/F,n}$ is an epimorphism.

Given a finite algebraic group $Z$ defined over $F$, necessarily multiplicative due to assumption that $F$ has characteristic zero, we will write $\tx{Hom}(u,Z)$ for the set of algebraic homomorphisms. Such a homomorphism is given by the composition of an algebraic homomorphism $u_{E/F,n} \rw Z$ for some suitable $(E/F,n) \in \mc{I}_0$ with the natural projection $u \rw u_{E/F,n}$. It is straightforward to check that we have the isomorphism
\begin{equation} \label{eq:homgamma} \tx{Hom}(u_{E/F,n},Z)^\Gamma \rw \tx{Hom}(\mu_n,Z)^{N_{E/F}},\qquad f \mapsto f \circ \delta_e, \end{equation}
where the superscript $N_{E/F}$ denotes the kernel of the norm map and $\delta_e : \mu_n \rw u_{E/F,n}$ is the homomorphism dual to the ``evaluation at $e$ map'' $\tx{ev}_e : X^*(u_{E/F,n}) = \Z/n\Z[\Gamma_{E/F}]_0 \rw \Z/n\Z$, (thus $\delta_e(x)$ is the map $\Gamma_{E/F} \rw \mu_n$ supported at $e$ and having the value $x$ there).

The group of $\ol{F}$-points of the pro-algebraic group $u$ carries a natural pro-finite topology and a continuous action of $\Gamma$. The continuous cohomology groups $H^i(\Gamma,u)$ are therefore defined.

\begin{thm} \label{thm:cohcomp} We have $H^1(\Gamma,u)=0$ and $H^2(\Gamma,u) = \begin{cases} \hat \Z&,F\tx{ is non-arch.}\\ \Z/2\Z&,F=\R\end{cases}$. \end{thm}

\begin{proof}

We begin by noting that the limit defining $u$ may be taken over any co-final subcategory of $\mc{I}$. We fix such a subcategory $\{(E_k,n_k)\}$ by taking a tower $F=E_0 \subset E_1 \subset E_2 \subset \dots$ of finite Galois extensions of $F$ with the property $\bigcup E_k=\ol{F}$ and a co-final sequence $\{n_k\} \subset \N^\times$.
According to \cite[Corollary 2.7.6]{NSW08}, we have an isomorphism $H^i(\Gamma,u) \rw \varprojlim H^i(\Gamma,u_k)$,
where the limits are taken over the above co-final subcategory and we have abbreviated $u_{E_k/F,n_k}$ by $u_k$. We must compute $\varprojlim H^i(\Gamma,u_k)$ for $i=1,2$.

We begin with $i=2$ and use the functorial isomorphism
\[ H^2(\Gamma,u_k) \cong H^0(\Gamma,X^*(u_k))^* \cong \left[\frac{n_k}{(n_k,[E_k:F])}\Z/n_k\Z\right]^* \cong \Z/(n_k,[E_k:F])\Z, \]
where $*$ denotes the group of $\Q/\Z$-valued characters. Here the first isomorphism is given by Poitou-Tate duality \cite{Tat63}. Note that in the archimedean case one in general needs to use the quotient $H^0_\tx{Tate}$ of $H^0$, but for $X^*(u_k)$ one sees that the two groups coincide. For $k>l$, the transition map $H^2(p) : H^2(\Gamma,u_k) \rw H^2(\Gamma,u_l)$ is translated by this isomorphism to the natural projection map $\Z/(n_k,[E_k:F])\Z \rw \Z/(n_l,[E_l:F])\Z$. If $F=\R$, then for $k>>0$ we have $(n_k,[E_k:F])=2$. If $F$ is non-archimedean, we can clarify the situation by setting $n_k=[E_k:F]$.
Then $(n_k,[E_k:F])=n_k$. This completes the computation in the case $i=2$.

Now we turn to $i=1$. Our goal is to show that for any $l$ there is $k>l$ such that the transition map $H^1(p) : H^1(\Gamma,u_k) \rw H^1(\Gamma,u_l)$ is the zero map. This would follow if we could find $k>j>l$ so that the transition maps $H^1(p) : H^1(\Gamma,R_j) \rw H^1(\Gamma,R_l)$ and $H^2(p) : H^2(\Gamma,\mu_{n_k}) \rw H^2(\Gamma,\mu_{n_j})$ are zero, by chasing through the diagram arising from the exact sequence $H^1(\Gamma,R_k) \rw H^1(\Gamma,u_k) \rw H^2(\Gamma,\mu_{n_k})$. To treat the case of $H^1(\Gamma,R_j)$, we apply Shapiro's lemma to obtain an isomorphism $H^1(\Gamma,R_j) \rw H^1(\Gamma_{E_j},\mu_{n_j}) \rw E_j^\times/E_j^{\times,n_j}$. Under this isomorphism the transition map is translated to the norm map $N_{E_j/E_l} :  E_j^\times/E_j^{\times,n_j} \rw E_l^\times/E_l^{\times,n_l}$ and we select $j$ so that $N_{E_j/E_l}(E_j^\times) \subset E_l^{\times,n_l}$. For the case of $H^2(\Gamma,\mu_{n_k})$ we note that the restriction of the map $p$ to the diagonally embedded copy of $\mu_{n_k}$ into $R_k$ is the $(\frac{n_k}{n_j})^2$-power map. We then choose $k$ so that $n_k/n_j$ is a multiple of $n_j$. This completes the proof.
\end{proof}

From now on we will denote by $\xi \in H^2(\Gamma,u)$ the element corresponding to $-1 \in \Z$ resp. $-1 \in \Z/2\Z$. The reason we use $-1$ instead of $1$ is that we want the isomorphism of Section \ref{sec:tn+iso} to be compatible with the classical Tate-Nakayama isomorphism, rather than its negative.

For any algebraic group $Z$ defined over $F$, we obtain a map
\begin{equation} \label{eq:h2map} \xi^* : \tx{Hom}(u,Z)^\Gamma \rw H^2(\Gamma,Z)\qquad \phi \mapsto \phi(\xi). \end{equation}
\begin{pro} \label{pro:h2map} If $Z$ is a finite algebraic group defined over $F$, then $\xi^*$ is surjective. If $Z$ is also split, then $\xi^*$ is also injective.
\end{pro}
\begin{proof}
\vspace{-10pt}We appeal again to the perfect duality of Poitou-Tate, which can be written uniformly in the archimedean and non-archimedean cases as
\[ H^2(\Gamma,Z) \otimes (\varprojlim X^*(Z)^\Gamma/N_{E/F} X^*(Z)) \rw \Q/\Z,\qquad (z,\chi) \mapsto \tx{inv}_F(z \cup \chi). \]
Under this duality, the map dual to $\xi^*$ takes the form
\begin{eqnarray*}
\varprojlim X^*(Z)^\Gamma/N_{E/F} X^*(Z)&\rw&\tx{Hom}(X^*(Z),X^*(u))^{\Gamma,*}\\
\tx{by }\eqref{eq:homgamma}&\cong&(\varinjlim \tx{Hom}(X^*(Z),\Z/n\Z)^{N_{E/F}})^*\\
&\cong&\tx{Hom}(\varprojlim X^*(Z)/N_{E/F} X^*(Z),\Q/\Z)^*\\
&\cong&\varprojlim X^*(Z)/N_{E/F} X^*(Z).
\end{eqnarray*}
The somewhat unorthodox manipulation of limits is justified by the finiteness of the appropriate arguments of $\tx{Hom}$. Tracing through the identifications, one sees that the composite map works out to be the obvious inclusion. In addition, when $X^*(Z)$ carries a trivial $\Gamma$-action, this map is an isomorphism.
\end{proof}

\subsection{Definition of $H^1(u \rw W,Z \rw G)$} \label{sec:defcohsub}

According to \cite[Theorem 2.7.7]{NSW08}, the class $\xi \in H^2(\Gamma,u)$ corresponds to an isomorphism class of extensions of profinite groups
\[ 1 \rw u \rw W \rw \Gamma \rw 1. \]
Furthermore, $H^1(\Gamma,u)=0$ implies that the only automorphisms of an extension belonging to this isomorphism class, which induce the identity on both $\Gamma$ and $u$, are given by inner automorphisms by elements of $u$. We now fix one such extension.

Let $\mc{A}$ be the category of monomorphisms $Z \rw G$ defined over $F$, where $G$ is an affine
algebraic group, $Z$ is a finite group, and the morphism embeds $Z$ into the center of $G$. Later we will also be interested in the subcategories $\mc{T} \subset \mc{R} \subset \mc{A}$, where $[Z \rw G] \in \mc{A}$ belongs to $\mc{T}$ if $G$ is a torus, and belongs to $\mc{R}$ if $G$ is a connected reductive group. For two such objects $Z_1 \rw G_1$ and $Z_2 \rw G_2$ we define the set of morphisms $\mc{A}(Z_1 \rw G_1, Z_2 \rw G_2)$ to be the set of commutative diagrams
\[ \xymatrix{
Z_1\ar[d]\ar[r]&Z_2\ar[d]\\
G_1\ar[r]&G_2
} \]
where the horizontal maps are morphisms of algebraic groups defined over $F$. Since $F$ has characteristic zero, such a diagram is determined by its bottom horizontal arrow.

Given $[Z \rw G] \in \mc{A}$, the set $G(\ol{F})$ taken with the discrete topology carries a continuous $\Gamma$-action, which we inflate to a continuous $W$-action. We define $Z^1(u \rw W,Z \rw G)$ to be the set of those continuous cocycles of $W$ in $G(\ol{F})$ whose restriction to $u$ is an algebraic homomorphism $u \rw Z$. Clearly, this definition is functorial in $[Z \rw G]$. Define further $\bar Z^1(u \rw W,Z \rw G)=Z^1(u \rw W,Z \rw G)/B^1(W,Z)$ and $H^1(u \rw W,Z \rw G)=Z^1(u \rw W,Z \rw G)/B^1(W,G)$.

If $1 \rw u \rw W' \rw \Gamma \rw 1$ is an isomorphic extension, then there are canonical isomorphisms of functors
\[ \bar Z^1(u \rw W) \rw \bar Z^1(u \rw W')\qquad\tx{and}\qquad H^1(u \rw W) \rw H^1(u \rw W'). \]
On the other hand, while there is also an isomorphism $Z^1(u \rw W) \rw Z^1(u \rw W')$, it is not canonical.

\subsection{Basic properties of $H^1(u \rw W,Z \rw G)$} \label{sec:basic}

We continue with a fixed extension $W$ of $\Gamma$ by $u$ belonging to the canonical isomorphism class determined by $\xi \in H^2(\Gamma,u)$. Let $[Z \rw G] \in \mc{A}$. A simple remark of fundamental importance for us is that any $z \in Z^1(u \rw W,Z \rw G)$ gives rise to an inner form $G^z$ of $G$. Namely, the image of $z$ in $Z^1(W,G_\tx{ad})$ belongs to $Z^1(\Gamma,G_\tx{ad})$ and we can use it to twist the $\Gamma$-action on $G(\ol{F})$. This will be the prime topic of discussion in Section \ref{sec:rigid}, but it will be also useful now, as we shall see momentarily.

The inflation-restriction sequence associated to the homomorphism $W \rw \Gamma$ specializes to the exact sequence
\begin{equation} \label{eq:infres} 1 \rw H^1(\Gamma,G) \rw H^1(u \rw W,Z \rw G) \rw \tx{Hom}(u,Z)^\Gamma \rw H^2(\Gamma,G), \end{equation}
where the last term is to be ignored if $G$ is not abelian.

\begin{lem} \label{lem:tg} If $G$ is abelian, then the map $\tx{Hom}(u,Z)^\Gamma \rw H^2(\Gamma,G)$ in \eqref{eq:infres} can be taken to be the composition of \eqref{eq:h2map} with the natural map $H^2(\Gamma,Z) \rw H^2(\Gamma,G)$.
\end{lem}
\begin{proof}
The map in question is usually taken to be the transgression map. Recall from \cite[Proposition 1.6.6]{NSW08} that the image of $\phi \in \tx{Hom}(u,Z)^\Gamma$ under the transgression map can be represented by choosing
a continuous section $s : \Gamma \rw W$ and taking the differential of the 1-cochain
\[ c : W \rw G,\qquad c(w) = \phi(w^{-1}s(w)). \]
By definition, $\dot\xi(\sigma,\tau) = s(\sigma)s(\tau)s(\sigma\tau)^{-1}$ represents the class $\xi$ and one computes that
$ dc(\sigma,\tau) = \phi(\dot\xi(\sigma,\tau))^{-1}$. Of course we may replace the transgression map by its negative and still keep the sequence \eqref{eq:infres} exact.
\end{proof}

\begin{fct} \label{fct:h1fin} The set $H^1(u \rw W,Z \rw G)$ is finite.
\end{fct}
\begin{proof}
The finiteness of $Z$ implies that $\tx{Hom}(u,Z)^\Gamma=\varinjlim\tx{Hom}(u_{E/F,n},Z)^\Gamma$ is also finite. For any $z \in Z^1(u \rw W,Z \rw G)$, let $G^z$ be the inner form of $G$ obtained by twisting the $\Gamma$-structure by $z$. Then the fiber of $H^1(u \rw W,Z \rw G) \rw \tx{Hom}(u,Z)^\Gamma$ through the class of $z$ is identified under \eqref{eq:infres} with the set $H^1(\Gamma,G^z)$, which is also known to be finite \cite[Thm. 6.14]{PR94}.
\end{proof}

\begin{pro} \label{pro:spltor} Let $[Z \rw S] \in \mc{A}$ and assume that $S$ is a split torus. Then $H^1(u \rw W,Z \rw S)=0$.
\end{pro}
\begin{proof}
\vspace{-10pt}By assumption, the groups $Z$ and $\bar S = S/Z$ are split, hence both maps $\tx{Hom}(u,Z)^\Gamma \rw H^2(\Gamma,Z) \rw H^2(\Gamma,S)$ are injective. The proposition follows from \eqref{eq:infres}.
\end{proof}

\begin{pro} \label{pro:bfd} Let $[Z \rw G] \in \mc{A}$. Put $\bar G=G/Z$. Then we have the commutative diagram with exact rows and columns
\begin{equation} \label{eq:bfd} \xymatrix{
&\bar G(F)\ar[d]\ar@{=}[r]&\bar G(F)\ar[d]\\
1\ar[r]&H^1(\Gamma,Z)\ar[r]^-{\tx{Inf}}\ar[d]&H^1(u \rw W,Z \rw Z)\ar[r]^-{\tx{Res}}\ar[d]&\tx{Hom}(u,Z)^\Gamma\ar@{=}[d]\\
1\ar[r]&H^1(\Gamma,G)\ar[r]^-{\tx{Inf}}\ar@{=}[d]&H^1(u\rw W,Z\rw G)\ar[r]^-{\tx{Res}}\ar[d]^a&\tx{Hom}(u,Z)^\Gamma\ar[d]^-{\eqref{eq:h2map}}\ar[r]&{*}\ar@{=}[d]\\
&H^1(\Gamma,G)\ar[r]&H^1(\Gamma,\bar G)\ar[r]\ar[d]&H^2(\Gamma,Z)\ar[d]\ar[r]&{*}\\
&&1&1
} \end{equation}
where $*$ is to be taken as $H^2(\Gamma,G)$ if $G$ is abelian and disregarded otherwise.
\end{pro}
\begin{proof}
The second and third rows come from \eqref{eq:infres}, the fourth row and the left column come from the long exact sequence for $\Gamma$-cohomology of the short exact sequence $1 \rw Z \rw G \rw \bar G \rw 1$. The middle column comes from the long exact sequence for $W$-cohomology associated to the same short exact sequence. The commutativity of all squares is obvious, except for the bottom right one, which is the content of Lemma \ref{lem:tg}, and the bottom middle one, which we turn to now.

Choose again a continuous section $s : \Gamma \rw W$. Let $z \in Z^1(u \rw W,Z \rw G)$. Then $c(\sigma) = z(s(\sigma))$ is an element of $C^1(\Gamma,G)$ which lifts $a(z)$, so $dc \in Z^2(\Gamma,Z)$ is the image of $a(z)$ under the connecting homomorphism. Using the fact that $z$ is a cocycle we see that $dc(\sigma,\tau) = z(\dot\xi(\sigma,\tau))$, where $\dot\xi(\sigma,\tau)=s(\sigma)s(\tau)s(\sigma\tau)^{-1}$ represents the class $\xi$.

To complete the proof of the proposition, we need to establish the surjectivity of the map $a$. If $G$ is abelian, this surjectivity follows from the already established surjectivity of \eqref{eq:h2map} and the four-lemma. For a general $G$, let $R \subset G$ be a complement to the unipotent radical of $G$, i.e. a Levi subgroup \cite[Thm 2.3]{PR94}, and let $S \subset R$ be a fundamental maximal torus \cite[\S10]{Kot86}. We have $Z \subset S$
and $\bar R \subset \bar G$ is a Levi subgroup as well and $\bar S \subset \bar R$ is a fundamental maximal torus. Then we have the diagram
\[ \xymatrix{
H^1(u \rw W,Z \rw S)\ar[r]\ar[d]&H^1(u \rw W,Z \rw G)\ar[d]\\
H^1(\Gamma,\bar S)\ar[r]&H^1(\Gamma,\bar G)
} \]
We already know that the left vertical map is surjective, and according to \cite[Prop. 9.2]{PR94} and \cite[Lem. 10.2]{Kot86}, the bottom horizontal map is surjective. It follows that the right vertical map is also surjective.
\end{proof}

\begin{cor} \label{cor:ellsur} Let $[Z \rw G] \in \mc{R}$.
\begin{enumerate}
\item If $G$ possesses anisotropic maximal tori, then the map $H^1(u \rw W,Z \rw G) \rw \tx{Hom}(u,Z)^\Gamma$ in Proposition \ref{pro:bfd} is surjective.
\item If $S \subset G$ is a fundamental torus, then the map
\[ H^1(u \rw W,Z \rw S) \rw H^1(u \rw W, Z \rw G) \]
is surjective.
\end{enumerate}
\end{cor}
\begin{proof}
\vspace{-10pt}The first point follows from the fact that if $S$ is an anisotropic torus, then $H^2(\Gamma,S)$ vanishes, so by \eqref{eq:infres} the map $H^1(u \rw W,Z \rw S) \rw \tx{Hom}(u,Z)^\Gamma$ is surjective, and then so is $H^1(u \rw W,Z \rw G) \rw \tx{Hom}(u,Z)^\Gamma$.

The second point follows right away from the four-lemma applied to the diagram
\[ \xymatrix{
H^1(u \rw W,Z \rw Z)\ar[r]\ar@{=}[d]&H^1(u \rw W,Z \rw S)\ar[r]\ar[d]&H^1(\Gamma,\bar S)\ar[r]\ar[d]&1\\
H^1(u \rw W,Z \rw Z)\ar[r]&H^1(u \rw W,Z \rw G)\ar[r]&H^1(\Gamma,\bar G)\ar[r]&1\\
} \]
\end{proof}

\begin{cor} \label{cor:surjad}
Let $G$ be a connected reductive group defined over $F$, let $Z$ be the center of $G_\tx{der}$, and set $\bar G=G/Z$. Then both natural maps
\[ H^1(u \rw W,Z \rw G ) \rw H^1(\Gamma,\bar G) \rw H^1(\Gamma,G_\tx{ad}) \]
are surjective. If $F$ is $p$-adic and $G$ is split, then both maps are bijective. If $F=\R$ and $G$ is split, then the second map is bijective and the first map has trivial kernel (but possibly non-trivial fibers away from the neutral element).
\end{cor}
\begin{proof}
The surjectivity of the first map is already stated in Proposition \ref{pro:bfd}, while that of the second maps follows from the fact that
$\bar G$ is the direct product $G_\tx{ad} \times Z(G)/Z$. Assume now that $G$ is split. The group $Z(G)/Z$ is a split torus and has trivial first cohomology, which accounts for the injectivity of the second map.

We will now discuss the first map. The bijectivity of \eqref{eq:h2map} implies that the map $\tx{Res}$ in the second row of Diagram \eqref{eq:bfd} is trivial, and hence that $\tx{Inf} : H^1(\Gamma,Z) \rw H^1(u \rw W,Z \rw Z)$ is bijective. From this it follows that the kernel of the the first map in the corollary, which coincides with the kernel of the composition
\[ H^1(u \rw W,Z \rw G ) \rw H^1(\Gamma,\bar G) \rw H^1(\Gamma,G_\tx{ad}), \]
must be equal to the kernel of the map $H^1(\Gamma,G) \rw H^1(\Gamma, G_\tx{ad})$. Using \cite[Thm 1.2]{Kot86} and that fact that the map $\hat Z_\tx{sc} \rw \pi_0(\hat Z)$ is surjective, we see that this kernel is equal to the image of $H^1(\Gamma,Z(G_\tx{sc})) \rw H^1(\Gamma,G)$.
This image however is trivial. Indeed, the map passes through $H^1(\Gamma,G_\tx{sc})$. When $F$ is $p$-adic, Kneser's theorem \cite{Kne65} shows that this set is trivial and this implies that the kernel of the first map in the corollary is trivial. Applying the usual twisting argument and using the fact that all arguments remain valid under twisting, we see that all other fibers are trivial as well.

If $F$ is real we will instead show that the connecting homomorphism $G_\tx{ad}(F) \rw H^1(\Gamma,Z(G_\tx{sc}))$ is surjective. If $T_\tx{sc} \subset G_\tx{sc}$ is a split maximal torus, with image $T_\tx{ad}$ in $G_\tx{ad}$, then $T_\tx{ad}(F) \rw H^1(\Gamma,Z(G_\tx{sc}))$ is surjective because $H^1(\Gamma,T_\tx{sc})$ is trivial and this suffices. Again we conclude that the kernel of the first map in the corollary is trivial. However, since this argument is not invariant under twisting any more (as it used the existence of a split torus), we cannot transfer this conclusion to other fibers.
\end{proof}

\begin{rmk} The following is a simple example that shows that some fibers of the first map in Corollary \ref{cor:surjad} can be non-trivial for $F=\R$ and $G$ a split group. Indeed, let $G=\tx{SL}_2/\R$. Then $H^1(\R,G_\tx{ad})$ has two elements. The fiber of the surjection
$H^1(u \rw W,Z \rw G) \rw H^1(\Gamma,G_\tx{ad})$
over the non-trivial element can be identified via twisting with the kernel of
$H^1(u \rw W,Z \rw G') \rw H^1(\Gamma,G'_\tx{ad})$
where $G'$ is the non-trivial inner form of $\tx{SL}_2/\R$, namely the anisotropic group $\tx{SU}_2$. The arguments of Corollary \ref{cor:surjad} show that this kernel is equal to the image of $H^1(\Gamma,Z) \rw H^1(\Gamma,G')$. One computes using \cite[Thm. 6.17]{PR94} that this map is a bijection between two sets of order $2$.
\end{rmk}

\subsection{Definition of $H^1_\tx{sc}(u \rw W,Z \rw G)$} \label{sec:h1sc}
Let $[Z \rw G] \in \mc{R}$. We will now describe a quotient of $H^1(u \rw W,Z \rw G)$ which will be used in the next section. For this, we impose on $H^1(u \rw W,Z \rw G)$ the following equivalence relation: Let $z_1,z_2 \in Z^1(u \rw W,Z \rw G)$. Let $G^1$ be the twist of $G$ by the image of $z_1$ in $Z^1(\Gamma,G_\tx{ad})$. Tautologically $z_2\cdot z_1^{-1} \in Z^1(u \rw W,Z \rw G^1)$ and we say that the images of $z_1$ and $z_2$ in $H^1(u \rw W,Z \rw G)$ are equivalent if $z_2 \cdot z_1^{-1}$ belongs to the image of $Z^1(\Gamma,G^1_\tx{sc})$. One checks easily that this defines an equivalence relation on $H^1(u \rw W,Z \rw G)$ and we denote the quotient by $H^1_\tx{sc}(u \rw W,Z \rw G)$. Since every homomorphism of reductive groups lifts uniquely to a homomorphism of the simply-connected covers of their derived groups, we obtain a functor $H^1_\tx{sc}(u \rw W) : \mc{R} \rw \tx{Sets}$ and a surjective map
\[ H^1(u \rw W) \rw H^1_\tx{sc}(u \rw W) \]
which is an isomorphism whenever $H^1(\Gamma,G_\tx{sc})=1$. This condition holds when $F$ is $p$-adic by Kneser's theorem \cite{Kne65}, as well as when $G$ is a torus because then $G_\tx{sc}=1$.

We remark that in the same way we can define $H^1_\tx{sc}(\Gamma,G)$ by imposing the same equivalence relation on $H^1(\Gamma,G)$. In that situation, we have $H^1_\tx{sc}(\Gamma,G) \cong H^1(\Gamma,G_\tx{sc} \rw G)$, where $G_\tx{sc} \rw G$ is regarded as a crossed module placed in degrees $-1$ and $0$.

\section{The isomorphism $\bar Y_{+,\tx{tor}} \rw H^1_\textrm{sc}(u \rw W)$} \label{sec:tn+iso}

Recall that if $S$ is an algebraic torus defined over $F$ and split over a finite Galois extension $E/F$, there is an isomorphism $H^{-1}_\tx{Tate}(\Gamma_{E/F},X_*(S)) \rw H^1(\Gamma,S)$ \cite{Tat66}. The source of this isomorphism can be computed to be $X_*(S)_{\Gamma,\tx{tor}}$, the torsion submodule of the $\Gamma$-coinvariants of $X_*(S)$. This has the advantage of eliminating the dependence of this isomorphism on the finite extension $E$ and in this way one obtains an isomorphism $X_*(S)_{\Gamma,\tx{tor}} \rw H^1(\Gamma,S)$ which is functorial in $S$.

In this section we are going to define a functor
\[ \bar Y_{+,\tx{tor}} : \mc{R} \rw \tx{AbGrp}, \]
which extends the functor $S \mapsto X_*(S)_{\Gamma,\tx{tor}}$, as well as a morphism of functors from $\bar Y_{+,\tx{tor}}$ to the functor $[Z \rw G] \mapsto \tx{Hom}(u,Z)^\Gamma$. We will then prove that there exists a unique isomorphism
\[ \bar Y_{+,\tx{tor}} \rw H^1_\tx{sc}(u \rw W), \]
which for objects $[1 \rw S] \in \mc{T}$ coincides with the Tate-Nakayama isomorphism, and such that the composition $\bar Y_{+,\tx{tor}}(Z \rw G) \rw H^1_\tx{sc}(u \rw W,Z \rw G) \rw \tx{Hom}(u,Z)^\Gamma$ coincides with the morphism just alluded to.

\subsection{Definition of $\bar Y_{+,\tx{tor}}$} \label{sec:yfunc}

Let $[Z \rw S] \in \mc{T}$. As before we write $\bar S = S/Z$. Let $X=X^*(S)$, $\bar X = X^*(\bar S)$, $Y=X_*(S)$ and $\bar Y = X_*(\bar S)$. We have the exact sequences
\[ 0 \rw \bar X \rw X \rw X/\bar X \rw 0\qquad\tx{and}\qquad 0 \rw Y \rw \bar Y \rw \bar Y/Y \rw 0 \]
and we will identify $\bar X$ with its image in $X$ and $Y$ with its image in $\bar Y$. The abelian group $\bar Y/Y$ is finite and the $\Z$-pairing between $\bar Y$ and $\bar X$ provides a $\Q$-pairing between $\bar Y$ and $X$ which in turn provides a $\Gamma$-equivariant perfect pairing
\[ \bar Y/Y \otimes X/\bar X \rw \Q/\Z. \]
This perfect pairing can also be formulated as the isomorphism
\begin{equation} \label{eq:homz} \bar Y/Y \rw \tx{Hom}(\mu_n,Z),\qquad \bar\lambda \mapsto (x \mapsto x^{n\bar\lambda}), \qquad \tx{for } [\bar Y:Y]|n. \end{equation}
We will write $Y^N$ and $\bar Y^N$ for the kernel of the norm map for the action of the Galois group $\Gamma_{E/F}$ for any finite Galois extension of $E/F$ over which $S$ splits. If $I \subset \Z[\Gamma_{E/F}]$ is the augmentation ideal, we define $\bar Y_+ = \bar Y/IY$. The modules $Y^N$, $\bar Y^N$ and $IY$ are independent of the choice of $E$, and we have the exact sequence
\[ 0 \rw Y_\Gamma \rw \bar Y_+ \rw \bar Y/Y \rw 0, \]
where $Y_\Gamma = Y/IY$ is the module of $\Gamma$-coinvariants in $Y$. Write $Y^N_\Gamma$ and $\bar Y^N_+$ for the quotients of $Y^N$ and $\bar Y^N$ by $IY$. The following fact is easily observed.
\begin{fct} \label{fct:yseq} \label{fct:sbarytor} For any field extension $E/F$ splitting $S$, we have $\bar Y^N_+ = \bar Y_{+,\tx{tor}}$.
Moreover, we have the exact sequence
\[ 0 \rw Y_{\Gamma,\tx{tor}} \rw \bar Y_{+,\tx{tor}} \rw [\bar Y/Y]^N \stackrel{N}{\rw} Y^\Gamma/N(Y). \]
\end{fct}
Composing the map $\bar Y_{+,\tx{tor}} \rw [\bar Y/Y]^N$ with \eqref{eq:homz} and the inverse of \eqref{eq:homgamma} we obtain a homomorphism $\bar Y_{+,\tx{tor}} \rw \tx{Hom}(u_{E/F,n},Z)^\Gamma$. For varying $E/F$ and $n$ these homomorphisms are compatible and splice to a homomorphism
\begin{equation} \label{eq:yhomz} \bar Y_{+,\tx{tor}} \rw \tx{Hom}(u,Z)^\Gamma. \end{equation}
Given a morphism $[Z_1 \rw S_1] \rw [Z_2 \rw S_2]$ in $\mc{T}$, the induced morphism $\bar S_1 \rw \bar S_2$ gives rise to a morphism $X_*(\bar S_1)_{+,\tx{tor}} \rw X_*(\bar S_2)_{+,\tx{tor}}$. In other words, the assignment $[Z \rw S] \mapsto \bar Y_{+,\tx{tor}}$ is functorial, i.e. we obtain a functor
\[ \bar Y_{+,\tx{tor}} : \mc{T} \rw \tx{FinAbGrp}. \]
The homomorphism $\bar Y_{+,\tx{tor}} \rw \tx{Hom}(u,Z)^\Gamma$ is functorial in $[Z \rw S]$.

In order to extend the functor $\bar Y_{+,\tx{tor}}$ to $\mc{R}$, we need the following lemma.

\begin{lem} \label{lem:conj1} Let $[Z \rw G] \in \mc{R}$, and let $S_1,S_2 \subset G$ be maximal tori. Any $g \in G(\ol{F})$ with $S_2 = \tx{Ad}(g)S_1$ provides a $\Gamma$-equivariant isomorphism
\[ \tx{Ad}(g) : \bar Y_1/Q_1^\vee \rw \bar Y_2/Q_2^\vee, \]
where $\bar Y_i=X_*(S_i/Z)$, and $Q_i^\vee=X_*(S_{i,\tx{sc}})$. Moreover, this isomorphism is independent of the choice of $g$.
\end{lem}

\begin{proof}
It is clear that $\tx{Ad}(g)$ provides an isomorphism between its source and target. If we show that this isomorphism is independent of the choice of $g$ the $\Gamma$-equivariance will follow. We may assume $S_1=S_2=S$ and $g \in N(S,G)$. Let $w$ be the image of $g$ in the Weyl group $\Omega(S,G)$. We want to show that $w$ acts trivially on $\bar Y/Q^\vee$. The isogeny $S/Z \rw S/(Z \cdot Z(G_\tx{der})) = S_\tx{ad} \times G/(Z\cdot G_\tx{der})$ gives an injection $\bar Y \rw P^\vee \oplus X_*(G/Z\cdot G_\tx{der})$, where $P^\vee=X_*(S_\tx{ad})$ is the coweight lattice. Let $\bar y \in \bar Y$ and decompose it as $\bar y=p+z$ with $p \in P^\vee$ and $z \in X_*(G/Z\cdot  G_\tx{der})$. Then $wz=z$ and $wp-p \in Q^\vee$ by \cite[Ch VI, \S1, no. 10, Prop 27]{Bou}.
\end{proof}

Let $[Z \rw G] \in \mc{R}$. For a maximal torus $S \subset G$, we consider the expression
\[ \varinjlim \frac{[X_*(S/Z)/X_*(S_\tx{sc})]^N}{I(X_*(S)/X_*(S_\tx{sc}))}, \]
where the colimit is taken over the set of Galois extensions $E/F$ splitting $S$. We define $\bar Y_{+,\tx{tor}}([Z \rw G])$ to be the limit of the system whose objects are these expressions and whose morphisms are given by Lemma \ref{lem:conj1}.

Given a morphism $f : [Z \rw G] \rw [C \rw H]$ in $\mc{R}$, the map $f: G \rw H$ lifts uniquely to a map $f_\tx{sc}: G_\tx{sc} \rw H_\tx{sc}$. Choose maximal tori $S \subset G$ and $T \subset H$ such that $f(S) \subset T$. Restricting $f$ to $S$ we obtain a morphism $f: [Z \rw S] \rw [C \rw T]$ in $\mc{T}$ and a compatible homomorphism $S_\tx{sc} \rw T_\tx{sc}$, and hence a map
\begin{equation} \label{eq:yr}
\varinjlim \frac{[X_*(S/Z)/X_*(S_\tx{sc})]^N}{I(X_*(S)/X_*(S_\tx{sc}))}  \stackrel{f_{S,T}}{\lrw} \varinjlim \frac{[X_*(T/C)/X_*(T_\tx{sc})]^N}{I(X_*(T)/X_*(T_\tx{sc}))} . \end{equation}
If $S'$ and $T'$ are other choices of maximal tori of $G$ and $H$ with $f(S') \subset T'$, then there exist $g \in G$ and $h \in \tx{Cent}(f(S),H)$ such that $S'=\tx{Ad}(g)S$ and $T'=\tx{Ad}(hf(g))T$. The commutativity of
\[ \xymatrix{
S\ar[rr]^{\tx{Ad}(g)}\ar[d]^f&&S'\ar[d]^f\\
T\ar[rr]^{\tx{Ad}(hf(g))}&&T'
}\]
implies that the maps $f_{S,T}$ for all possible choices of $S$ and $T$ splice together to a map
\[ \bar Y_{+,\tx{tor}}(f) : \bar Y_{+,\tx{tor}}(G) \rw \bar Y_{+,\tx{tor}}(H). \]
This completes the definition of the functor
\[ \bar Y_{+,\tx{tor}} : \mc{R} \rw \tx{FinAbGrp}. \]
One checks that \eqref{eq:yhomz} extends to a homomorphism
\begin{equation} \label{eq:yhomz2} \bar Y_{+,\tx{tor}} \rw \tx{Hom}(u,Z)^\Gamma \end{equation}
of functors $\mc{R} \rw \tx{FinAbGrp}$.

\subsection{Uniqueness of the isomorphism} \label{sec:uni}

Let $\iota^{(1)},\iota^{(2)} : \bar Y_{+,\tx{tor}} \rw H^1_\tx{sc}(u \rw W)$ be two isomorphisms, both of which coincide with the Tate-Nakayama isomorphism for objects $[1 \rw S] \in \mc{T}$ and lift the morphism $\bar Y_{+,\tx{tor}} \rw \tx{Hom}(u,Z)^\Gamma$ defined by \eqref{eq:yhomz}. We will show that $\iota^{(1)}=\iota^{(2)}$.

\ul{Step 1:} Let $[Z \rw S] \in \mc{T}$ with $S$ an anisotropic torus. Then we have the equations $\varinjlim (\bar Y/Y)^{N_{E/F}}= X_*(\bar S)/X_*(S)$ and $\bar Y_{+,\tx{tor}}=X_*(\bar S)/IX_*(S)$ and conclude that the composition $(\iota^{(1)}_{[Z \rw S]})^{-1}\iota^{(2)}_{[Z \rw S]}$ is therefore an automorphism of the extension
\[ 0 \rw X_*(S)/IX_*(S) \rw X_*(S/Z)/IX_*(S) \rw X_*(S/Z)/X_*(S) \rw 0. \]
This automorphism induces the identities on the first and third terms and thus differs from the identity by a homomorphism
\[ \delta_{[Z \rw S]} : X_*(S/Z)/X_*(S) \rw X_*(S)/IX_*(S). \]
As we fix $S$ and vary $Z$, this homomorphism is still functorial in $Z$ and hence determines a homomorphism
\[ \delta_S : \varinjlim\limits_Z X_*(S/Z)/X_*(S) \rw X_*(S)/IX_*(S). \]
This homomorphism has a divisible source and a finite target and is thus zero.
Each individual homomorphism $\delta_{[Z \rw S]}$ is a restriction of $\delta_S$ and thus also zero.

\ul{Step 2:} Now let $[Z \rw S] \in \mc{T}$ be arbitrary. Let $S_a \subset S$ be the maximal anisotropic subtorus and let $Z_a = Z \cap S_a$. Then we obtain the diagram with exact rows
\[ \xymatrix@-.7pc{
H^1(u \rw W,Z_a \rw S_a)\ar[r]&H^1(u \rw W,Z \rw S)\ar[r]&H^1(u \rw W,Z/Z_a \rw S/S_a)\\
\bar Y_{+,\tx{tor}}(Z_a \rw S_a)\ar[u]\ar[r]&\bar Y_{+,\tx{tor}}(Z \rw S)\ar[u]\ar[r]&\bar Y_{+,\tx{tor}}(Z/Z_a \rw S/S_a)\ar[u]
}\]
According to Proposition \ref{pro:spltor}, the top third term vanishes, and then so must the bottom third. Thus $\iota^{(k)}_{[Z \rw S]}$ is determined by $\iota^{(k)}_{[Z_a \rw S_a]}$ for $k=1,2$ and by Step 1 we have $\iota^{(1)}_{[Z \rw S]} = \iota^{(2)}_{[Z \rw S]}$.

\ul{Step 3:} Let $[Z \rw G] \in \mc{R}$ and let $S \subset G$ be a fundamental maximal torus. According to Corollary \ref{cor:ellsur}, $\iota^{(k)}_{Z \rw G}$ is determined by $\iota^{(k)}_{Z \rw S}$.

\subsection{Unbalanced cup-products} \label{sec:ucp}

The construction of the isomorphism $\bar Y_{+,\tx{tor}} \rw H^1_\tx{sc}(u \rw W)$ will be based on a modified version of the cup-product between a positive-degree cochain and a negative-degree cochain in Tate cohomology. It is defined as follows. Let $\Delta \rw \Theta$ be a surjection of finite groups, let $A$ be a $\Delta$-module, and let $B$ a $\Theta$-module. Recall that for any integer $i$, the set of homogenous $i$-cochains $C^{i,\tx{hom}}_\tx{Tate}(\Theta,B)$ is defined as $\tx{Hom}_{\Theta}(P^\Theta_i,B)$, where $(P^\Theta_i)_{i \in \Z}$ is the standard complete resolution of the trivial $\Theta$-module $\Z$. Analogously, we have $C^{i,\tx{hom}}_\tx{Tate}(\Delta,A)=\tx{Hom}_{\Delta}(P^\Delta_i,A)$. When $i \geq 0$ the set $C^{i,\tx{hom}}_\tx{Tate}(\Delta,A)$ can be identified with the set of $\Delta$-equivariant functions from $\Delta^{i+1}$ to $A$, where $\Delta$ acts by diagonal left multiplication on $\Delta^{i+1}$. Moreover, we may work in the more general situation where $\Delta$ is not finite, but rather a compact topological group and $\Delta \rw \Theta$ is continuous, as long as we take the functions from $\Delta^{i+1}$ to $A$ to be continuous with respect to the discrete topology on $A$. We will also occasionally drop the subscript ``Tate'' in that situation.

Let $i>j>0$. We will be interested in a subset
\[ C^{i,j,\tx{hom}}(\Delta,\Theta,A) \subset C^{i,\tx{hom}}(\Delta,A) \]
defined as follows: An $i$-cochain of $\Delta$ with values in $A$ belongs to this subset if and only if its values remain unchanged when we multiply any of its last $j$-many variables by an element in the kernel of $\Delta \rw \Theta$. This is equivalent to saying that this $i$-cochain is the composition of a (continuous) function $\Delta^{i+1-j} \times \Theta^j \rw A$ with the natural projection $\Delta^{i+1} \rw \Delta^{i+1-j} \times \Theta^j$. Notice that the differential $d : C^{i,\tx{hom}}(\Delta,A) \rw C^{i+1,\tx{hom}}(\Delta,A)$ carries $C^{i,j,\tx{hom}}(\Delta,\Theta,A)$ to $C^{i+1,j,\tx{hom}}(\Delta,\Theta,A)$.

We can consider $B$ as a $\Delta$ module as well and form the $\Delta$-module $A \otimes B$. Let now $i>j' \geq j>0$. Given two cochains $f \in C^{i,j',\tx{hom}}(\Delta,\Theta,A)$ and $g \in C^{-j,\tx{hom}}_\tx{Tate}(\Theta,B)$, we define
\[ f \sqcup g \in C^{i-j,j'-j,\tx{hom}}_\tx{Tate}(\Delta, \Theta, A \otimes B) \]
using the formula
\[ (f \sqcup g)(g_0,\dots,g_{i-j}) = \sum_{(s_1,\dots,s_j) \in \Theta^j}f(g_0,\dots,g_{i-j},s_1,\dots,s_j) \otimes g(s_1^*,\dots,s_j^*),  \]
The condition on $f$ ensures that this formula makes sense. In the degenerate case $\Delta=\Theta$ this is just the formula for $f \cup g$ given in \cite[Prop. 1.4.7]{NSW08}. We also define
\[ (f \sqcup g)(g_0,\dots,g_i) = (f \cup g)(g_0,\dots,g_i) = f(g_0,\dots,g_i)\otimes g(g_i) \]
in the case $j=0$. Then we observe that $\sqcup$ has the following feature in common with the usual cup-product.

\begin{fct} \label{fct:dsqcup}
\[ d(f \sqcup g) = df \sqcup g + (-1)^i f \sqcup dg. \]
\end{fct}
The proof relies on the fact that for $i \geq 0$ the differential for homogenous $i$-cochains is given by a formula which is independent of the group, i.e. it is the same for both $\Delta$ and $\Theta$ -- a fact that is not true for $i<0$. Once this has been observed, the proof is straightforward and we shall leave it to the reader.

It will be more convenient to work with inhomogenous cochains in the subsequent sections. For $i>j' \geq j > 0$, the set of inhomogenous $i$-cochains $C^i(\Delta,A)$ is then just the set of (continuous) functions $\Delta^i \rw A$, and the subset $C^{i,j'}(\Delta,\Theta,A)$ is defined by the same condition as for the homogenous case. For negative degrees, we will be particularly interested in the case $j=1$. An inhomogenous $(-1)$-cochain with values in $B$ is simply an element of $B$.  Given $f \in C^{i,1}(\Delta,\Theta,A)$ and $\lambda \in B$. We have
\[ (f \sqcup \lambda)(g_1,\dots,g_{i-1}) = \sum_{a \in \Theta} f(g_1,\dots,g_{i-1},a) \otimes g_1\dots g_{i-1}a \lambda. \]

In a situation where more than one pair of groups $\Delta \rw \Theta$ is involved, we will write
$f \underset{\Theta}{\sqcup} g$ to keep track of which group is being used. When $\Theta$ is the Galois group of some finite Galois extension $E/F$, we will also write $f \underset{E/F}{\sqcup} g$.

\subsection{Arithmetic preparations} \label{sec:arith}
We choose again an increasing tower $E_k$ of Galois extensions of $F$ with $\bigcup E_k=\ol{F}$. For each $k$, the relative Weil group $W_{E_k/F} = W_F/W_{E_k}^c$ fits into the exact sequence \cite{Tat79}
\[ \xymatrix{
1\ar[r]&E_k^\times\ar[r]^-{\tx{rec}_k}&W_{E_k/F}\ar[r]^-{p_k}&\Gamma_{E_k/F}\ar[r]&1.
} \]
For each $k$, choose an arbitrary (set-theoretic) section $\zeta_k$ for the natural projection $\pi^\Gamma_k: \Gamma_{E_{k+1}/F} \rw \Gamma_{E_k/F}$. Every element of $\Gamma_{E_{k+1}/F}$ can be written as a product $y\zeta_k(x)$ for unique $y \in \Gamma_{E_{k+1}/E_k}$ and $x \in \Gamma_{E_k/F}$. For each $k$, choose inductively a section $s_{k+1} : \Gamma_{E_{k+1}/F} \rw W_{E_{k+1}/F}$ of $p_{k+1}$ with the properties
\[ s_{k+1}(y\zeta_k(x)) = s_{k+1}(y)s_{k+1}(\zeta_k(x))\quad\tx{and}\quad s_k(x)=\pi^W_k(s_{k+1}(\zeta_k(x))), \]
where $\pi^W_k: W_{E_{k+1}/F} \rw W_{E_k/F}$ is the natural projection.

Define $c_k \in Z^2(\Gamma_{E_k/F},E_k^\times)$ by $c_k(\sigma,\tau) = \tx{rec}_k^{-1}(s_k(\sigma) s_k(\tau) s_k(\sigma\tau)^{-1})$. Then $c_k$ represents the canonical class of the extension $E_k/F$. The following lemma expresses the compatibility between the different $c_k$. Its proof is contained in the discussion found in \cite[VI.1]{Lan83} and we reproduce it here for the convenience of the reader.
\begin{lem} \label{lem:fcc}
For any $\sigma,\tau \in \Gamma_{E_{k+1}/F}$ we have
\[ c_k(\pi^\Gamma_k(\sigma),\pi^\Gamma_k(\tau)) =\sprod{v \in \Gamma_{E_{k+1}/E_k}} c_{k+1}(v\sigma,\zeta_k(\pi^\Gamma_k(\tau))) = \sprod{v \in \Gamma_{E_{k+1}/E_k}}c_{k+1}(\sigma,v\tau)c_{k+1}(\sigma,v)^{-1}. \]
\end{lem}
\begin{proof}
For any $x,y \in \Gamma_{E_k/F}$ and $z,w \in \Gamma_{E_{k+1}/E_k}$, our choice of the sections $s_k$ implies the following equations, whose derivation we leave to the reader:
\begin{eqnarray}
c_{k+1}(w,z)&=&c_{k+1}(w,z\zeta_k(x)). \label{eq:ck}\\
c_k(x,y)&=&N_{E_{k+1}/E_k}(c_{k+1}(\zeta_k(x),\zeta_k(y)))\cdot \tx{rec}_k^{-1}(\pi^W_k(s_{k+1}(c_\Gamma(x,y)))).\qquad \label{eq:ce}
\end{eqnarray}
Here $c_\Gamma :\Gamma_{E_k/F}^2 \rw \Gamma_{E_{k+1}/E_k}$ is defined by $c_\Gamma(x,y) = \zeta_k(x)\zeta_k(y)\zeta_k(xy)^{-1}$ and the second factor in the right hand side of \eqref{eq:ce} is defined because the function $\pi^W_k\circ s_{k+1}$ maps $\Gamma_{E_{k+1}/E_k}$ into the image of $\tx{rec}_k$. Given $z \in \Gamma_{E_{k+1}/E_k}$, we have
\[ \tx{rec}_k^{-1}(\pi^W_k(s_{k+1}(z))) = \tx{rec}_{k+1}^{-1}(\tx{tr}(\pi^W_k(s_{k+1}(z)))),  \]
where $\tx{tr} : W_{E_k}^\tx{ab} \rw W_{E_{k+1}}^\tx{ab}$ is the transfer map. Because $W_{E_k}^\tx{ab}=W_{E_{k+1}/E_k}^\tx{ab}$, this transfer map is equal to the transfer map $\tx{tr} : W_{E_{k+1}/E_k}^\tx{ab} \rw W_{E_{k+1}/E_{k+1}}^\tx{ab} = W_{E_{k+1}/E_{k+1}}$. In order to compute it, we need a section of the natural projection $W_{E_{k+1}/E_k} \rw W_{E_{k+1}/E_k}/W_{E_{k+1}/E_{k+1}} = \Gamma_{E_{k+1}/E_k}$, and for this we can take $s_{k+1}$. Then we have
\[ \tx{rec}_{k+1}^{-1}\!(\tx{tr}(\pi^W_k\!\!(s_{k+1}(z))))\!=\!\sprod{v \in \Gamma_{E_{k+1}/E_k}} \tx{rec}_{k+1}^{-1}\!(s_{k+1}(v)s_{k+1}(z) s_{k+1}(v z)^{-1}\!)\!=\!\sprod{v \in \Gamma_{E_{k+1}/E_k}} c_{k+1}(v,z). \]
Plugging this into equation \eqref{eq:ce} we obtain
\[ c_k(x,y) = \prod_{v \in \Gamma_{E_{k+1}/E_k}} vc_{k+1}(\zeta_k(x),\zeta_k(y))\cdot c_{k+1}(v,c_\Gamma(x,y)) \]
and according to equation \eqref{eq:ck} this leads to
\begin{eqnarray*}
c_k(x,y)&=&\prod_{v \in \Gamma_{E_{k+1}/E_k}} vc_{k+1}(\zeta_k(x),\zeta_k(y))\cdot c_{k+1}(v,c_\Gamma(x,y)\zeta_k(xy))\\
&=&\prod_{v \in \Gamma_{E_{k+1}/E_k}} vc_{k+1}(\zeta_k(x),\zeta_k(y))\cdot c_{k+1}(v,\zeta_k(x)\zeta_k(y))\\
&=&\prod_{v \in \Gamma_{E_{k+1}/E_k}}c_{k+1}(v\zeta_k(x),\zeta_k(y)) \cdot c_{k+1}(v,\zeta_k(x))\\
&=&\prod_{v \in \Gamma_{E_{k+1}/E_k}}c_{k+1}(v\zeta_k(x),\zeta_k(y)).
\end{eqnarray*}
This completes the proof of the first of the two equations that are stated in the lemma. The second equation follows from
\begin{eqnarray*}
&&c_k(\pi^\Gamma_k(\sigma),\pi^\Gamma_k(\tau))\\
&=&\prod_{v \in \Gamma_{E_{k+1}/E_k}}c_{k+1}(\sigma v,\zeta_k(\pi^\Gamma_k(\tau)))\\
&=&\prod_{v \in \Gamma_{E_{k+1}/E_k}}\sigma c_{k+1}(v,\zeta_k(\pi^\Gamma_k(\tau))) \cdot c_{k+1}(\sigma,v\zeta_k(\pi^\Gamma_k(\tau))) \cdot c_{k+1}(\sigma,v)^{-1} \\
\end{eqnarray*}
and the fact that $c_{k+1}(v,\zeta_k(\pi^\Gamma_k(\tau)))=1$ according to equation \eqref{eq:ck}.
\end{proof}

\subsection{An explicit realization of the extension of $\Gamma$ by $u$} \label{sec:expw}

In the previous section we chose the sequence of extensions $E_k$ as well as the maps $\zeta_k$ and $s_k$. We obtained 2-cocycles $c_k \in Z^2(\Gamma_{E_k/F},E_k^\times)$ representing the canonical classes. We choose in addition a co-final sequence $\{n_k\} \subset \N^\times$ and maps $l_k : \ol{F}^\times \rw \ol{F}^\times$ having the properties $l_k(x)^{n_k}=x$ and $l_{k+1}(x)^\frac{n_{k+1}}{n_k} = l_k(x)$ for all $x \in \ol{F}^\times$.

For each $k$ write again $u_k=u_{E_k/F,n_k}$. Recall the homomorphism $\delta_e : \mu_{n_k} \rw R_{E_k/F}[n_k]$ which sends $x$ to the map $\Gamma_{E_k/F} \rw \mu_{n_k}$ supported at $e$ and having value $x$ there. It induces a homomorphism $\delta_e : \mu_{n_k} \rw u_k$ which is easily seen to be killed by the norm map for the group $\Gamma_{E_k/F}$ acting on $\tx{Hom}(\mu_{n_k},u_k)$. Thus $\delta_e \in Z^{-1}_\tx{Tate}(\Gamma_{E_k/F},\tx{Hom}(\mu_{n_k},u_k))$. On the other hand we have the cochain $l_kc_k \in C^{2,2}(\Gamma_{F},\Gamma_{E_k/F},\ol{F}^\times)$ and its differential $dl_kc_k \in Z^{3,2}(\Gamma_F,\Gamma_{E_k/F},\mu_{n_k})$. We define
\begin{equation} \label{eq:xidef} \xi_k = dl_kc_k \underset{E_k/F}{\sqcup} \delta_e \in Z^2(\Gamma,u_k), \end{equation}
and let $W_k = u_k \boxtimes_{\xi_k} \Gamma$ be the extension of $\Gamma$ by $u_k$ determined by this 2-cocycle. We will now arrange the extensions $W_k$ into a projective system. In order to define the transition maps, it will be convenient to introduce the torus $R_{E_k/F}=\tx{Res}_{E_k/F}\mb{G}_m$ and let $S_{E_k/F}$ by the quotient of $R_{E_k/F}$ by the diagonally embedded copy of $\mb{G}_m$. Then $u_k$ is the subgroup of $n_k$-torsion points in $S_{E_k/F}$. Recall the map $p : u_{k+1} \rw u_k$ defined by \eqref{eq:defp}. Define
\begin{equation} \label{eq:alphake}
\alpha_k = (l_kc_k \underset{E_k/F}{\sqcup} \delta_e)^{-1} \cdot  p(l_{k+1}c_{k+1} \underset{E_{k+1}/F}{\sqcup} \delta_e) \in C^1(\Gamma,S_{E_k/F}), \end{equation}

\begin{lem} \begin{enumerate}
\item The cochain $\alpha_k$ takes values in $u_k$ and the equality $d\alpha_k = p(\xi_{k+1})\xi_k^{-1}$ holds in $C^2(\Gamma,u_k)$.
\item The element $([\xi_k])$ of $\varprojlim H^2(\Gamma,u_k)$ is equal to the canonical class $\xi$.
\end{enumerate}
\end{lem}
\begin{proof}
In order to prove the first statement, we will rewrite $\alpha_k$ as follows. Define for $\sigma,\tau \in \Gamma_{E_{k+1}/F}$ the element
\[ \eta_k(\sigma,\tau) = l_kc_k(\pi^\Gamma_k(\sigma),\pi^\Gamma_k(\tau))^{-1}\cdot\prod_{v \in \Gamma_{E_{k+1}/E_k}}\left[ l_kc_{k+1}(\sigma,\tau v)\cdot l_kc_{k+1}(\sigma,v)^{-1}\right]. \]
According to Lemma \ref{lem:fcc} we have $\eta_k \in C^{2,1}(\Gamma,\Gamma_{E_k/F},\mu_{n_k})$ and we claim that
\[ \alpha_k = \eta_k \underset{E_k/F}{\sqcup} \delta_e \in C^1(\Gamma,u_k). \]
Indeed, in $C^2(\Gamma,S_{E_k/F})$ one computes that
\[ \left[(\sigma,\tau) \mapsto \prod_v l_kc_{k+1}(\sigma,\tau v)\right] \underset{E_k/F}{\sqcup} \delta_e = p(l_{k+1}c_{k+1} \underset{E_{k+1}/F}{\sqcup} \delta_e) = p(\xi_{k+1}), \]
while
\[ \left[(\sigma,\tau) \mapsto \prod_v l_kc_{k+1}(\sigma,v)\right] \underset{E_k/F}{\sqcup} \delta_e \]
represents an element of $R_{E_k/F}$ which lies in the image of the diagonal embedding of $\mb{G}_m$ and is thus trivial in $S_{E_k/F}$.

Turning to the second statement, we need to show that under the isomorphism $H^2(\Gamma,u_k) \rw H^0(\Gamma,X^*(u_k))^* \rw \Z/(n_k,[E_k:F])\Z$ used in the proof of Theorem \ref{thm:cohcomp} the class of $\xi_k$ maps to the element $-1$. For this we compute the cup-product of $\xi_k$ with the element $\frac{n_k}{(n_k,[E_k:F])} \in \frac{n_k}{(n_k,[E_k:F])}\Z/n_k\Z \cong H^0(\Gamma,X^*(u_k))$ and obtain the $\frac{n_k}{(n_k,[E_k:F])}$-th power of the element of $C^2(\Gamma,\mu_{n_k})$ given by
\[ (\sigma,\tau) \mapsto \prod_{a \in \Gamma_{E_k/F}} dl_kc_k(\sigma,\tau,(\sigma\tau)^{-1}a). \]
In $C^2(\Gamma,\ol{F}^\times)$ this power is cohomologous to the $[E_k:F]/(n_k,[E_k:F])$-th power of $c_k^{-1}$ and is thus a 2-cocycle of invariant $-1/(n_k,[E_k:F])$. It follows that the class $\xi_k$ corresponds to the character of $H^0(\Gamma,X^*(u_k))$ which sends the element $\frac{n_k}{(n_k,[E_k:F])}$ to $\frac{-1}{(n_k,[E_k:F])} \in \Q/\Z$.
\end{proof}

It follows from the first part of the above lemma that the map
\[ W_{k+1} \rw W_k,\qquad x \boxtimes \sigma \mapsto p(x)\alpha_k(\sigma) \boxtimes \sigma, \]
is a homomorphism of extensions. Since it is surjective, the limit $W = \varprojlim W_k$ is an extension of $\Gamma$ by $u$. The element of $H^2(\Gamma,u)=\varprojlim H^2(\Gamma,u_k)$ defined by this extension is given by the system $([\xi_k])$. Thus, by the second part of the above lemma, the extension $W$ belongs to the isomorphism class of extensions of $\Gamma$ by $u$ determined by $\xi$.

We also have the following explicit description of the homomorphism \eqref{eq:h2map}, which follows right away from the explicit formula \eqref{eq:xidef} for $\xi_k$.
\begin{fct} \label{fct:h2mapexp}
Let $\phi \in \tx{Hom}(u_k,Z)^\Gamma$ and let $\varphi = \phi \circ\delta_e \in \tx{Hom}(\mu_{n_k},Z)^{N_{E_k/F}}$. Then
\[ \phi(\xi_k) = dl_kc_k \underset{E_k/F}{\sqcup} \varphi. \]
\end{fct}

\subsection{Construction of the isomorphism in the case of tori} \label{sec:tn+s}

Let now $[Z \rw S] \in \mc{T}$. We write again $Y=X_*(S)$ and $\bar Y=X_*(S/Z)$. Let $k$ be large enough so that $E_k$ splits $S$ and $|Z|$ divides $n_k$. Let $\bar\lambda \in \bar Y^{N_{E_k/F}}$ and let $\phi_{\bar\lambda,k} \in \tx{Hom}(u_k,Z)^\Gamma$ be its image under the isomorphism
\[ [\bar Y/Y]^{N_{E_k/F}} \rw \tx{Hom}(\mu_{n_k},Z)^{N_{E_k/F}} \rw \tx{Hom}(u_k,Z)^\Gamma \]
given by the composition of \eqref{eq:homz} and \eqref{eq:homgamma}. For $x \boxtimes \sigma \in W_k = u_k \boxtimes \Gamma$, set
\[ z_{\bar\lambda,k}(x \boxtimes \sigma) = \phi_{\bar\lambda,k}(x) \cdot (l_kc_k \underset{E_k/F}{\sqcup} n_k\bar\lambda)(\sigma) \in S(\ol{F}). \]
Then $z_{\bar\lambda,k}$ is a map $W_k \rw S$ which we inflate to a map $W \rw S$.
\begin{lem} \label{lem:zdef} The map $z_{\bar\lambda,k} : W \rw S$ belongs to $Z^1(u \rw W,Z \rw S)$. If $l>k$, then $z_{\bar\lambda,l}$ and $z_{\bar\lambda,k}$ are equal in $Z^1(u \rw W,Z \rw S)$.
\end{lem}
\begin{proof}
By definition $z_{\bar\lambda,k}$ is a continuous 1-cochain $W \rw S$ whose restriction to $u$ factors through an algebraic homomorphism $u_k \rw Z$. It remains to show that this 1-cochain is in fact a 1-cocycle. We compute
\begin{eqnarray*}
dz_{\bar\lambda,k}(x\boxtimes \sigma,y \boxtimes\tau)&=&\phi_{\bar\lambda,k}(\xi_k(\sigma,\tau))^{-1} \cdot d(l_kc_k \sqcup n_k\bar\lambda)(\sigma,\tau)\\
(\tx{Fact \ref{fct:dsqcup}})&=&\phi_{\bar\lambda,k}(\xi_k(\sigma,\tau))^{-1}\cdot (dl_kc_k \sqcup n_k\bar\lambda)(\sigma,\tau)\\
(\tx{Fact \ref{fct:h2mapexp}})&=&1.
\end{eqnarray*}
To compare $z_{\bar\lambda,l}$ and $z_{\bar\lambda,k}$, we will show that the inflation of $z_{\bar\lambda,k}$ to $W_l$ is equal to $z_{\bar\lambda,l}$. For this it is enough to consider $l=k+1$. According to the definition of the transition map $W_{k+1} \rw W_k$, the inflation of $z_{\bar\lambda,k}$ takes at $x \boxtimes \sigma \in W_{k+1}$ the value
\[ \phi_{\bar\lambda,k}(p(x)\alpha_k(\sigma)) \cdot (l_kc_k \underset{E_k/F}{\sqcup} \delta_e)(\sigma). \]
One checks that the equalities
\[ \phi_{\bar\lambda,k} \circ p = \phi_{\bar\lambda,k+1},\quad \phi_{\bar\lambda,k}(l_kc_k \underset{E_k/F}{\sqcup} \delta_e)=l_kc_k \underset{E_k/F}{\sqcup} n_k\bar\lambda \]
hold and the proof is complete.
\end{proof}
We will henceforth denote the inflation of $z_{\bar\lambda,k}$ to a map $W \rw S(\ol{F})$ by $z_{\bar\lambda}$, due to its independence of $k$.

\begin{thm} \label{thm:tn+s} The assignment $\bar\lambda \mapsto z_{\bar\lambda}$ induces an isomorphism
\[ \iota: \bar Y_{+,\tx{tor}} \rw H^1(u \rw W) \]
of functors $\mc{T} \rw \tx{AbGrp}$. This isomorphism coincides with the Tate-Nakayama isomorphism for objects $[1 \rw S] \in \mc{T}$ and lifts the morphism \eqref{eq:yhomz}.
\end{thm}
\begin{proof}
The fact that $z_{\bar\lambda}$ is additive in $\bar\lambda$ is clear from its definition, so that we indeed obtain a homomorphism of groups from the subgroup of $\bar Y$ of elements killed by $N_{E_k/F}$ for some $k$, to the group $H^1(u \rw W,Z \rw S)$. The functoriality in $Z \rw S$ is also clear from the construction, as both factors in the product defining $z_{\bar\lambda,k}$ are functorial in $\bar\lambda$. If $\bar\lambda \in Y$, then $\phi_{\bar\lambda,k}=0$ and moreover $l_kc_k \sqcup n_k\bar\lambda=c_k \cup \bar\lambda$, thus $z_{\bar\lambda,k}$ represents the image in $H^1(\Gamma,S)$ of $\bar\lambda$ under the classical Tate-Nakayama isomorphism \cite{Tat66}. This shows that the group homomorphism $\bar\lambda \mapsto z_{\bar\lambda}$ annihilates $IY$ and thus descends to a homomorphism $\iota_{[Z \rw S]} : \bar Y_{+,\tx{tor}}(Z \rw S) \rw H^1(u \rw W,Z \rw S)$. It furthermore shows that the latter homomorphism is equal to the Tate-Nakayama isomorphism if $Z=1$. The fact that $\iota_{[Z \rw S]}$ lifts \eqref{eq:yhomz} is evident from the construction of $z_{\bar\lambda,k}$. What remains to be shown is that $\iota_{[Z \rw S]}$ is an isomorphism. For this, we consider the diagram
\[ \xymatrix@-.5pc{
1\ar[r]&H^1(\Gamma,S)\ar[r]&H^1(u \rw W,Z \rw S)\ar[r]&\tx{Hom}(u,Z)^\Gamma\ar[r]&H^2(\Gamma,S)\\
0\ar[r]&Y_{\Gamma,\tx{tor}}\ar[u]\ar[r]&\bar Y_{+,\tx{tor}}\ar[u]\ar[r]&\varinjlim[\bar Y/Y]^{N_k}\ar[u]\ar[r]&\varinjlim Y^\Gamma/N_k(Y)\ar[u]
} \]
where $N_k$ denotes the norm map for the action of $\Gamma_{E_k/F}$ and the limits are taken with respect to $k$, beginning with $k$ large enough. The transition maps for the first limit are the natural inclusions, and the transition map $k \rw k+1$ for the second limit is given by the norm map for the action of $\Gamma_{E_{k+1}/E_k}$ (in fact this action is trivial, so the map is just multiplication by $[E_{k+1}:E_k]$). The right-most bottom horizontal map is given by the system of maps $[\bar Y/Y]^{N_k} \rw Y^\Gamma/N_k(Y)$ sending $\bar\lambda +Y$ to $N_k(\bar\lambda)+N_k(Y)$. The right vertical map sends $\lambda \in Y^\Gamma$ to $\lambda \cup c_k^{-1}$, i.e. it is the composition of the negative of the Tate-Nakayama isomorphism $Y^\Gamma/N_k(Y) \rw H^2(\Gamma_{E_k/F},S(E_k))$ with the inclusion into $H^2(\Gamma,S)$.
The commutativity of the left and middle square has just been established. For the commutativity of the right square, we know that $\bar\lambda \in \bar Y$ maps up and across to $\phi_{\bar\lambda,k}(\xi_k) \in H^2(\Gamma,S)$, which by Fact \ref{fct:h2mapexp} equals $dl_kc_k \underset{E_k/F}{\sqcup} n_k\bar\lambda$. According to Fact \ref{fct:dsqcup}, we have in $C^2(\Gamma,S)$ the equality $d(l_kc_k \sqcup n_k\bar\lambda)=dl_kc_k\sqcup n_k\bar\lambda+l_kc_k\sqcup n_kd\bar\lambda$, which implies in $H^2(\Gamma,S)$ the equality $dl_kc_k\sqcup n_k\bar\lambda=-l_kc_k\sqcup n_kd\bar\lambda$. Since $\bar\lambda$ represents an element of $[\bar Y/Y]^{N_k}$ and the differential in degree $-1$ is the map $N_k$, we see that $d\bar\lambda \in Y^\Gamma$, so the right hand side of the last equation is equal to $-c_k \cup d\bar\lambda$, which coincides with the image of $\bar\lambda$ across and then up.

We have shown that the above diagram is commutative. The top row is exact by \eqref{eq:infres}, and the bottom row is exact, being derived from  the exact sequence of Fact \ref{fct:yseq}. We know that the first and third vertical maps are isomorphisms. We also know that the fourth vertical map is injective, being given by a system of compositions $Y_\Gamma/N_k(Y) \rw H^2(\Gamma_{E_k/F},S(E_k)) \rw H^2(\Gamma,S)$, of which the first is bijective and the second injective (recall that $k$ is large enough). We now appeal to the five-lemma and the proof is complete.
\end{proof}

\subsection{Construction of the isomorphism for reductive groups} \label{sec:tn+g}

We will now extend the isomorphism of Theorem \ref{thm:tn+s} to an isomorphism
\[ \iota: \bar Y_{+,\tx{tor}} \rw H^1_\tx{sc}(u \rw W) \]
of functors $\mc{R} \rw \tx{Sets}$. When $F$ is $p$-adic, this isomorphism will endow the set $H^1(u \rw W,Z \rw G)$ with the structure of an abelian group. Moreover, this group structure will be compatible with the group structure on $H^1(\Gamma,G)$ obtained by Kottwitz \cite[Thm 1.2]{Kot86} and with the natural group structure on $\tx{Hom}(u,Z)$, and the maps in diagram \eqref{eq:bfd} will all become group homomorphisms.

\begin{lem} \label{lem:ygq} Let $[Z \rw G] \in \mc{R}$ and let $S \subset G$ be a maximal torus. The fibers of the composition
\[ \bar Y_{+,\tx{tor}}(Z \rw S) \rw H^1(u \rw W,Z \rw S) \rw H^1_\tx{sc}(u \rw W,Z \rw G) \]
are torsors under the image of $X_*(S_\tx{sc})_{\Gamma,\tx{tor}}$ in $\bar Y_{+,\tx{tor}}(Z \rw S)$.
\end{lem}
\begin{proof}
The usual twisting argument reduces the question to studying the fiber of the given map over the equivalence class containing the trivial element. This is the preimage under $\bar Y_{+,\tx{tor}}(Z \rw S) \rw H^1(u \rw W,Z \rw G)$ of the the image of $H^1(\Gamma,G_\tx{sc})$ in $H^1(u \rw W,Z \rw G)$. An element of $H^1(u \rw W,Z \rw S)$ will map to that image only if it belongs to $H^1(\Gamma,S)$. This reduces the problem studying the preimage under
\[ Y_{\Gamma,\tx{tor}} \rw H^1(\Gamma,S) \rw H^1(\Gamma,G) \]
of the image of $H^1(\Gamma,G_\tx{sc})$. According to \cite[Thm 1.2]{Kot86}, this preimage is dual to the cokernel of the map
\[ \pi_0(Z(\hat G)^\Gamma) \rw \pi_0(\hat S^\Gamma). \]
The cokernel of this map is equal to $\pi_0(\hat S^\Gamma/Z(\hat G)^\Gamma)$ and this is a subgroup of $\pi_0([\hat S/Z(\hat G)]^\Gamma)$. It follows that dually the map $H^1(\Gamma,S_\tx{sc}) \rw H^1(\Gamma,S) \rw H^1(\Gamma,G)$ surjects onto the image of $H^1(\Gamma,G_\tx{sc})$. The lemma follows.
\end{proof}

\begin{lem} \label{lem:conj2} Let $[Z \rw G] \in \mc{R}$ and let $S_1,S_2 \subset G$ be maximal tori. Let $g \in G(\ol{F})$ with $\tx{Ad}(g)S_1=S_2$. If $\bar\lambda_i \in \bar Y_i^N$ are such that $\bar\lambda_2 = \tx{Ad}(g)\bar\lambda_1$, then the images of $\iota_{[Z \rw S_1]}(\bar\lambda_1)$ and $\iota_{[Z \rw S_2]}(\bar\lambda_2)$ in $H^1_\tx{sc}(u \rw W,Z \rw G)$ are equal.
\end{lem}
\begin{proof}
Consider the isogeny $S_i/Z \rw S_i/(Z\cdot Z(G_\tx{der}))$. It provides an injection $\bar Y_i \rw P_i^\vee \oplus X_*(G/Z\cdot G_\tx{der})$, where $P_i^\vee = X_*(S_{i,\tx{ad}})$. Write $\bar\lambda_1 = p_1 + z$ accordingly. Then $\bar\lambda_2 = p_2+z$ with $p_2=\tx{Ad}(g)p_1$. The map $Z(G)^\circ \rw G/Z\cdot G_\tx{der}$ is an isogeny and leads to an injection $X_*(Z(G)^\circ) \rw X_*(G/Z\cdot G_\tx{der})$ with finite cokernel. We choose $k$ large enough so that $n_kp_1 \in Q_1^\vee = X_*(S_{1,\tx{sc}})$ and $n_kz \in X_*(Z(G)^\circ)$. By construction (Lemma \ref{lem:zdef}), we have for $x \boxtimes \sigma \in W_k$
\[ z_{\bar\lambda_i,k}(x \boxtimes \sigma) = \phi_{\bar\lambda_i,k}(x) \cdot [l_kc_k \sqcup n_kp_i](\sigma) \cdot [l_kc_k \sqcup n_kz](\sigma). \]
We have
\[ \phi_{\bar\lambda_1,k}(x)=\phi_{\bar\lambda_2,k}(x) \in Z \quad \tx{and} \quad [l_kc_k \sqcup n_kz](\sigma) \in Z(G)^\circ \]
and conclude that
\[ z_{\bar\lambda_2,k}(x \boxtimes \sigma) \cdot z_{\bar\lambda_1,k}(x \boxtimes \sigma)^{-1} = a_2(\sigma) \cdot a_1(\sigma)^{-1} \]
where $a_i = l_kc_k \sqcup n_kp_i \in C^1(\Gamma,S_{i,\tx{sc}})$. The image of $a_1$ in $C^1(\Gamma,S_{1,\tx{ad}})$ is equal to $c_k \cup p_1$ and is thus a 1-cocycle, so we can twist the $\Gamma$-structure on $G_\tx{sc}$ using it. We call the twisted structure $G_\tx{sc}^1$. We need to show that
\[ a_2\cdot a_1^{-1} \in Z^1(\Gamma,G_\tx{sc}^1). \]
For this, we compute the coboundary at $(\sigma,\tau)$ and obtain
\begin{eqnarray*}
&&[a_2(\sigma) a_1(\sigma)^{-1}] \cdot a_1(\sigma)\sigma [a_2(\tau)a_1(\tau)^{-1}]a_1(\sigma)^{-1} \cdot [a_2(\sigma\tau)a_1(\sigma\tau)^{-1}]^{-1}\\
&=&a_2(\sigma)\sigma a_2(\tau) [\sigma a_1(\tau)^{-1}a_1(\sigma)^{-1}a_1(\sigma\tau)]a_2(\sigma\tau)^{-1}
\end{eqnarray*}
The three bracketed factors in the second line belong to $S_{1,\tx{sc}}$ and we can rearrange them, obtaining $da_1(\sigma,\tau)^{-1}$. This is an element of $Z^2(\Gamma,Z(G_\tx{sc}))$ and can be pulled in front of the other terms, which themselves produce $da_2(\sigma,\tau)$. However, by Fact \ref{fct:dsqcup} we have
\[ da_1 = dl_kc_k \sqcup n_kp_1 = dl_kc_k \sqcup n_kp_2 = da_2, \]
because the images of $p_1$ and $p_2$ under $P_i^\vee \rw P_i^\vee/Q_i^\vee \rw \tx{Hom}(\mu_n,Z(G_\tx{sc}))$ coincide. This proves the claim that $a_2 \cdot a_1^{-1} \in Z^1(\Gamma,G_\tx{sc}^1)$.
\end{proof}

\begin{thm} \label{thm:tn+g} The isomorphism $\iota$ of Theorem \ref{thm:tn+s} extends to an isomorphism
\[ \iota : \bar Y_{+,\tx{tor}} \rw H^1_\tx{sc}(u \rw W) \]
of functors $\mc{R} \rw \tx{Sets}$, which lifts \eqref{eq:yhomz2}.
\end{thm}

\begin{proof} Let $[Z \rw G] \in \mc{R}$ and let $S \subset G$ be a fundamental maximal torus. According to Corollary \ref{cor:ellsur} the map
\[ \bar Y_{+,\tx{tor}}(Z \rw S) \rw H^1(u \rw W,Z \rw S) \rw H^1(u \rw W,Z \rw G) \]
is surjective and according to Lemma \ref{lem:ygq} it descends to a bijection
\[ \bar Y_{+,\tx{tor}}(Z \rw S)/X_*(S_\tx{sc})_{\Gamma,\tx{tor}} \rw H^1_\tx{sc}(u \rw W,Z \rw G). \]
We claim that
\begin{equation} \label{eq:pizza} \bar Y_{+,\tx{tor}}(Z \rw S)/X_*(S_\tx{sc})_{\Gamma,\tx{tor}} =  \varinjlim \frac{[X_*(S/Z)/X_*(S_\tx{sc})]^N}{I(X_*(S)/X_*(S_\tx{sc}))}, \end{equation}
where the colimit is again taken over all finite Galois extensions $E/F$ splitting $S$.
Indeed, for any maximal torus $S$ and finite Galois extension $E/F$ splitting it we have the exact sequence
\[ \frac{X_*(S_\tx{sc})^N}{IX_*(S_\tx{sc})} \rw \frac{X_*(S/Z)^N}{IX_*(S)} \rw \frac{[X_*(S/Z)/X_*(S_\tx{sc})]^N}{I(X_*(S)/X_*(S_\tx{sc}))} \rw \frac{X_*(S_\tx{sc})^\Gamma}{N(X_*(S_\tx{sc}))}, \]
in which the last map sends an element represented by $x \in X_*(S/Z)$ to $N(x)$. This gives the inclusion $\subset$ in \eqref{eq:pizza}. The reverse inclusion follows from the fact that a fundamental maximal torus of a simply connected semi-simple group has vanishing $H^0_\tx{Tate}$, so the fourth term in the above exact sequence vanishes.

If $S'$ is a second fundamental  maximal torus of $G$ and $g \in G(\ol{F})$ is such that $S'=\tx{Ad}(g)S$, then Lemmas \ref{lem:conj1} and \ref{lem:conj2} imply that we have a commutative diagram
\[ \xymatrix{
\varinjlim \frac{[X_*(S/Z)/X_*(S_\tx{sc})]^N}{I(X_*(S)/X_*(S_\tx{sc}))}\ar[rd]\ar[dd]^{\tx{Ad}(g)}\\
&H^1_\tx{sc}(u \rw W,Z \rw G)\\
\varinjlim \frac{[X_*(S'/Z)/X_*(S'_\tx{sc})]^N}{I(X_*(S')/X_*(S'_\tx{sc}))}\ar[ru]
} \]
By definition of $\bar Y_{+,\tx{tor}}(Z \rw G)$, the diagonal bijections in the above diagram splice to a bijection
\[ \iota_{[Z \rw G]}: \bar Y_{+,\tx{tor}}(Z \rw G) \rw H^1_\tx{sc}(u \rw W,Z \rw G). \]
The fact that $\iota_{[Z \rw S]}$ lifts \eqref{eq:yhomz} implies that $\iota_{[Z \rw G]}$ lifts \eqref{eq:yhomz2}.
Moreover, for any maximal torus $S \subset G$ (fundamental or not), the diagram
\[ \xymatrix{
\bar Y_{+,\tx{tor}}(Z \rw S)\ar^{\iota_{[Z \rw S]}}[d]\ar[r]&\bar Y_{+,\tx{tor}}(Z \rw G)\ar[d]^{\iota_{[Z \rw G]}}\\
H^1(u \rw W,Z \rw S)\ar[r]&H^1_\tx{sc}(u \rw W,Z \rw G)
} \]
commutes. This and the functoriality of $\iota_{[Z \rw S]}$ imply the functoriality of $\iota_{[Z \rw G]}$.
\end{proof}

\section{Applications to the Langlands conjectures} \label{sec:llc}

\subsection{Rigid inner twists} \label{sec:rigid}

We continue to work with a fixed extension $W$ of $\Gamma$ by $u$ belonging to the isomorphism class determined by $\xi$. For any connected reductive group $G$ and finite central subgroup $Z$, both defined over $F$, the sets $Z^1(u \rw W,Z \rw G)$ and $H^1(u \rw W,Z \rw G)$ are then defined. Set $\bar G=G/Z$. The natural projection $G \rw \bar G$ induces maps $Z^1(u \rw W,Z \rw G) \rw Z^1(\Gamma,\bar G) \rw Z^1(\Gamma,G_\tx{ad})$ and $H^1(u \rw W,Z \rw G) \rw H^1(\Gamma,\bar G) \rw H^1(\Gamma,G_\tx{ad})$. If $Z' \subset G$ is another finite central subgroup defined over $F$ and $Z \subset Z'$, we obtain an isogeny $G/Z \rw G/Z'$ as well as natural injective maps $Z^1(u \rw W,Z \rw G) \rw Z^1(u \rw W,Z' \rw G)$ and $H^1(u \rw W,Z \rw G) \rw H^1(u \rw W,Z' \rw G)$.

Recall that an inner twist of $G$ is an isomorphism $\psi : G \rw G'$ defined over $\ol{F}$ between $G$ and a connected reductive group $G'$ defined over $F$ such that for any $\sigma \in \Gamma$ the automorphism $\psi^{-1}\sigma(\psi)$ of $G$ is inner. An isomorphism between two inner twists $\psi_1 : G \rw G_1$ and $\psi_2 : G \rw G_2$ is an isomorphism $f : G_1 \rw G_2$ defined over $F$ and having the property that $\psi_2^{-1}\circ f\circ \psi_1$ is an inner automorphism of $G$. The set of isomorphism classes of inner twists is in bijection with $H^1(\Gamma,G_\tx{ad})$.

By a \emph{rigid inner twist} $(\psi,z) : G \rw G'$ we will mean an inner twist $\psi : G \rw G'$ and an element $z \in Z^1(u \rw W,Z \rw G)$, where $Z$ is any finite central subgroup, having the property that for all $\sigma \in \Gamma$ we have
\[ \psi^{-1}\sigma(\psi) = \tx{Ad}(\bar z(\sigma)), \]
where $\bar z \in Z^1(\Gamma,G_\tx{ad})$ is the image of $z$. We will say that $(\psi,z)$ is \emph{realized by} $Z$ if we want to keep track of $Z$. Given two rigid inner twists $(\psi_1,z_1) : G \rw G_1$ and $(\psi_2,z_2) : G \rw G_2$, with $z_i \in Z^1(u \rw W,Z \rw G)$, an isomorphism $(f,g) : (\psi_1,z_1) \rw (\psi_2,z_2)$ consists of an element $g \in G(\ol{F})$ and an isomorphism $f : G_1 \rw G_2$ defined over $F$, for which the equality $z_2(w)=gz_1(w)w(g^{-1})$ holds in $Z^1(u \rw W,Z \rw G)$ and the diagram
\[ \xymatrix{
G\ar[r]^{\psi_1}\ar[d]_{\tx{Ad}(g)}&G_1\ar[d]^f\\
G\ar[r]^{\psi_2}&G_2
} \]
is commutative. The following fact is obvious but very important.

\begin{fct} \label{fct:autinn} Every automorphism of a rigid inner twist $(\psi,z) : G \rw G'$ is given by an inner automorphism by an element of $G'(F)$.
\end{fct}

We will denote by $RI_Z(G)$ the category, whose objects are rigid inner twists of $G$ realized by $Z$, and whose morphisms are isomorphisms of rigid inner twists. For $Z \subset Z'$, the obvious functor $RI_Z(G) \rw RI_{Z'}(G)$ is fully faithful. We will denote by $RI(G) = \varinjlim RI_Z(G)$ the category of rigid inner twists of $G$. The set of isomorphism classes of $RI_Z(G)$ is $H^1(u \rw W,Z \rw G)$, and the set of isomorphism classes of $RI(G)$ is
\[ H^1(u \rw W,G) := \varinjlim H^1(u \rw W,Z \rw G), \]
the limit being taken over all finite central subgroups $Z$ defined over $F$. We point out that even though the set $H^1(u \rw W,G)$ may appear more natural than $H^1(u \rw W,Z \rw G)$, it is in fact less natural, as it is not functorial in $G$.

A remark is in order on the dependence of the category $RI_Z(G)$ on the extension $W$. Let us temporarily write $RI_Z^W(G)$ to emphasize this dependence. If $W'$ is another extension in the isomorphism class determined by $\xi$, then any isomorphism of extensions $W \rw W'$ determines an equivalence of categories $RI_Z^W(G) \rw RI_Z^{W'}(G)$, and in particular a bijection between their sets of isomorphism classes. Two different isomorphisms $W \rw W'$ will in general produce two different equivalences $RI_Z^W(G) \rw RI_Z^{W'}(G)$, but these equivalences will determine the same bijection on the level of isomorphism classes. Since it is only the isomorphism class of a rigid inner form that matters for applications to endoscopy, the dependence of $RI_Z^W(G)$ on the particular choice of $W$ is for us inessential and we will drop the superscript $W$.

According to Corollary \ref{cor:surjad}, for every inner twist $\psi : G \rw G'$ there exists $z \in Z^1(u \rw W,Z(G_\tx{der}) \rw G)$ such that $(\psi,z)$ is a rigid inner twist. If $G$ is split and $F$ is $p$-adic, then the map $(\psi,z) \mapsto \psi$ sets up a 1-1 correspondence between the set of isomorphism classes of $RI_{Z(G_\tx{der})}$ and and the set of isomorphism classes of (ordinary) inner twists. Thus it appears natural, at least in the $p$-adic case, to fix the finite central subgroup $Z$ to be equal to $Z(G_\tx{der})$. However, the additional flexibility that comes from not fixing $Z$ makes some arguments more transparent, in particular parabolic descent (we refer to the discussion surrounding equation \eqref{eq:tlpinj} for an example). Moreover, as we will see in the next subsection, taking the limit over all $Z$ allows us to reconcile our notion of rigid inner twists for $F=\R$ with the notion of strong real forms defined in \cite{ABV92}. For any fixed $Z$ the category $RI_Z(G)$ has only finitely many isomorphism classes by Fact \ref{fct:h1fin}, so there are only finitely many isomorphism classes of rigid inner twists realized by $Z$ that map to the same isomorphism class of a given (ordinary) inner twist. After taking the limit over all $Z$, this is no longer true and we obtain an infinite set of isomorphism classes of rigid inner twists mapping to a given (ordinary) inner twist. In practice it will often be enough to work with an arbitrary fixed $Z$ as long as one keeps track of how the constructions change upon enlarging $Z$.

A rigid inner twist $(\psi,z)$ is called \emph{pure} if it is realized by $\{1\}$. The pure rigid inner twists are of course just the pure inner twists studied in \cite[\S2]{Kal11} and originally formulated in \cite[Definition 2.6]{Vog93}. In order to accommodate the fact that for $F=\R$ the set $H^1_\tx{sc}(u \rw W,Z \rw G)$ might be a proper quotient of $H^1(u \rw W,Z \rw G)$, we introduce the notion of a \emph{$K$-group of rigid inner twists of $G$}, or simply a \emph{rigid $K$-group}, to mean a set of isomorphism classes of rigid inner twists which comprises one fiber of the map
\[ H^1(u \rw W,Z \rw G) \rw H^1_\tx{sc}(u \rw W,Z \rw G). \]
Note that this is very close to the notion of a $K$-group studied in \cite[\S1]{Art99} and \cite[\S4]{SheTE3}. In fact, if we choose a set of rigid inner twists that represents the isomorphism classes comprising a given rigid $K$-group, then the reductive groups we obtain comprise a $K$-group in the sense of these references. The difference is however that in these references the individual reductive groups are endowed with Galois 1-cocycles valued in their simply connected covers that measure the relative position of one group to another, while in our case the reductive groups are endowed with elements of $Z^1(u \rw W,Z \rw G)$ that measure the absolute position of each reductive group relative to the fixed group $G$. Moreover, it is possible that two non-equivalent rigid $K$-groups give rise to the same $K$-group, that is, the two rigid $K$-groups give disjoint sets of isomorphism classes of rigid inner twists, but give the same set of inner forms of $G$. An example of this phenomenon is the quasi-split unitary group in an odd number of variables with $Z=\{1\}$, where there are precisely two rigid $K$-groups, both of which contain the same inner forms.

Now let $Z \subset G$ be a finite central subgroup of $G$ defined over $F$ and let $(\psi_i,z_i) : G \rw G_i$, $i=1,2$, be two rigid inner twists realized by $Z$, and let $\delta_i \in G_{i,\tx{sr}}(F)$. We say that $(G_1,\psi_1,z_1,\delta_1)$ and $(G_2,\psi_2,z_2,\delta_2)$ are (rationally) conjugate if there exists an isomorphism $(f,g) : (\psi_1,z_1) \rw (\psi_2,z_2)$ with $f(\delta_1)=\delta_2$. We say that $(G_1,\psi_1,z_1,\delta_1)$ and $(G_2,\psi_2,z_2,\delta_2)$ are stably conjugate if $\psi_1^{-1}(\delta_1)$ and $\psi_2^{-1}(\delta_2)$ are $G(\ol{F})$-conjugate. If $G$ is quasi-split, then for any $(G_1,\psi_1,z_1,\delta_1)$ there exists $\delta \in G_\tx{sr}(F)$ such that $(G_1,\psi_1,z_1,\delta_1)$ is stably conjugate to $(G,\tx{id},1,\delta)$. This follows from \cite[Cor 2.2]{Kot82} by taking as $T$ the centralizer of $\delta_1$ in $G_1$ and taking as $i$ the restriction to $T$ of $\psi_1^{-1}$. Now fix $\delta \in G_\tx{sr}(F)$ and consider the category $\mc{C}_Z(\delta)$ whose objects are the tuples $(G_1,\psi_1,z_1,\delta_1)$, with $z_1 \in Z^1(u \rw W,Z \rw G)$, which are stably conjugate to $(G,\tx{id},1,\delta)$, and where the set of morphisms from $(G_1,\psi_1,z_1,\delta_1)$ to $(G_2,\psi_2,z_2,\delta_2)$ is the set of isomorphisms $(f,g) : (\psi_1,z_1) \rw (\psi_2,z_2)$ satisfying $f(\delta_1)=\delta_2$. The category $\mc{C}_Z(\delta)$ can be seen as a generalization of the concept of a stable conjugacy class, and the set of isomorphism classes of $\mc{C}_Z(\delta)$ corresponds to the set of rational classes inside the stable conjugacy class of $\delta$. Set $S=\tx{Cent}(\delta,G)$. For an object $\dot\delta_1=(G_1,\psi_1,z_1,\delta_1) \in \mc{C}_Z(\delta)$ choose $g \in G$ such that $\psi_1(g\delta g^{-1})=\delta_1$. One checks easily that
\[ [w \mapsto g^{-1}z_1(w)w(g)] \in Z^1(u \rw W,Z \rw S), \]
that the class of this element in $H^1(u \rw W,Z \rw S)$ is independent of the choice of $g$, and that it remains unchanged if we replace $(G_1,\psi_1,z_1,\delta_1)$ by an isomorphic object of $\mc{C}_Z(\delta)$. We call this class $\tx{inv}(\delta,\dot\delta_1)$. The usual argument (e.g. \cite[Lemma 2.1.5]{Kal11}) shows that the map
\[ \dot\delta_1 \mapsto \tx{inv}(\delta,\dot\delta_1) \]
sets up a bijection between the set of isomorphism classes in $\mc{C}_Z(\delta)$ and $H^1(u \rw W,Z \rw S)$. Moreover, if we fix a rigid inner twist $(G_1,\psi_1,z_1)$, then this map restricts to a bijection between the set of $G_1(F)$-conjugacy classes of elements $\delta_1 \in G_{1,\tx{sr}}(F)$ which are stably conjugate to $\delta$ and the fiber of $H^1(u \rw W,Z \rw S) \rw H^1(u \rw W,Z \rw G)$ over the class of $z_1$. If $Z' \subset G$ is a finite central subgroup defined over $F$ and containing $Z$, there is an obvious fully-faithful functor $\iota : C_Z(\delta) \rw C_{Z'}(\delta)$, as well as a natural embedding $H^1(u \rw W,Z \rw S) \rw H^1(u \rw W,Z' \rw S)$ and for any $\dot\delta_1 \in C_Z(\delta)$ the image of $\tx{inv}(\delta,\dot\delta_1)$ in $H^1(u \rw W,Z' \rw S)$ coincides with $\tx{inv}(\delta,\iota(\dot\delta_1))$. For this reason, we will not include $Z$ in the notation for $\tx{inv}$ and will also identify $\mc{C}_Z(\delta)$ with its image under $\iota$.

It is useful to rephrase this discussion in an equivalent but slightly different way. We continue to assume that $G$ is quasi-split and consider only rigid inner twists realized by the fixed finite central subgroup $Z$. If we are given a torus $S$ and a $\Gamma$-invariant $G(\ol{F})$-conjugacy class of embeddings $S \rw G$ defined over $\ol{F}$ and having maximal tori of $G$ as their images, we obtain for every rigid inner twist $(G',\psi,z)$ a $\Gamma$-invariant $G'(\ol{F})$-conjugacy class of embeddings $S \rw G'$. Kottwitz's result cited above shows that there exist embeddings $S \rw G$ defined over $F$ and belonging to the given $G(\ol{F})$-conjugacy class. Fix one such $\eta : S \rw G$. For a given rigid inner twist $(G',\psi,z)$ there may or may not exist embeddings of $S$ into $G'$ defined over $F$ and belonging to the given $G'(\ol{F})$-conjugacy class. Say $\eta' : S \rw G'$ is one such, then there exists $g \in G(\ol{F})$ such that $\eta' = \psi\circ\tx{Ad}(g)\circ\eta$. One checks in the same way as above that $g^{-1}z(w)w(g)$ provides a well-defined element $\tx{inv}(\eta,\eta') \in H^1(u \rw W,Z \rw S)$ and that $\eta' \mapsto \tx{inv}(\eta,\eta')$ is a bijection between equivalence classes of embeddings of $S$ into rigid inner twists of $G$ and the group $H^1(u \rw W,Z \rw S)$. Here two embeddings $\eta_1 : S \rw G_1$ and $\eta_2 : S \rw G_2$ of $S$ into the rigid inner twists $(G_i,\psi_i,z_i)$ are equivalent if there exists an isomorphism $(f,g) : (G_1,\psi_1,z_1) \rw (G_2,\psi_2,z_2)$ such that $\eta_2=f\circ\eta_1$. Moreover, for a fixed rigid inner twist $(G',\psi,z)$ the set of embeddings $S \rw G'$, taken up to $G'(F)$-conjugacy, is in bijection with the fiber over the class of $z$ of the map $H^1(u \rw W,Z \rw S) \rw H^1(u \rw W,Z \rw G)$ induced by $\eta$. In particular, such embeddings exist if and only if this fiber is non-empty.

By a representation of a rigid inner twist of $G$ we mean a tuple $(G_1,\psi_1,z_1,\pi_1)$, where $(\psi_1,z_1) : G \rw G_1$ is a rigid inner twist and $\pi_1$ is an admissible representation of $G_1(F)$. By an isomorphism $(f,g): (G_1,\psi_1,z_1,\pi_1) \rw (G_2,\psi_2,z_2,\pi_2)$ we mean an isomorphism $(f,g) : (\psi_1,z_1) \rw (\psi_2,z_2)$ such that the representations $\pi_2\circ f$ and $\pi_1$ are isomorphic (in the real case we take this to mean infinitesimally equivalent). According to Fact \ref{fct:autinn}, two representations $(G_1,\psi_1,z_1,\pi_1)$ and $(G_1,\psi_1,z_1,\pi_1')$ of the same rigid inner twist are isomorphic if and only if $\pi_1$ and $\pi_1'$ are isomorphic in the usual sense as representations of $G_1(F)$. That this is not true if one uses the classical notion of inner twists was observed by Vogan \cite[\S2]{Vog93}. We will write $\Pi^\tx{rig}(G)$ for the set of isomorphism classes of irreducible admissible representations of rigid inner twists of $G$. The subsets $\Pi^\tx{rig}_\tx{unit}(G)$, $\Pi^\tx{rig}_\tx{temp}(G)$ and $\Pi^\tx{rig}_2(G)$ will then be those consisting of unitary, tempered, and essentially square-integrable representations.

Given $\dot\pi \in \Pi^\tx{rig}(G)$, $\dot\pi = (G_1,\psi_1,z_1,\pi_1)$, the Harish-Chandra character $\Theta_{\pi_1}$ of $\pi_1$ is an invariant distribution on $G_1(F)$, represented by a locally integrable function. Any isomorphism of rigid inner twists $(G_1,\psi_1,z_1) \rw (G_2,\psi_2,z_2)$ allows us to transport this distribution to $G_2(F)$. According to Fact \ref{fct:autinn}, the resulting distribution on $G_2(F)$ is independent of the choice of isomorphism. Thus we get a well-defined distribution on the $F$-points of any rigid inner twist that is isomorphic to  $(G_1,\psi_1,z_1)$, and we will use the symbol $\Theta_{\dot\pi}$ to denote this distribution.

\subsection{Comparison between rigid inner twists of real groups and strong real forms} \label{sec:strong}

In this section we consider the ground field $F=\R$. Let $G$ be a connected reductive group defined over $\R$. In Chapter 2 of \cite{ABV92}, Adams, Barbasch, and Vogan, define the notion of a strong real form of $G(\C) \rtimes \Gamma$. This notion is a refinement of the notion of an inner twist of $G$ and the purpose of this section is to compare this notion with the notion of a rigid inner twist of $G$. While the two notions are defined by completely different methods, it turns out (quite surprisingly, as we find) that they are equivalent.

We begin by recalling from \cite[Definition 2.13]{ABV92} that a strong real form of $G(\C) \rtimes \Gamma$ is an element of the coset $G(\C) \rtimes \sigma$ of $G(\C) \rtimes \Gamma$ whose square is a central element of finite order in $G(\C)$. Here $\sigma \in \Gamma$ denotes complex conjugation. Such an element, usually denoted by $\delta$ in loc. cit., leads to the inner form of $G$ with $\R$-points given by
\[ G(\delta,\R) = \{ g \in G(\C)| \delta g \delta^{-1} = g \}. \]
Two strong real forms $\delta,\delta'$ are called equivalent if they are conjugate under the action of $G(\C)$ on the coset $G(\C) \rtimes \sigma$. The set of strong real forms can be given the structure of a small category by setting $\tx{Hom}(\delta,\delta') = \{g \in G(\C)| \delta' = g\delta g^{-1} \}$.

\begin{thm} The category of strong real forms of $G(\C)\rtimes \Gamma$ is equivalent to the category $RI(G)$.
\end{thm}
\begin{proof}

To construct an equivalence we will use the following objects traditionally associated with the fields $\R$ and $\C$, which are also used in \cite{ABV92}. First, we have a preferred primitive fourth root of unity $i \in \C$. This element leads, for any natural number $n$, to the function
\[ k_n(r\cdot e^{i\phi}) = \sqrt[n]{r} \cdot e^\frac{i\phi}{n},\qquad r \in \R_{>0}, \phi \in [0,2\pi). \]
We have $k_n(z)^n=z$ and $k_m(z)^\frac{m}{n}=k_n(z)$ for all $z \in \C^\times$ and $n|m$.
Second, the relative Weil-group $W_{\C/\R}$ has a traditional presentation \cite[Definition 5.2]{ABV92} as the group $\C^\times \boxtimes_c \Gamma$, where $c \in Z^2(\Gamma,\C^\times)$ is the 2-cocycle satisfying $c(\sigma,\sigma)=-1$. In particular, we have a section $s : \Gamma \rw W_{\C/\R}$. We set $E_k=\C$ for all $k$, $n_k = k!$, $l_k = k_{n_k}$, $\zeta_k=\tx{id}$, $s_k=s$. The construction of Section \ref{sec:expw} now gives us an inverse system $W_k=u_k \boxtimes_{\xi_k} \Gamma$ of extensions of $\Gamma$ by $u_k$ whose limit is an extension $W$ of $\Gamma$ by $u$ belonging to the isomorphism class given by $\xi$.

Let $(\psi,z) : G \rw G'$ be a rigid inner twist of $G$. Thus $z \in Z^1(u_k \rw W_k,Z \rw G)$ for some suitable finite subgroup $Z \subset G$ defined over $\R$ and some $k$. We define $\delta_z = z(1 \boxtimes \sigma) \cdot \sigma$ and claim that this is a strong real form. Indeed, we have $\delta_z^2 = z(1 \boxtimes \sigma) \cdot \sigma( z(1 \boxtimes \sigma)) = z(\xi_k(\sigma,\sigma) \boxtimes 1)$, which is an element of $Z$, thus a central element of finite order. Notice that the transition map $W_{k+1} \rw W_k$ is given by $x \boxtimes \sigma \mapsto p(x)\boxtimes\sigma$, because the 1-cochain $\alpha_k$ defined in \eqref{eq:alphake} is trivial in this case. This shows that $\delta_z$ is independent of the choice of $k$. It is clearly independent of the choice of $Z$. The assignment
$(\psi,z) \mapsto \delta_z$
is thus well-defined and will be the equivalence of categories that we seek on the level of objects. On the level of morphisms, this equivalence sends the isomorphism $(f,g) : (\psi,z) \rw (\psi',z')$ to the isomorphism $g \in \tx{Hom}(\delta_z,\delta_{z'})$. We will now show that the resulting functor is indeed an equivalence of categories. It is clear that it is fully faithful, with the inverse of $(f,g) \mapsto g$ being given by $g \mapsto (\psi' \circ\tx{Ad}(g)\circ\psi^{-1},g)$. To show essential surjectivity, let $\delta \in G(\C) \rtimes \sigma$ be a strong real form. Then $\delta^2 \in Z(G)(\C)$ is of finite order, and moreover $\sigma(\delta^2) = \delta \cdot \delta^2 \cdot \delta^{-1} = \delta^2$ shows that it actually belongs to $Z(G)(\R)$. Let $Z \subset G$ be the subgroup generated by $\delta^2$. This is a finite central subgroup of $G$ defined over $\R$. If $n$ is a multiple of $|Z|$, we have the isomorphisms
\[ \tx{Hom}(u_{\C/\R,n},Z)^\Gamma \rw \tx{Hom}_\tx{alg.grp}(\mu_n,Z)^{N_{\C/\R}} \rw Z(\R), \]
the first one being \eqref{eq:homgamma}, and the second one being $\phi \mapsto \phi(e^\frac{2\pi i}{n})$. The resulting system of isomorphisms is compatible with the maps $p: u_{\C/\R,m} \rw u_{\C/\R,n}$ defined by \eqref{eq:defp} for $n|m$ and induces an isomorphism
\[ \tx{Hom}(u,Z)^\Gamma \rw Z(\R). \]
Let $\phi_\delta \in \tx{Hom}(u,Z)^\Gamma$ be the preimage of $\delta^2$ under this isomorphism. We choose $k$ so that $n_k$ is a multiple of $|Z|$ and define $z_{\delta,k} : W_k = u_k \boxtimes_{\xi_k} \Gamma \rw G(\C)$
by $z_{\delta,k}(x \boxtimes 1) = \phi_\delta(x) \in Z$ and $z_{\delta,k}(1 \boxtimes \sigma) = \delta\sigma^{-1} \in G(\C)$. The inflation $z_\delta$ of $z_{\delta,k}$ to $W$ does not depend on $k$ and we claim that it belongs to $Z^1(u \rw W,Z \rw G)$. For this we compute
\[ z_{\delta,k}(1 \boxtimes \sigma) \cdot \sigma(z_{\delta,k}(1 \boxtimes \sigma)) = \delta^2 = \sigma(\delta^2) = \sigma(\phi_\delta(\delta_e(e^\frac{2\pi i}{n_k}))), \]
and recalling the definition \eqref{eq:xidef} of $\xi_k$ we see
\[ \xi_k(\sigma,\sigma) = [dl_kc \cup \delta_e](\sigma,\sigma)= dl_kc(\sigma,\sigma,\sigma)^{\sigma\delta_e}
= (e^\frac{-2\pi i}{n_k})^{\delta_\sigma}=\sigma( (e^\frac{2\pi i}{n_k})^{\delta_e} ). \]
We conclude that $z_{\delta,k}(1 \boxtimes \sigma) \cdot \sigma(z_{\delta,k}(1 \boxtimes \sigma)) = \phi_\delta(\xi_k(\sigma,\sigma))$ and this is enough to establish that $z_\delta \in Z^1(W,G)$. It is then evident from the construction that in fact $z_\delta \in Z^1(u \rw W,Z \rw G)$. Letting $G^\delta$ be the twist of $G$ by the image of $z_\delta$ in $Z^1(\Gamma,G_\tx{ad})$, we have that $(\tx{id},z_\delta) : G \rw G^\delta$ is a rigid inner twist and that our functor maps it to $\delta$.
\end{proof}

\subsection{Refined endoscopic data and transfer factors} \label{sec:refined}

Let $[Z \rw G] \in \mc{R}$ and let $\hat G$ be a complex Langlands dual group for $G$. The isogeny $G \rw \bar G$ dualizes to an isogeny $\hat{\bar G} \rw \hat G$.
Let $\hat Z$ be the kernel of the latter isogeny. We let $Z(\hat{\bar G})^+$ denote the preimage of $Z(\hat G)^\Gamma$ in $Z(\hat{\bar G})$. In this way one obtains a functor
\[ \mc{R} \rw \tx{FinAbGrp},\qquad [Z \rw G] \mapsto \pi_0(Z(\hat{\bar G})^+)^*, \]
where $*$ denotes Pontryagin-dual.
\begin{pro} There is a functorial embedding
\[ \bar Y_{+,\tx{tor}}(Z \rw G) \rw \pi_0(Z(\hat{\bar G})^+)^*. \]
When $G$ is a torus, it is an isomorphism. For general $G$, this embedding is an isomorphism when $F$ is $p$-adic and when $F=\R$ its image consists of those characters of $\pi_0(Z(\hat{\bar G})^+)$ which kill the image of the norm map $N_{\C/R} : Z(\hat{\bar G}) \rw Z(\hat{\bar G})$.
\end{pro}
\begin{proof}
For every maximal torus $S \subset G$ there exists a canonical embedding $Z(\hat G) \rw \hat S$ and analogously a canonical embedding $Z(\hat{\bar G}) \rw \hat{\bar S}$. These embeddings provide the identifications
\[ X^*(Z(\hat G))=\frac{X_*(S)}{X_*(S_\tx{sc})} \qquad\tx{and}\qquad X^*(Z(\hat{\bar G}))=\frac{X_*(\bar S)}{X_*(S_\tx{sc})}. \]
Since $Z(\hat{\bar G})^+$ is the fiber product of $Z(\hat{\bar G})$ with $Z(\hat G)^\Gamma$ over $Z(\hat G)$, we see that
\[ X^*(Z(\hat{\bar G})^+) = \frac{X_*(\bar S)}{IX_*(S)+X_*(S_\tx{sc})} \]
and the elements on the left which kill the connected component of the identity are precisely the torsion elements on the right. On the other hand, for every finite Galois extension $E/F$ splitting $S$ we have the obvious embedding
\[ \frac{[X_*(\bar S)/X_*(S_\tx{sc})]^N}{I(X_*(S)/X_*(S_\tx{sc}))} \subset \left[\frac{X_*(\bar S)}{IX_*(S)+X_*(S_\tx{sc})}\right]_\tx{tor}. \]
The image of this embedding consists of those characters of $Z(\hat{\bar G})^+$ which kill the image of $N : Z(\hat{\bar G}) \rw Z(\hat{\bar G})$. Assume that $F$ is $p$-adic. For varying $E/F$ the above displayed embeddings are compatible and lead to an embedding of the colimit over all $E/F$ of the left hand side into the right hand side. This embedding is in fact surjective, because for a large enough extension $E/F$ the image of $N : Z(\hat{\bar G}) \rw Z(\hat{\bar G})$ is contained in $Z(\hat{\bar G})^{\Gamma,\circ}=Z(\hat{\bar G})^{+,\circ}$. If we drop the assumption that $F$ is $p$-adic and instead assume that $G$ is a torus, then the surjectivity of same embedding is stated in Fact \ref{fct:sbarytor}.
\end{proof}

\begin{cor} \label{cor:tn+pair} There is a pairing
\[ H^1(u \rw W,Z \rw G) \otimes \pi_0(Z(\hat{\bar G})^+) \rw \Q/\Z, \]
which is functorial in $[Z \rw G] \in \mc{R}$. Its left kernel is trivial. If $F$ is $p$-adic or $G$ is a torus, then its right kernel is also trivial, i.e. it is perfect. If $Z=\{1\}$ this pairing coincides with the one defined by Kottwitz in \cite{Kot86}.
\end{cor}
\begin{proof}
Only the last statement requires proof. For this, use that both pairings coincide for tori and are functorial with respect to the inclusion $S \subset G$ of a fundamental maximal torus.
\end{proof}

In the situation that $Z=Z(G_\tx{der})$, we have $\bar G = G_\tx{ad} \times G/G_\tx{der}$ and consequently $\hat{\bar G}=[\hat G]_\tx{sc} \times Z(\hat G)^\circ$. Note that if $G$ is split, this implies $\pi_0(Z(\hat{\bar G})^+)=Z([\hat G]_\tx{sc})$, which is in accordance with Corollary \ref{cor:surjad}.

Our next task is to show how any fixed normalization of the endoscopic transfer factor of Langlands and Shelstad \cite{LS87} for the group $G$ extends to a normalization of the transfer factor for any rigid inner twist of $G$. Before we can address this issue, we need a slight refinement of the notion of endoscopic datum. We begin by recalling that notion, following \cite[\S1.2]{LS87} and \cite[\S2.1]{KS99}. An endoscopic datum for $G$ is a tuple $(H,\mc{H},s,\eta)$, where $H$ is a quasi-split connected reductive group defined over $F$; $\mc{H}$ is a split extension of $W_F$ by $\hat H$ such that the homomorphism $W_F \rw \tx{Out}(\hat H)$ provided by this extension is identified with the homomorphism $W_F \rw \tx{Out}(H)$ provided by the rational structure of $H$ via the natural isomorphism $\tx{Out}(\hat H) \cong \tx{Out}(H)$; $s \in Z(\hat H)$; $\eta : \mc{H} \rw {^LG}$ is an $L$-embedding mapping $\hat H$ isomorphically to $\tx{Cent}(\eta(s),\hat G)^\circ$ and $s \in Z(\hat H)^\Gamma\cdot\eta^{-1}(Z(\hat G))$. An isomorphism from $(H,\mc{H},s,\eta)$ to another such tuple $(H',\mc{H'},s',\eta')$ is an element $g \in \hat G$ satisfying the following two conditions. First, $g\eta(\mc{H})g^{-1}=\eta'(\mc{H'})$. Write $\beta : \hat H \rw \hat H'$ for the isomorphism determined by $\tx{Ad}(g)$. The second condition is that $\beta(s)$ and $s'$ become equal modulo $Z(\hat H')^{\Gamma,\circ}\cdot \eta'^{-1}(Z(\hat G))$.

Given an endoscopic datum $(H,\mc{H},s,\eta)$ of $G$, we may replace it by an equivalent one and assume that $s \in Z(\hat H)^\Gamma$. Furthermore, there is a canonical embedding $Z(G) \rw H$ and we may form $\bar H=H/Z$. The isogeny $H \rw \bar H$ dualizes to an isogeny $\hat{\bar H} \rw \hat H$ and we obtain as before $Z(\hat{\bar H})^+$ which is the preimage of $Z(\hat H)^\Gamma$ under that isogeny.

We now propose the following refinement of the notions of endoscopic data and of an isomorphism of endoscopic data. A \emph{refined endoscopic datum} is a tuple $(H,\mc{H},\dot s,\eta)$ where $H$ and $\mc{H}$ are as before; $\dot s$ is an element of $Z(\hat{\bar H})^+$, whose image in $Z(\hat H)^\Gamma$ we denote by $s$; and $\eta$ is again as before. It is obvious that a refined endoscopic datum $(H,\mc{H},\dot s,\eta)$ gives rise to an (ordinary) endoscopic datum $(H,\mc{H},s,\eta)$ and we argued that up to equivalence every (ordinary) endoscopic datum comes from a refined one. An \emph{isomorphism of refined endoscopic data} between $(H,\mc{H},\dot s,\eta)$ and $(H',\mc{H'},\dot s',\eta')$ is an element $g \in \hat G$ satisfying the following two conditions. First, $g\eta(\mc{H})g^{-1}=\eta'(\mc{H'})$. Write $\beta : \hat H \rw \hat H'$ for the isomorphism determined by $\tx{Ad}(g)$. This isomorphism lifts uniquely to an isomorphism $\bar\beta : \hat{\bar H} \rw \hat{\bar H'}$.
The second condition is that $\bar\beta(\dot s)$ and $\dot s'$ become equal in $\pi_0(Z(\hat{\bar H'})^+)$.

Clearly, an isomorphism of refined endoscopic data induces an isomorphism between the corresponding ordinary endoscopic data. However, the converse is not true -- two non-isomorphic refined endoscopic data may give isomorphic ordinary endoscopic data. Indeed, the requirement imposed on an isomorphism of ordinary endoscopic data is that $\beta(s)$ and $s'$ be equal in a quotient of $\pi_0(Z(\hat H)^\Gamma)$, while the requirement for an isomorphism of refined endoscopic data is that $\bar\beta(\dot s)$ and $\dot s'$ are equal in $\pi_0(Z(\hat{\bar H})^+)$, and the latter surjects onto $\pi_0(Z(\hat H)^\Gamma)$ with finite kernel. It follows that every isomorphism class of ordinary endoscopic data can be refined in only finitely many ways up to isomorphism. We will see that this new notion of an isomorphism is necessary in order for the value of the endoscopic transfer factor we are about to define to be invariant under isomorphisms of endoscopic data. This is related to the problem discovered by Arthur \cite[\S3]{Art06} that an absolute transfer factor for a non-quasi-split connected reductive group will not be invariant under all automorphisms of the (ordinary) endoscopic datum. Our stricter notion of an isomorphism resolves this problem.

We now proceed to show how rigid inner twists provide the means to extend normalizations of transfer factors. We let $G$ be a connected reductive group defined and quasi-split over $F$, $Z$ a finite central subgroup defined over $F$, and $(\psi,z) : G \rw G'$ a rigid inner twist of $G$ realized by $Z$. Let $\mf{\dot e} = (H,\mc{H},\dot s,\eta)$ be a refined endoscopic datum for $G$ and let $\mf{e}=(H,\mc{H},s,\eta)$ be the corresponding ordinary endoscopic datum. Let $\mf{z} = (H_\mf{z},\eta_\mf{z})$ be a z-pair for $\mf{e}$. We recall \cite[\S2.2]{KS99} that $H_\mf{z}$ is an extension of $H$ by an induced torus, and $\eta_\mf{z} : \mc{H} \rw {^LH_\mf{z}}$ is an $L$-embedding extending the embedding $\hat H \rw \hat H_\mf{z}$ dual to the surjection $H_\mf{z} \rw H$. Then this data (without the element $z$) gives rise to relative transfer factors for both $G$ and $G'$, which are functions
\[ \Delta[\mf{e},\mf{z}] : H_{\mf{z},G\tx{-sr}}(F) \times G_\tx{sr}(F) \times H_{\mf{z},G\tx{-sr}}(F) \times G_\tx{sr}(F) \rw \C. \]
\[ \Delta[\mf{e},\mf{z},\psi] : H_{\mf{z},G\tx{-sr}}(F) \times G'_\tx{sr}(F) \times H_{\mf{z},G\tx{-sr}}(F) \times G'_\tx{sr}(F) \rw \C. \]
These functions are constructed in \cite{LS87}. See in particular \cite[\S3.7]{LS87}, where their values are denoted by $\Delta(\gamma_H,\gamma_G;\bar\gamma_H,\bar\gamma_G)$. We have added the notation in brackets to indicate the additional data they depend on. Recall that $G_\tx{sr}$ is the set of strongly-regular semi-simple elements. To explain $H_{\mf{z},G\tx{-sr}}$, recall that to any maximal torus $S \subset G$ corresponds a canonical conjugacy class of embeddings $\hat S \rw \hat G$, called admissible. An isomorphism $i:S^H \rw S$ from a maximal torus of $H$ to a maximal torus of $G$ is called admissible if it is the composition of an admissible embedding $\hat{S^H} \rw \hat H$ with the datum $\eta$ and the inverse of an admissible embedding $\hat S \rw \hat G$. For $h \in S^H$, the element $i(h)$ is called an image of $h$ and the elements $h$ and $i(h)$ are said to be related. The subset $H_{\mf{z},G\tx{-sr}}$ consists of the preimages of those elements of $H$ that are related to strongly-regular semi-simple elements of $G$.

In fact, as pointed out in \cite{KS12}, there are four different ways to normalize the relative transfer factor, and these are called $\Delta, \Delta', \Delta_D, \Delta_D'$. We will work here with $\Delta$, but our discussion applies equally well with trivial modifications to the other three versions. In Section \ref{sec:lpack}, we will use $\Delta'$.

An absolute transfer factor is a function
\[ \Delta[\mf{e},\mf{z}]_\tx{abs} : H_{\mf{z},G\tx{-sr}}(F) \times G_\tx{sr}(F) \rw \C, \]
which is non-zero for any pair $(\gamma_\mf{z},\delta)$ of related elements, and has the property
\[ \Delta[\mf{e},\mf{z}]_\tx{abs}(\gamma_{\mf{z},1},\delta_1)\cdot \Delta[\mf{e},\mf{z}]_\tx{abs}(\gamma_{\mf{z},2},\delta_2)^{-1} = \Delta[\mf{e},\mf{z}](\gamma_{\mf{z},1},\delta_1,\gamma_{\mf{z},2},\delta_2) \]
for any two pairs $(\gamma_{\mf{z},1},\delta_1)$ and $(\gamma_{\mf{z},2},\delta_2)$ of related elements. In contrast to the case of $\Delta[\mf{e},\mf{z}]$, a function $\Delta[\mf{e},\mf{z}]_\tx{abs}$ as above is not unique. The condition imposed on it specifies it only up to multiplication by a complex scalar in the unit circle. A normalization of $\Delta[\mf{e},\mf{z}]_\tx{abs}$ can be specified by either choosing a splitting for $G$ \cite[\S3.7]{LS87}, or by choosing a Whittaker datum for $G$ \cite[\S5.3]{KS99}. Since $G'$ is not quasi-split, these normalizations are not available for the absolute transfer factor $\Delta[\mf{e},\mf{z},\psi]_\tx{abs}$ corresponding to the relative transfer factor $\Delta[\mf{e},\mf{z},\psi]$. However, using the rigid inner twist $(\psi,z)$ and the refinement $\mf{\dot e}$ of $\mf{e}$, we can obtain from any fixed normalization of $\Delta[\mf{e},\mf{z}]_\tx{abs}$ a corresponding normalization of $\Delta[\mf{\dot e},\mf{z},\psi]_\tx{abs}$. We do this following the recipe of \cite[\S2.2]{Kal11} but using the rigid inner twists developed here.

Let $\delta' \in G'_\tx{sr}(F)$ and $\gamma_\mf{z} \in H_\mf{z}(F)$ be related elements. Let $\gamma \in H(F)$ be the image of $\gamma_\mf{z}$ under the map $H_\mf{z} \rw H$. As remarked in Section \ref{sec:rigid}, there exists $\delta \in G_\tx{sr}(F)$ such that $\dot\delta' = (G',\psi,z,\delta') \in \mc{C}_Z(\delta)$. Let $S^H$ be the centralizer of $\gamma$ and $S$ be the centralizer of $\delta$. There exists a unique admissible isomorphism $\phi_{\gamma,\delta} : S^H \rw S$ carrying $\gamma$ to $\delta$. This isomorphism identifies the embedded copies of $Z$ into both tori and induces an isomorphism $\bar\phi_{\gamma,\delta} : \bar S^H \rw \bar S$. Composing the dual of the inverse of this isomorphism with the embedding $Z(\hat{\bar H})^+ \rw [\hat{\bar S^H}]^+$ we obtain from $\dot s$ an element $\dot s_{\gamma,\delta} \in [\hat{\bar S}]^+$.
We put
\begin{equation} \label{eq:tf} \Delta[\mf{\dot e},\mf{z},\psi,z]_\tx{abs}(\gamma_\mf{z},\delta') = \Delta[\mf{e},\mf{z}]_\tx{abs}(\gamma_\mf{z},\delta) \cdot \<\tx{inv}(\delta,\dot\delta'),\dot s_{\gamma,\delta}\>^{-1}. \end{equation}
where the pairing $\<-,-\>$ is the one from Corollary \ref{cor:tn+pair} applied to $G=S$.

Before we prove that this object is a transfer factor, it would be useful to know how its definition depends on the finite central subgroup $Z$. For this, let $Z'$ be a second finite central subgroup of $G$ defined over $F$ and assume $Z \subset Z'$. For the sake of bookkeeping, let us write $z' \in Z^1(u \rw W,Z' \rw G)$ for the image of $z$ under the natural inclusion. We then have the rigid inner twist $(\psi,z')$ The refined endoscopic datum $\mf{\dot e}$ doesn't serve $(\psi,z')$ any more, because the set $Z(\hat{\bar H})^+$ to which $\dot s$ belongs was constructed with respect to $Z$. Let us write, for a lack of better notation, $Z(\hat{\bar{\bar H}})^+$ for the corresponding set defined with respect to $Z'$. By definition this set surjects onto $Z(\hat{\bar H})^+$ and we may choose a preimage $\ddot s \in Z(\hat{\bar{\bar H}})^+$ of $\dot s$. Now $\mf{\ddot e} = (H,\mc{H},\ddot s,\eta)$ is a refined endoscopic datum that serves $(\psi,z')$. We now define $\Delta[\mf{\ddot e},\mf{z},\psi,z']_\tx{abs}(\gamma_\mf{z},\iota(\delta'))$ by the same formula as above, but using the element $\ddot s$ instead of $\dot s$ and $\iota(\dot\delta') \in \mc{C}_{Z'}(\delta)$ instead of $\dot\delta' \in \mc{C}_Z(\delta)$. We had remarked in Section \ref{sec:rigid} that $\tx{inv}(\delta,\dot\delta') \in H^1(u \rw W,Z \rw G)$ maps to $\tx{inv}(\delta,\iota(\dot\delta')) \in H^1(u \rw W,Z' \rw G)$. The functoriality of the pairing from Corollary \ref{cor:tn+pair} now implies the following fact.

\begin{fct} \label{fct:tfcz} We have the equality
\[ \Delta[\mf{\dot e},\mf{z},\psi,z]_\tx{abs}(\gamma_\mf{z},\delta') = \Delta[\mf{\ddot e},\mf{z},\psi,z']_\tx{abs}(\gamma_\mf{z},\iota(\delta')). \]
\end{fct}

Having handled passage to a larger $Z$, we can now focus again on a particular fixed $Z$.

\begin{pro} \label{pro:tf} The value of $\Delta[\mf{\dot e},\mf{z},\psi,z]_\tx{abs}(\gamma_\mf{z},\delta')$ does not depend on the choice of $\delta$, and the function $\Delta[\mf{\dot e},\mf{z},\psi,z]_\tx{abs}$ is an absolute transfer factor. Moreover, the function $\Delta[\mf{\dot e},\mf{z},\psi,z]_\tx{abs}$ does not change if we replace $\mf{\dot e}$ by an isomorphic refined endoscopic datum, or if we replace $(\psi,z)$ by an isomorphic rigid inner twist realized by $Z$.
\end{pro}
\begin{proof}
We begin by showing independence of $\delta$. Let $\delta_0 \in G_\tx{sr}(F)$ be another element stably conjugate to $\delta'$. Then $\delta$ and $\delta_0$ are also stably-conjugate, and if $g_0 \in G(\ol{F})$ is such that $\tx{Ad}(g_0)\delta_0=\delta$, then we have \cite[\S3.4]{LS87}
\[ \Delta[\mf{e},\mf{z}]_\tx{abs}(\gamma_\mf{z},\delta) = \Delta[\mf{e},\mf{z}]_\tx{abs}(\gamma_\mf{z},\delta_0) \cdot \<\tx{inv}(\delta_0,\delta),s_{\gamma,\delta_0}\>^{-1}. \]
Here $\tx{inv}(\delta_0,\delta)$ is the class of $\sigma \mapsto g_0^{-1}\sigma(g_0)$ in $H^1(\Gamma,S_0)$, with $S_0$ the centralizer of $\delta_0$, and $s_{\gamma,\delta_0} \in \hat S_0^\Gamma$ is the image of $s$ under the composition of the inclusion $Z(\hat H)^\Gamma \rw [\hat S^H]^\Gamma$ and the inverse of the dual of the admissible isomorphism $\phi_{\gamma,\delta_0} : S^H \rw S_0$ carrying $\gamma$ to $\delta_0$. The pairing $\<\>$ is the usual Tate-Nakayama pairing. By Theorem \ref{thm:tn+s}, we have
\[ \<\tx{inv}(\delta_0,\delta),s_{\gamma,\delta_0}\> = \<\tx{inv}(\delta_0,\delta),\dot s_{\gamma,\delta_0}\>, \]
where on the right we are using the pairing of Corollary \ref{cor:tn+pair}. Now consider
$\<\tx{inv}(\delta,\dot\delta'),\dot s_{\gamma,\delta}\>$.
The functoriality of the pairing implies that if we transport $\tx{inv}(\delta,\dot\delta')$ to $H^1(u \rw W,Z \rw S_0)$ and $\dot s_{\gamma,\delta}$ to $[\hat{\bar S_0}]^+$ using the isomorphism $\tx{Ad}(g_0) : S_0 \rw S$ and pair them afterwards, we will obtain the same result. The image of $\dot s_{\gamma,\delta}$ in $[\hat{\bar S_0}]^+$ is equal to $\dot s_{\gamma,\delta_0}$. If $g \in G(\ol{F})$ is the element with $\delta'=\psi(g\delta g^{-1})$, then $\tx{inv}(\delta,\dot\delta')$ is represented by the 1-cocycle $w \mapsto g^{-1}z(w)w(g)$ and its image in $H^1(u \rw W,Z \rw S_0)$ is represented by the 1-cocycle $w \mapsto g_0^{-1}g^{-1}z(w)w(g)g_0$. It follows that the product
$\<\tx{inv}(\delta_0,\delta),s_{\gamma,\delta_0}\> \cdot \<\tx{inv}(\delta,\dot\delta'),\dot s_{\gamma,\delta}\>$
is equal to the pairing of $\dot s_{\gamma,\delta_0}$ with the product
\[ g_0^{-1}g^{-1}z(w)w(g)g_0 \cdot g_0^{-1}w(g_0) \]
and this represents $\tx{inv}(\delta_0,\dot\delta')$. We conclude
\[ \<\tx{inv}(\delta_0,\delta),s_{\gamma,\delta_0}\> \cdot \<\tx{inv}(\delta,\dot\delta'),\dot s_{\gamma,\delta}\> = \<\tx{inv}(\delta_0,\dot\delta'),\dot s_{\gamma,\delta_0} \> \]
and this completes the proof of independence of the choice of $\delta$.

The invariance under isomorphisms of rigid inner twists follows immediately from the fact that $\tx{inv}(\delta,\dot\delta')$ depends only on the isomorphism class of $\dot\delta'$ in $\mc{C}_Z(\delta)$. For the invariance under isomorphisms of the refined endoscopic datum, say $g \in \hat G$ is an isomorphism from $\mf{\dot e}=(H,\mc{H},\dot s,\eta)$ to $\mf{\dot e'}=(H',\mc{H'},\dot s',\eta')$. Write $\hat\beta : \hat H \rw \hat H'$ for the isomorphism induced by $\tx{Ad}(g)$, and write $\beta : H' \rw H$ for an isomorphism defined over $F$ and dual to $\hat\beta$ ($\beta$ is determined up to $H(F)$-conjugacy, and the particular choice is irrelevant). We let $H_\mf{z}'$ be the extension of $H'$ obtained by composing the map $H_\mf{z} \rw H$ with with $\beta^{-1}$, and set
\[ \mf{z'}=(H_\mf{z}',\eta_\mf{z}\circ\eta^{-1}\circ\tx{Ad}(g^{-1})\circ\eta'). \]
The second component of $\mf{z'}$ make sense, because $\eta$ is injective and its image is equal to the image of $\tx{Ad}(g^{-1})\circ\eta'$. Now $\mf{z'}$ is a $z$-pair for $\mf{e'}$ and we want to show that for any $\gamma_\mf{z} \in H_\mf{z}'(F)=H_\mf{z}(F)$ and $\delta' \in G'_\tx{sr}(F)$ we have
\[ \Delta_\tx{abs}[\mf{\dot e'},\mf{z'},\psi,z](\gamma_\mf{z},\delta') = \Delta[\mf{\dot e},\mf{z},\psi,z]_\tx{abs}(\gamma_\mf{z},\delta'). \]
For this we choose $\delta \in G_\tx{sr}(F)$ stably-conjugate to $\delta'$. We first claim that we have the equality
\begin{equation} \label{eq:temp1} \Delta[\mf{e'},\mf{z'}]_\tx{abs}(\gamma_\mf{z},\delta) = \Delta[\mf{e},\mf{z}]_\tx{abs}(\gamma_\mf{z},\delta). \end{equation}
Since all absolute transfer factors for $\mf{e}$ and $\mf{z}$ differ from each other by a scalar, it is enough to check this equation for a specific normalization. We choose a pinning of $G$ and discuss the normalization $\Delta_0$ of \cite{LS87}. After fixing $a$-data and $\chi$-data for $S=\tx{Cent}(\delta,G)$ and using the admissible isomorphism $\phi_{\gamma,\delta} : S^H \rw S$, the factor $\Delta_0[\mf{e},\mf{z}]$ can be written as a product of terms $\Delta_I$, $\Delta_{II}$, $\Delta_{III_2}$, $\Delta_{IV}$, with $\Delta_{III_1}=1$ due to our choice of $\phi_{\gamma,\delta}$. We use the same $a$-data and $\chi$-data to form $\Delta_0[\mf{e'},\mf{z'}]$. We must however use a different admissible isomorphism, because the endoscopic group $H'$ underlying $\mf{e'}$ is different from $H$. Let $\gamma' \in H'(F)$ be the image of $\gamma_\mf{z}$ under the projection $H_\mf{z}(F) \rw H'(F)$. By construction we have $\beta(\gamma')=\gamma$ and hence $\phi_{\gamma',\delta}=\phi_{\gamma,\delta}\circ\beta$. This is the admissible isomorphism we will use, and now we can compare the individual terms. The terms $\Delta_{II}$ and $\Delta_{IV}$ depend only on $\delta$ and the chosen $a$-data and $\chi$-data for $S$, so they do not notice the passage from $\mf{e},\mf{z}$ to $\mf{e'},\mf{z'}$. The term $\Delta_I[\mf{e},\mf{z}]$ depends on the $a$-data, the pinning of $G$, as well as $s_{\gamma,\delta}$. But one observes immediately that $s'_{\gamma',\delta}=s_{\gamma,\delta}$ as elements of $\pi_0(\hat S^\Gamma)$ and thus passing to $\mf{e'},\mf{z'}$ preserves $\Delta_I$. For the term $\Delta_{III_2}$ one considers the diagram
\[
\xymatrix{
^LH_\mf{z}&^LS^{H_\mf{z}}\ar@{_(->}[l]&^LS^{H_\mf{z}}\ar@{..>}[l]&^LS^H\ar@{_(->}[l]\\
\mc{H}\ar@{^(->}[u]^{\eta_\mf{z}}\ar@{^(->}[r]^{\eta}&^LG&&^LS\ar@{_(->}[ll]\ar[u]_{^L\phi_{\gamma,\delta}}
} \]
Here $^LS \rw {^LG}$ comes from the chosen $\chi$-data. This $\chi$-data is transported to $S^H$ via $\phi_{\gamma,\delta}$ and then to $S^{H_\mf{z}}$ via the projection $S^{H_\mf{z}} \rw S^H$ and leads to ${^LS^{H_\mf{z}}} \rw {^LH_\mf{z}}$. The dotted arrow is the unique $L$-automorphism extending the identity on $\hat S^{H_\mf{z}}$ and making the diagram commute. Its restriction to $W_F$ is a Langlands parameter $a : W_F \rw {^LS^{H_\mf{z}}}$ and $\Delta_{III_2}=\<a,\gamma_\mf{z}\>$, where $\<-,-\>$ is the Langlands duality pairing for tori. Passing to $\mf{e'},\mf{z'}$ only the two corners $\mc{H}$ and $^LS^H$ of this diagram change. In both cases, the change is balanced out by passing from $\eta_\mf{z}$ to $\eta_\mf{z'}=\eta_\mf{z}\circ\eta^{-1}\circ\tx{Ad}(g^{-1}\circ\eta'$ and from $^L\phi_{\gamma,\delta}$ to $^L\phi_{\gamma',\delta}={^L\beta^{-1}}\circ{^L\phi_{\gamma,\delta}}$. The dotted arrow in the diagram, as well as $\Delta_{III_2}$, remain unchanged. This completes the proof of \eqref{eq:temp1}. It remains to check that
\[ \<\tx{inv}(\delta,\dot\delta'),\dot s'_{\gamma',\delta}\> = \<\tx{inv}(\delta,\dot\delta'),\dot s_{\gamma,\delta} \>, \]
which follows immediately from the equation $\dot s'_{\gamma',\delta}=\dot s_{\gamma,\delta} \in \pi_0([\hat{\bar S}]^+)$.

We now proceed to show that $\Delta_\tx{abs}[\mf{\dot e},\mf{z},\psi,z]$ is an absolute transfer factor. This amounts to showing that for two pairs $(\gamma_{\mf{z},1},\delta_1')$ and $(\gamma_{\mf{z},2},\delta_2')$ of related elements we have
\[ \Delta[\mf{\dot e},\mf{z},\psi,z]_\tx{abs}(\gamma_{\mf{z},1},\delta_1') \cdot \Delta[\mf{\dot e},\mf{z},\psi,z]_\tx{abs}(\gamma_{\mf{z},2},\delta_2')^{-1} = \Delta[\mf{e},\mf{z},\psi](\gamma_{\mf{z},1},\delta_1',\gamma_{\mf{z},2},\delta_2'). \]
At this point it will be convenient to enlarge $Z$. Fact \ref{fct:tfcz} allows us to do so, after replacing $\mf{\dot e}$ by an appropriate refined endoscopic datum for the new $Z$ in the same way we constructed $\mf{\ddot e}$. Thus we assume that $Z$ contains $Z(G_\tx{der})$.

We return to the proof of above equation. Following the construction of the left hand side, we choose $\delta_1,\delta_2 \in G(F)$ such that $\delta_1$ is stably conjugate to $\delta_1'$ and $\delta_2$ is stably conjugate to $\delta_2'$ and have to show that
\[ \frac{\<\tx{inv}(\delta_1,\dot\delta_1'),\dot s_{\gamma_1,\delta_1}\>^{-1}}{\<\tx{inv}(\delta_2,\dot\delta_2'),\dot s_{\gamma_2,\delta_2}\>^{-1}} =
\frac{\Delta[\mf{e},\mf{z},\psi](\gamma_{\mf{z},1},\delta_1',\gamma_{\mf{z},2},\delta_2')}{\Delta[\mf{e},\mf{z}](\gamma_{\mf{z},1},\delta_1,\gamma_{\mf{z},2},\delta_2)},
\]
where $\gamma_i$ is the image in $H(F)$ of $\gamma_{\mf{z},i}$.
Applying \cite[Lemma 4.2.A]{LS87}, we see that the right hand side is equal to
\[ \left\< \tx{inv}\left(\frac{\gamma_1,\delta_1'}{\gamma_2,\delta_2'}\right),s_U\right\>. \]
This object is constructed in \cite[\S3.4]{LS87} and we will now recall its construction. We let $S_i^H$ be the centralizer of $\gamma_i$ in $H$, $S_i$ the centralizer of $\delta_i$ in $G$, and $S_i'$ the centralizer of $\delta_i'$ in $G'$. We fix $g_i \in G_\tx{sc}$ with $\psi(g_i\delta_i g_i^{-1})=\delta_i'$. We fix an arbitrary 1-cochain $u : \Gamma \rw G_\tx{sc}$ with the property $\psi^{-1}\sigma(\psi)=\tx{Ad}(u_\sigma)$, and form the 1-cochains $v_i(\sigma)=g_i^{-1}u(\sigma)\sigma(g_i)$. Then $v_i$ takes values in $S_{i,\tx{sc}}$ and its differential takes values in $Z(G_\tx{sc})$. Moreover, $dv_1=dv_2$, so if we form the torus $U=S_{1,\tx{sc}} \times S_{2,\tx{sc}} /Z(G_\tx{sc})$, where $Z(G_\tx{sc})$ is embedded into the product by $z \mapsto (z,z^{-1})$, the 1-cochain $\Gamma \rw U, \sigma \mapsto (v_1(\sigma)^{-1},v_2(\sigma))$ is in fact a 1-cocycle and is independent of the choice of $u$. We denote its class by
\[ \tx{inv}\left(\frac{\gamma_1,\delta_1'}{\gamma_2,\delta_2'}\right). \]
To explain $s_U$, consider the admissible isomorphisms $\phi_{\gamma_i,\delta_i} : S_i^H \rw S_i$. Let $s_{\gamma_i,\delta_i} \in \hat S_i^\Gamma$ be the image of $s$ under the composition of the natural inclusion $Z(\hat H)^\Gamma \rw \hat S_i^H$ with $\hat\phi_{\gamma_i,\delta_i}^{-1}$. For any $h \in H(\ol{F})$ with $\tx{Ad}(h)S^H_2 = S^H_1$, the (not necessarily $\Gamma$-equivariant) admissible isomorphism
\[ \hat\phi = \hat\phi_{\gamma_2,\delta_2}^{-1}\circ\hat{\tx{Ad}}(h)\circ\hat\phi_{\gamma_1,\delta_1} : \hat S_1 \rw \hat S_2 \]
intertwines the embeddings of $Z(\hat H)$ into both tori. The isomorphism $\hat\phi$ lifts to an isomorphism $\hat\phi : [\hat S_1]_\tx{sc} \rw [\hat S_2]_\tx{sc}$ and induces an isomorphism
\begin{equation} \label{eq:isoszh} [\hat S_1]_\tx{sc} \times_{\hat S_1} Z(\hat H) \rw [\hat S_2]_\tx{sc} \times_{\hat S_2} Z(\hat H) \end{equation}
which is independent of the choice of $h$, and hence $\Gamma$-equivariant. We choose an arbitrary lift $\tilde s_{\gamma_1,\delta_1} \in [\hat S_1]_\tx{sc} \times_{\hat S_1} Z(\hat H) \subset [\hat S_1]_\tx{sc}$ of the image of $s_{\gamma_1,\delta_1}$ in $[\hat S_1]_\tx{ad}^\Gamma$ and set $\tilde s_{\gamma_2,\delta_2} = \hat\phi(\tilde s_{\gamma_1,\delta_1}) \in [\hat S_2]_\tx{sc}$.
Since $\hat\phi$ is $\Gamma$-equivariant on $[\hat S_1]_\tx{sc} \times_{\hat S_1} Z(\hat H)$, it identifies the differentials of $\tilde s_{\gamma_i,\delta_i} \in C^0(\Gamma,[\hat S_i]_\tx{sc})$. These differentials however belong to $Z(\hat G_\tx{sc})$ and we therefore conclude that they are equal. It follows that the image $s_U$ of $(\tilde s_{\gamma_1,\delta_1},\tilde s_{\gamma_2,\delta_2})$ in the quotient $\hat U$ of $[\hat S_1]_\tx{sc} \times [\hat S_2]_\tx{sc}$ by the image of $Z(\hat G_\tx{sc})$ under the diagonal embedding belongs to $\hat U^\Gamma$. One checks that $\hat U$ is the dual torus to $U$ and thus the element $s_U$ can be paired with the element $\tx{inv}(\gamma_1,\delta_1'/\gamma_2,\delta_2')$ using Tate-Nakayama duality.

Let us now prove the equality
\begin{equation} \label{eq:d1} \frac{\<\tx{inv}(\delta_1,\delta_1'),\dot s_{\gamma_1,\delta_1}\>^{-1}}{\<\tx{inv}(\delta_2,\delta_2'),\dot s_{\gamma_2,\delta_2}\>^{-1}} =
\left\< \tx{inv}\left(\frac{\gamma_1,\delta_1'}{\gamma_2,\delta_2'}\right),s_U\right\>.
\end{equation}
We begin by forming the torus $V=S_1 \times S_2/Z(G)$, where $Z(G)$ is embedded into $S_1 \times S_2$ via $z \mapsto (z,z^{-1})$. The homomorphism $S_1 \times S_2 \rw V$ defines the morphism $[Z \times Z \rw S_1 \times S_2] \rw [Z \times Z/Z \rw V]$ in the category $\mc{T}$. We claim that the image in $H^1(u \rw W,Z \times Z/Z \rw V)$ of the element
\[ (\tx{inv}(\delta_1,\delta_1')^{-1},\tx{inv}(\delta_2,\delta_2')) \in H^1(u \rw W,Z\times Z \rw S_1 \times S_2)\]
under that morphism belongs to $H^1(\Gamma,V)$. Indeed, the restriction maps $H^1(u \rw W,Z \rw S_i) \rw \tx{Hom}(u,Z)^\Gamma$ factor as the composition of $H^1(u \rw W,Z \rw S_i) \rw H^1(u \rw W,Z \rw G)$ and the restriction $H^1(u \rw W,Z \rw G) \rw \tx{Hom}(u,Z)^\Gamma$, and the images of $\tx{inv}(\delta_i,\delta_i')$ in $H^1(u \rw W,Z \rw G)$ are both equal to the class of the rigidifying element $z$, so the claim follows from the exact sequence \eqref{eq:infres}.

We also have a homomorphism $U \rw V$ and we claim that the image of the element $\tx{inv}(\gamma_1,\delta_1'/\gamma_2,\delta_2') \in H^1(\Gamma,U)$ in $H^1(\Gamma,V)$ is equal to the image of the element $(\tx{inv}(\delta_1,\delta_1')^{-1},\tx{inv}(\delta_2,\delta_2'))$.
To see this, fix a section $s: \Gamma \rw W$. From the rigidifying element $z \in Z^1(u \rw W,Z \rw G)$ we obtain the 1-cochain $z\circ s : \Gamma \rw G$ and decompose it as $z(s(\sigma)) = \bar u(\sigma) \cdot x(\sigma)$ with $\bar u(\sigma) \in G_\tx{der}$ and $x(\sigma) \in Z(G)$. Fix a lift $u(\sigma) \in G_\tx{sc}$ of $\bar u(\sigma)$. By construction then
\[ \tx{inv}\left(\frac{\gamma_1,\delta_1'}{\gamma_2,\delta_2'}\right)=((g_1^{-1}u(\sigma)\sigma(g))^{-1},(g_2^{-1}u(\sigma)\sigma(g))). \]
The image in $V$ of this 1-cochain equals $((g_1^{-1}z(s(\sigma))\sigma(g))^{-1},(g_2^{-1}z(s(\sigma))\sigma(g)))$, which is indeed the image of
$(\tx{inv}(\delta_1,\delta_1')^{-1},\tx{inv}(\delta_2,\delta_2'))$, as claimed.

Since the pairing of Corollary \ref{cor:tn+pair} is functorial and extends the Tate-Nakayama pairing for tori, the equality \eqref{eq:d1} will be established once we produce an element of $[\hat{\bar V}]^+$ whose image in $[\hat{\bar S_1}]^+ \times [\hat{\bar S_2}]^+$ is equal to $(\dot s_{\gamma_1,\delta_1},\dot s_{\gamma_2,\delta_2})$ and whose image in $[\hat{\bar U}]^+$ maps to $s_U$ under the isogeny $[\hat{\bar U}]^+ \rw \hat U^\Gamma$. Here we have formed $\bar U$ from the object $[Z(G_\tx{sc}) \rw U] \in \mc{T}$.

We must first understand the tori $\hat V$ and $\hat{\bar V}$. For this we apply a very useful technique from the papers \cite{LS90} and \cite{KS99}: The admissible isomorphism $\phi : S_2 \rw S_1$ dual to $\hat\phi$ is defined over $\ol{F}$ and using it we can form the map
\[ f: S_1 \times S_2 \rw S_1 \times S_2,\qquad (s_1,s_2) \mapsto (s_1\phi(s_2),s_2). \]
This is an isomorphism of algebraic tori defined over $\ol{F}$. The transport of the $\Gamma$-action from its source to its target is given by
\[ f\circ\sigma\circ f^{-1}(t_1,t_2) = (\sigma t_1,\sigma t_2)\cdot \left(\frac{\phi(\sigma t_2)}{\sigma\phi(t_2)},1\right). \]
We obtain the commutative diagram
\begin{equation} \label{eq:fdiag} \xymatrix{
U\ar[r]\ar[d]&V\ar[d]&\ar[l]\ar[d]S_1 \times S_2\\
S_{1,\tx{sc}} \times S_{2,\tx{ad}}\ar[r]&S_1 \times S_{2,\tx{ad}}&\ar[l]S_1 \times S_2
} \end{equation}
where all vertical arrows are isomorphisms induced by $f$, and all horizontal arrows are the natural componentwise arrows. The dual diagram is then
\begin{equation} \label{eq:hatfdiag} \xymatrix{
\hat U&\ar[l]\hat V\ar[r]&\hat S_1 \times \hat S_2\\
[\hat S_1]_\tx{ad} \times [\hat S_2]_\tx{sc}\ar[u]&\ar[l]\hat S_1 \times [\hat S_2]_\tx{sc}\ar[u]\ar[r]&\hat S_1 \times \hat S_2\ar[u]_{\hat f}
} \end{equation}
The left vertical isomorphism sends $(\bar s_1,s_2) \in [\hat S_1]_\tx{ad} \times [\hat S_2]_\tx{sc}$ to $(s_1,\hat\phi(s_1)s_2) \in [\hat S_1]_\tx{sc} \times [\hat S_2]_\tx{sc}/Z(\hat G_\tx{sc})=\hat U$, where $s_1 \in [\hat S_1]_\tx{sc}$ is any lift of $\bar s_1$. The right isomorphism is given by $(s_1,s_2) \mapsto (s_1,\hat\phi(s_1)s_2)$ and transports the $\Gamma$-action on its target to the twisted $\Gamma$-action on its source given by $\hat f^{-1}\circ\sigma\circ\hat f(s_1,s_2) = (\sigma s_1,\frac{\sigma\hat\phi(s_1)}{\hat\phi(\sigma s_1 )}\sigma s_2 )$.

Under the middle vertical isomorphism of \eqref{eq:fdiag}, $Z \times Z/Z \rw V$ is identified with $Z \rw S_1 \times S_{2,\tx{ad}}$ where $Z$ is embedded into $S_1$. Our assumption $Z(G_\tx{der}) \subset Z$ implies that $\bar S_1 = S_1/Z = S_{1,\tx{ad}} \times [Z(G)/Z]$. Dually we have $\hat{\bar S_1} = [\hat S_1]_\tx{sc} \times \hat C$, where $\hat C$ is an algebraic torus that is a finite cover of $Z(\hat G)^\circ$. Thus we obtain the commutative diagram
\[ \xymatrix{
\hat{\bar U}&\ar[l]\hat{\bar V}\ar[r]&\hat{\bar S_1} \times \hat{\bar S_2}\\
[\hat S_1]_\tx{sc} \times [\hat S_2]_\tx{sc}\ar[u]&\ar[l][\hat S_1]_\tx{sc} \times \hat C \times [\hat S_2]_\tx{sc}\ar[r]\ar[u]&[\hat S_1]_\tx{sc} \times \hat C \times [\hat S_2]_\tx{sc} \times \hat C\ar[u]
} \]
where the vertical arrows are isomorphisms and the bottom arrows are given by $(s_1,s_2) \mapsfrom (s_1,z,s_2)$ and $(s_1,z,s_2) \mapsto (s_1,z,s_2,1)$.

We will now show that there exists $\dot s_V \in [\hat{\bar V}]^+$ which maps to $(\dot s_{\gamma_1,\delta_1},\dot s_{\gamma_2,\delta_2}) \in [\hat{\bar S_1}]^+ \times [\hat{\bar S_2}]^+$ and to $s_U$ in $\hat U^\Gamma$. The right vertical isomorphism sends $(s_1,z_1,s_2,z_2)$ to $(s_1z_1,\hat\phi(s_1z_1)s_2z_2)$. By construction, we have $\dot s_{\gamma_2,\delta_2} = \hat\phi(\dot s_{\gamma_1,\delta_1})$ and hence the preimage of $(\dot s_{\gamma_1,\delta_1},\dot s_{\gamma_2,\delta_2})$ under that isomorphism is equal to $(s_1,z_1,1,1)$, with $s_1z_1 = \dot s_{\gamma_1,\delta_1}$ being the decomposition according to $\hat{\bar S_1} = [\hat S_1]_\tx{sc} \times \hat C$. Thus if we let $\dot s_V$ be the image of $(s_1,z_1,1)$ under the middle vertical isomorphism, then $\dot s_V$ does indeed map to $(\dot s_{\gamma_1,\delta_1},\dot s_{\gamma_2,\delta_2})$. To show that $\dot s_V \in [\hat{\bar V}]^+$, we must argue that $(s_{\gamma_1,\delta_1},1) \in \hat S_1 \times [\hat S_2]_\tx{sc}$ is fixed under the twisted action of $\Gamma$ for which the middle vertical isomorphism of \eqref{eq:hatfdiag} is $\Gamma$-equivariant. This follows from the equality $\sigma\hat\phi(s_{\gamma_1,\delta_1}) = \sigma s_{\gamma_2,\delta_2} = s_{\gamma_2,\delta_2} = \hat\phi(s_{\gamma_1,\delta_1}) = \hat\phi(\sigma s_{\gamma_1,\delta_1})$ for all $\sigma \in \Gamma$. Finally, since  $s_1 \in [\hat S_1]_\tx{sc}$ is a preimage of the image of $s_{\gamma_1,\delta_1}$ in $[\hat S_1]_\tx{ad}$, we see that $\dot s_V$ maps to $s_U \in \hat U^\Gamma$.
\end{proof}

\subsection{Structure of tempered $L$-packets and endoscopic transfer} \label{sec:lpack}

The discussion of the previous section suggests a way to phrase the conjectural local Langlands correspondence for an arbitrary connected reductive group defined over a local field $F$ of characteristic zero in an unambiguous way. Let $G$ again be a connected reductive algebraic group defined and quasi-split over $F$, and let $Z$ be a finite central subgroup defined over $F$. Fix a Whittaker datum $\mf{w}$ for $G$. Given a tempered Langlands parameter $\varphi : W_F' \rw {^LG}$, we write $S_\varphi = \tx{Cent}(\varphi,\hat G)$ and let $S_\varphi^+$ be the preimage of $S_\varphi$ in $\hat{\bar G}$. Then $Z(\hat{\bar G})^+ \subset S_\varphi^+$. We expect that there is a finite subset $\Pi_\varphi \subset \Pi^\tx{rig}_\tx{temp}(G)$ and a commutative diagram
\begin{equation} \label{eq:llc+} \xymatrix{
\Pi_\varphi\ar[rr]^-{\iota_\mf{w}}\ar[d]&&\tx{Irr}(\pi_0(S_\varphi^+))\ar[d]\\
H^1(u \rw W,Z \rw G)\ar[rr]&&\pi_0(Z(\hat{\bar G})^+)^*
} \end{equation}
in which the top arrow is a bijection, the bottom arrow is given by the pairing of Corollary \ref{cor:tn+pair}, the right arrow assigns to each irreducible representation (the restriction of) its central character, and the left arrow sends a quadruple $(G',\psi,z,\pi')$ to the class of $z$. In accordance with Shahidi's tempered $L$-packet conjecture \cite[\S9]{Sh90} we expect that $\Pi_\varphi$ contains a unique element $(G,\tx{id},1,\pi)$ such that $\pi$ is $\mf{w}$-generic and that the top arrow identifies the trivial representation of $\pi_0(S_\varphi^+)$ with that constituent.

We alert the reader familiar with \cite{Art06} that the group $\pi_0(S_\phi^+)$ is in general different from Arthur's group $\mc{\tilde S}_\phi$ appearing in \cite[\S3]{Art06}. While there are some situations in which the two groups coincide, for example when $G$ is simply-connected and $Z=Z(G)$, they are in general different. Nonetheless, there is a direct relationship between the conjecture proposed here and Arthur's local conjecture of \cite{Art06}, as discussed in \cite[\S4.6]{KalGRI}.

If we fix a rigid inner twist $(\psi,z) : G \rw G'$ realized by $Z$, the elements of $\Pi_\varphi$ over the class of $z$ constitute the $L$-packet on $G'(F)$ corresponding to $\varphi$. This set should be empty precisely when $\varphi$ is not $G'$-relevant. When $G$ is split, $F$ is $p$-adic, and $Z=Z(G_\tx{der})$, Corollary \ref{cor:surjad} implies that every inner twist $\psi : G \rw G'$ can be rigidified in exactly one way. Thus, the $L$-packet of each inner twist of $G$ appears exactly once in the compound $L$-packet $\Pi_\varphi$. In general there will be multiple ways to rigidify an inner twist $\psi : G \rw G'$. The conjecture stated here implies that the corresponding subsets of $\Pi_\varphi$ contain the same representations and can be explicitly identified, see \cite[\S6]{KalRIBG}.

We expect that the diagram \eqref{eq:llc+} behaves naturally with respect to the finite central subgroup $Z$. By this we mean that, if $Z \subset Z'$ is a larger finite central subgroup of $G$ defined over $F$ and if we consider the analogous diagram for $Z'$, the two left corners and the bottom right corner in the diagram for $Z$ embed naturally into the corresponding corners of the diagram for $Z'$, and we expect that the upper right corner also embeds and that the four embeddings commute with all the arrows in the diagrams.

Given $\dot\pi=(G_1,\psi_1,z_1,\pi_1) \in \Pi_\varphi$, write $\<-,\dot\pi\>$ for the conjugation-invariant function on $\pi_0(S_\varphi^+)$ given by the character of the irreducible representation $\iota_\mf{w}(\dot\pi)$. We expect that for a fixed rigid inner twist $(\psi,z) : G \rw G'$, the virtual character
\begin{equation} \label{eq:stchar} S\Theta_{\varphi,\psi,z} = e(G')\sum_{\substack{\dot\pi \in \Pi_\varphi\\ \dot\pi \mapsto [z]}} \<1,\dot\pi\>\Theta_{\dot\pi} \end{equation}
is a stable function on $G'(F)$ and is independent of $\mf{w}$, where $e(G')$ is the sign defined in \cite{Kot83}. Furthermore, for any semi-simple element $\dot s \in S_\varphi^+$, put
\begin{equation} \label{eq:ustchar} \Theta^{\dot s}_{\varphi,\mf{w},\psi,z} = e(G')\sum_{\substack{\dot\pi \in \Pi_\varphi\\ \dot\pi \mapsto [z]}} \<\dot s,\dot\pi\>\Theta_{\dot\pi}. \end{equation}
The element $\dot s$ determines a refined endoscopic datum $\mf{\dot e} = (H,\mc{H},\dot s,\eta)$ just as an element $s \in S_\varphi$ determines a usual endoscopic datum $\mf{e}=(H,\mc{H},s,\eta)$: We let $s$ be the image in $S_\varphi$ of $\dot s$, take $\hat H=\hat G_s^\circ$ and $\mc{H}=\hat H \cdot \varphi(W_F)$, and let $\eta : \mc{H} \rw {^LG}$ be the natural inclusion. By construction the image of $\varphi$ is contained in $\mc{H}$ and taking a z-pair $\mf{z}=(H_\mf{z},\eta_\mf{z})$ for $\mf{e}$ we obtain the tempered Langlands-parameter $\varphi_\mf{z} = \eta_\mf{z} \circ \varphi$. Let $\Delta'[\mf{e},\mf{z},\mf{w}] : H_{\mf{z},G\tx{-sr}}(F) \times G_\tx{sr}(F) \rw \C$ be the Whittaker normalization of the absolute transfer factor \cite[\S5.5]{KS12} corresponding to the Whittaker datum $\mf{w}$ (as defined by equation (5.5.2) in \cite{KS12}, where it is denoted by $\Delta_\lambda'$). Using the formula \eqref{eq:tf} and Proposition \ref{pro:tf} we obtain a corresponding normalization, let's call it $\Delta'[\mf{\dot e},\mf{z},\mf{w},\psi,z] : H_{\mf{z},G\tx{-sr}}(F) \times G'_\tx{sr}(F) \rw \C$, of the transfer factor for the group $G'$. Note that we must add an inverse to \eqref{eq:tf} since we are dealing with $\Delta'$ rather than $\Delta$ here. Thus we now have
\begin{equation} \label{eq:tf'} \Delta'[\mf{\dot e},\mf{z},\mf{w},\psi,z]_\tx{abs}(\gamma_\mf{z},\delta') = \Delta'[\mf{e},\mf{z},\mf{w}](\gamma_\mf{z},\delta) \cdot \<\tx{inv}(\delta,\dot\delta'),\dot s_{\gamma,\delta}\>. \end{equation}

We then expect that if $f^\mf{\dot e}$ and $f$ are smooth compactly supported functions on $H_\mf{z}(F)$ and $G'(F)$ respectively, whose orbital integrals are $\Delta'[\mf{\dot e},\mf{z},\mf{w},\psi,z]$-matching \cite[\S5.5]{KS99}, then we have
\begin{equation} \label{eq:charid} S\Theta_{\varphi_\mf{z},\tx{id},1}(f^\mf{\dot e}) = \Theta^{\dot s}_{\varphi,\mf{w},\psi,z}(f). \end{equation}

\tb{Remark:} We conclude this section with some remarks on the group $\pi_0(S_\varphi^+)$ and an example involving $\tx{SL}_2$. We will see in Section \ref{sec:real} that when the ground field $F$ is the field of real numbers, the finite group $\pi_0(S_\varphi^+)$ is always abelian. This statement follows from Shelstad's result \cite[(5.4.5)]{She82} and is part of Proposition \ref{pro:sphireal}. However, contrary to the finite abelian group $\mb{S}_\varphi^\tx{ad}=\pi_0(S_\varphi/Z(\hat G)^\Gamma)$ commonly used to parameterize $L$-packets for real groups, the group $\pi_0(S_\varphi^+)$ will in general not be an elementary 2-group. This happens already for discrete series parameters for the group $\tx{SU}_2/\R$, in which case the group $\pi_0(S_\varphi^+)$ coincides with the group $\mb{S}_\varphi^\tx{sc}$ introduced by Shelstad and is a cyclic group of order $4$ \cite[\S11]{SheTE3}.

On the other hand, when the ground field $F$ is $p$-adic, the finite group $\pi_0(S_\varphi^+)$ need not be abelian. While this phenomenon was already present in the usual finite group $\mb{S}_\varphi^\tx{ad}$, it is much more prevalent with $\pi_0(S_\varphi^+)$ and already occurs for $G=\tx{SL}_2/F$, as the following example shows.

\tb{Example:} We will now discuss the case of a particular parameter $\varphi$ for the group $G=\tx{SL}_2/F$ over a $p$-adic field $F$ and see that Diagram \eqref{eq:llc+} and Equations \eqref{eq:stchar},\eqref{eq:ustchar},\eqref{eq:charid} are in accordance with the known behaviour of endoscopy in that case. This example is also discussed in \cite{She79c} and \cite{Art06}. We let $Z$ denote the center of $G$. According to Corollary \ref{cor:surjad}, $H^1(u \rw W,Z \rw G) \cong H^1(\Gamma,G_\tx{ad}) \cong \Z/2\Z$, so there are two equivalence classes of rigid inner twists of $G$, one corresponding to $G=\tx{SL}_2$ and one corresponding to the unique inner form $G'$ of $G$, i.e. $G'(F)=\tx{SL}_1(D)$, where $D$ is the quaternion algebra over $F$.

We have $\hat G=\tx{PGL}_2(\C)$. Let $E/F$ be a quadratic extension, and $\theta : E^\times \rw \C^\times$ a character having the property that $\theta^{-1} \cdot (\theta\circ\sigma)$ is a character of order $2$, where $\sigma \in \Gamma_{E/F}$ is the non-trivial element. Fix an element $\sigma^\circ \in W_{E/F}$ that maps to $\sigma$. Then, as discussed in \cite[\S11]{She79c}, we obtain a tempered Langlands parameter $\varphi : W_{E/F} \rw \hat G$ by the rule
\[ \varphi(e) = \begin{bmatrix} \theta(e)&0\\ 0&\theta(\sigma(e)) \end{bmatrix}, e \in E^\times,\qquad \varphi(\sigma^\circ) = \begin{bmatrix} 0&1\\ 1&0 \end{bmatrix}. \]
As argued there, the subgroup $S_\varphi \subset \hat G=\tx{PGL}_2(\C)$ is given by
\[ S_\varphi = \left\{ \begin{bmatrix} 1&0 \\ 0&1 \end{bmatrix},\begin{bmatrix} -1&0 \\ 0&1 \end{bmatrix},\begin{bmatrix} 0&1 \\ 1&0 \end{bmatrix}, \begin{bmatrix} 0&1 \\ -1&0 \end{bmatrix} \right\}. \]
It follows that the subgroup $S_\varphi^+ \subset \hat{\bar G} = \tx{SL}_2(\C)$ is given by
\[ \left\{\!\begin{bmatrix} 1&0 \\ 0&1 \end{bmatrix}\!,\!\begin{bmatrix} -1&0 \\ 0&-1 \end{bmatrix}\!,\!\begin{bmatrix} -i&0 \\ 0&i \end{bmatrix}\!,\!\begin{bmatrix} i&0 \\ 0&-i \end{bmatrix}\!,\!\begin{bmatrix} 0&i \\ i&0 \end{bmatrix}\!,\!\begin{bmatrix} 0&-i \\ -i&0 \end{bmatrix}\!,\!\begin{bmatrix} 0&1 \\ -1&0 \end{bmatrix}\!,\!\begin{bmatrix} 0&-1 \\ 1&0 \end{bmatrix}\!\right\}. \]
Thus $S_\varphi^+$ is the quaternion group. The group $Z(\hat{\bar G})^+$ is equal to the center of $\tx{SL}_2(\C)$ and is of order two, with non-trivial element $-\tb{1}_2$. The group $S_\varphi^+$ has exactly five irreducible representations, four of which $\rho_1,\dots,\rho_4$ are of dimension $1$ and kill the central element $-\tb{1}_2$, while the fifth $\rho_5$ is of dimension $2$ and its central character sends $-\tb{1}_2$ to $-1$.

Diagram \eqref{eq:llc+} now implies that the compound $L$-packet $\Pi_\varphi$ should have five elements, four of which should correspond to irreducible representations $\pi_1$, $\pi_2$, $\pi_3$, $\pi_4$ of $G(F)=\tx{SL}_2(F)$, and the fifth should correspond to an irreducible representation $\pi_5$ of $G'(F)=\tx{SL}_1(D)$. Indeed, it is argued in \cite[\S12]{She79c} that the $L$-packet $\Pi_\varphi^G$ contains four members, and the $L$-packet $\Pi^{G'}_\varphi$ contains a single element. It is furthermore argued that $\sum_{i=1}^4 \Theta_{\pi_i}$ is a stable distribution on $\tx{SL}_2(F)$ and that $\Theta_{\pi_5}$, hence also $2\Theta_{\pi_5}$, is a stable distribution on $\tx{SL}_1(D)$, in accordance with equation \eqref{eq:stchar}. With regards to endoscopic transfer, it is argued in \cite[\S12]{She79c} that if we take the endoscopic element $1=s \in S_\varphi$, so that the associated endoscopic group $H$ is equal to $G$, then the stable distribution $\sum_{i=1}^4 \Theta_{\pi_i}$ on $H(F)$ transfers to the stable distribution $2\Theta_{\pi_5}=\pm e(G')\tx{tr}\rho_5(s)\Theta_{\pi_5}$ on $G'(F)$. The sign-ambiguity comes from the fact that the normalizations of the transfer factors in \cite{She79c} are left somewhat arbitrary, and this corresponds precisely to the fact that the character values of $\tx{tr}\rho_5$ on the two lifts $\tb{1}_2,-\tb{1}_2 \in S_\varphi^+$ of $1 \in S_\varphi$ differ by a sign. On the other hand, if we take an endoscopic element $1 \neq s \in S_\varphi$, so that the associated endoscopic group $H$ is an anisotropic torus, then the corresponding stable distribution on $H(F)$, which is just given by a character, transfers to the distribution $\sum_{i=1}^4 \tx{tr}\rho_i(s)\Theta_{\pi_i}$ on $G(F)$ (which makes sense since the characters $\rho_1,\dots,\rho_4$ descend to $S_\varphi$), and transfers to the zero distribution on $G'(F)$. For any lift $\dot s \in S_\varphi^+$ of $s$, we have $\tx{tr}(\rho_5)(\dot s)=0$. Thus the results of \cite[\S12]{She79c} are consistent with Equation \eqref{eq:charid} in this case.

\subsection{A note on the relevance of parameters} \label{sec:rel}

An examination of Diagram \eqref{eq:llc+} reveals that part of our claim is the following: If a tempered Langlands parameter $\varphi$ is relevant for an inner form $G'$ of $G$ which is part of a rigid inner twist $(G',\psi,z)$, then the character of $\pi_0(Z(\hat{\bar G})^+)$ that corresponds to the class of $z$ must annihilate the kernel of the map
\begin{equation} \label{eq:centinj} \pi_0(Z(\hat{\bar G})^+) \rw \pi_0(S_\varphi^+). \end{equation}
In fact, something more precise is true. Given a $K$-group of rigid inner twists of $G$, we will say that $\varphi$ is relevant for that $K$-group if $\varphi$ is relevant for some inner form of $G$ that occurs in that $K$-group. Recall that the isomorphism classes of $K$-groups of rigid inner twists of $G$ that are realized by $Z$ are in bijection with $H^1_\tx{sc}(u \rw W,Z \rw G)$ and thus by Corollary \ref{cor:tn+pair} in bijection with certain characters of $\pi_0(Z(\hat{\bar G})^+)$.

\begin{lem} \label{lem:centinj} The characters of $\pi_0(Z(\hat{\bar G})^+)$ which annihilate the kernel of \eqref{eq:centinj} are precisely those which correspond to $K$-groups of rigid inner twists of $G$ for which $\varphi$ is relevant.
\end{lem}

Before we prove the lemma, it would be useful to illuminate some of the structure of $\pi_0(S_\varphi^+)$ and $\pi_0(Z(\hat{\bar G})^+)$ in the case where $Z(G_\tx{der}) \subset Z$. In that case we have $\hat{\bar G} = [\hat G]_\tx{sc} \times \hat C$, where $\hat C$ is a complex torus with an isogeny to $Z(\hat G)^\circ$. If we let $\hat C^+$ be the preimage of $Z(\hat G)^{\circ,\Gamma}$ under that isogeny, then $Z(\hat{\bar G})^\circ = \hat C$ implies $Z(\hat{\bar G})^{+,\circ} = \hat C^{+,\circ}$. We conclude that we have an injection
\[ Z([\hat G]_\tx{sc})^+ \times \pi_0(\hat C^+) \irw \pi_0(Z(\hat{\bar G})^+), \]
where $Z([\hat G]_\tx{sc})^+$ denotes the preimage in $[\hat G]_\tx{sc}$ of $Z(\hat G)^\Gamma$. Turning to $S_\varphi^+$, we have
\[ S_\varphi^+ = \{ g \times c| g \in [\hat G]_\tx{sc}, c \in \hat C, \forall w \in W_F': \varphi(w)\bar g\varphi(w)^{-1}\bar g^{-1} = \bar c\sigma(\bar c^{-1}) \}, \]
where $\bar g$ and $\bar c$ denote the images of $g$ and $c$ in $\hat G$ and $\sigma \in \Gamma$ denotes the image of $w$. Thus we obtain a continuous map
\[ S_\varphi^+ \rw Z^1(\Gamma,Z([\hat G]_\tx{der})),\qquad g\times c \mapsto (\sigma \mapsto \bar c\sigma(\bar c^{-1})), \]
whose kernel is equal to $S_\varphi^{\tx{sc},+} \times \hat C^+$, where $S_\varphi^{\tx{sc},+}$ is the preimage in $[\hat G]_\tx{sc}$ of $S_\varphi$. The group $Z^1(\Gamma,Z([\hat G]_\tx{der}))$ is finite and this tells us that this kernel must contain the neutral component of $S_\varphi^+$. Applying a similar argument to the group $S_\varphi^{\tx{sc},+}$
we find that its neutral component is contained in the group $S_\varphi^\tx{sc}$ of elements of $[\hat G]_\tx{sc}$ which are fixed under $\tx{Ad}(\varphi(W'))$. Thus $S_\varphi^{+,\circ} = S_\varphi^{\tx{sc},\circ} \times \hat C^{+,\circ}$ and we conclude that we have an injection
\[ \pi_0(S_\varphi^\tx{sc}) \times \pi_0(\hat C^+) \irw \pi_0(S_\varphi^+). \]
We now come to the proof of the lemma.
\begin{proof}
First, we will reduce to the case $Z(G_\tx{der}) \subset Z$ so that we can use the above discussion. Set $_1Z = Z(G_\tx{der})\cdot Z$ and let us temporarily use the subscript $1$ on the left for objects built with respect to $_1Z$ rather than $Z$. Then it is not hard to see that the surjection $\pi_0(Z(_1\hat{\bar G})^+) \rw \pi_0(Z(\hat{\bar G})^+)$ restricts to a surjection from the kernel $_1K$ of the version of \eqref{eq:centinj} taken with respect to $_1Z$ to the kernel $K$ of the version of \eqref{eq:centinj} taken with respect to $Z$. Thus a character of $\pi_0(Z(\hat{\bar G})^+)$ kills $K$ if and only if its pullback to $\pi_0(_1Z(\hat{\bar G})^+)$ kills $_1K$. This reduces the proof to the case $Z(G_\tx{der}) \subset Z$, which we from now on assume.

Fix a Borel pair $(\hat T,\hat B)$ of $\hat G$. Let $\hat M$ be a standard Levi subgroup of $\hat G$ with the property that $\hat M \rtimes W_F$ contains the image of some $\hat G$-conjugate $\varphi$ and assume that $\hat M$ is minimal with this property. Replace $\varphi$ by a conjugate whose image lies in $\hat M \rtimes W_F$. This does not change the kernel of \eqref{eq:centinj}. Now $(S_\varphi \cap \hat M)^\circ$ is a maximal torus of $S_\varphi^\circ$ and equals $Z(\hat M)^{\Gamma,\circ}$. The cokernel of $S_\varphi^{\tx{sc},\circ} \rw S_\varphi^\circ$ is a torus and we conclude that the connected component of the preimage in $S_\varphi^{\tx{sc},\circ}$ of the maximal torus $Z(\hat M)^{\Gamma,\circ}$ is a maximal torus. This maximal torus is thus equal to $Z(\hat M_\tx{sc})^{\Gamma,\circ}$, where $\hat M_\tx{sc}$ is the Levi subgroup of $[\hat G]_\tx{sc}$ corresponding to $\hat M$ (and not the simply connected cover of $\hat M$).

Now consider the kernel of \eqref{eq:centinj}. It is represented by elements of
\[ Z(\hat{\bar G})^+ \cap S_\varphi^{+,\circ} = (Z([\hat G]_\tx{sc})^+ \cap S_\varphi^{\tx{sc},\circ}) \times \hat C^{+,\circ} \]
and hence equals to $Z([\hat G]_\tx{sc})^+ \cap S_\varphi^{\tx{sc},\circ}$. This group is central in $S_\varphi^{\tx{sc},\circ}$ and hence is contained in every maximal torus of that connected reductive group. We conclude that the kernel of \eqref{eq:centinj} is given by
\[ Z([\hat G]_\tx{sc})^+ \cap Z(\hat M_\tx{sc})^{\Gamma,\circ} = Z([\hat G]_\tx{sc})^\Gamma \cap Z(\hat M_\tx{sc})^{\Gamma,\circ}. \]
Applying \cite[Lemma 1.1]{Art99} to the group $[\hat G]_\tx{sc}$ and its Levi $\hat M_\tx{sc}$ we conclude that the characters of $\pi_0(Z(\hat{\bar G})^+)$ which annihilate the kernel of \eqref{eq:centinj} are precisely those whose restriction to $Z([\hat G]_\tx{sc})^\Gamma$ is the inflation of a character of $\pi_0(Z(\hat M_\tx{sc})^\Gamma)$. An application of \cite[Thm. 1.2]{Kot86} finishes the proof.
\end{proof}

\subsection{$L$-packets and endoscopic transfer for real groups} \label{sec:real}
When the ground field $F$ is the field of real numbers, the local Langlands correspondence, including the internal structure of $L$-packets and the transfer of distributions, is very well understood by the work of many mathematicians, including Adams, Barbasch, Johnson, Langlands, Shelstad, and Vogan. The purpose of this section is to show how these results imply the validity of the expectations formulated in Section \ref{sec:lpack}. Given the mature state of the theory there will be little more for us to do then to combine well-known arguments and give appropriate references. Our focus in this section will be on tempered $L$-packets and we will primarily use the references \cite{Lan73}, \cite{She79a}, \cite{She79b}, \cite{She81}, \cite{She82}, \cite{SheTE1}, \cite{SheTE2}, \cite{SheTE3}. Given the comparison between strong real forms and rigid inner twists of Section \ref{sec:strong}, one could ask further how the exposition of this section relates to the geometrically constructed parameterization of $L$-packets of \cite{ABV92}, how the $L$-packets on certain covering groups introduced in \cite{AV92} fit into the picture, and how non-tempered $L$-packets on real groups, in particular the cohomological ones from \cite{AJ87}, can be incorporated. We leave the discussion of these questions to a separate paper.

Let $G$ be a connected reductive group defined and quasi-split over $\R$, let $Z \subset G$ be a finite central subgroup,  and let $\varphi : W_\R \rw {^LG}$ be a tempered Langlands parameter. We first construct $\Pi_\varphi$. Let $(G',\psi,z)$ be a rigid inner twist of $G$ realized by $Z$. When $\varphi$ is relevant for $G'$, Langlands has constructed in \cite{Lan73} the $L$-packet $\Pi_\varphi^{G'}$ for the group $G'$ corresponding to $\varphi$. When $\varphi$ is not relevant, we take $\Pi_\varphi^{G'}$ to be the empty set. We obtain a map
\[ \Pi_\varphi^{G'} \rw \Pi^\tx{rig}_\tx{temp}(G),\qquad \pi' \mapsto (G',\psi,z,\pi'). \]
According to Fact \ref{fct:autinn}, this map is an injection. We define $\Pi_\varphi$ to be the union of the images of these maps for $(G',\psi,z)$ varying over the (isomorphism classes of) rigid inner twists of $G$ realized by $Z$.

Next we turn to the bijection $\iota_\mf{w}$ of Diagram \eqref{eq:llc+}. For each $\dot\pi = (G',\psi,z,\pi') \in \Pi_\varphi$ and $\dot s \in \pi_0(S_\varphi^+)$ we will define a complex number $\<\dot s,\dot \pi\> \in \C^\times$. We will then argue that
\begin{enumerate}
\item The function $\dot s \mapsto \<\dot s,\dot \pi\>$ is a character of $\pi_0(S_\varphi^+)$.
\item The character $\<-,\dot\pi\>$ is trivial if and only if $\dot\pi=(G,\tx{id},1,\pi)$ and $\pi$ is $\mf{w}$-generic.
\item The characters $\<-,\dot\pi_1\>$ and $\<-,\dot\pi_2\>$ are equal if and only if $\dot\pi_1$ and $\dot\pi_2$ are isomorphic as defined in Section \ref{sec:lpack}.
\item All characters of $\pi_0(S_\varphi^+)$ are obtained in this way.
\end{enumerate}
The bijection $\iota_\mf{w}$ will then be determined by $\iota_\mf{w}(\dot\pi)(\dot s) = \<\dot s,\dot\pi\>$. Before we go into the definition of the pairing $\<-,-\>$, let us recall the relationship between $G$ and $\hat G$ (see e.g. \cite[\S3]{Vog93}). For a $\Gamma$-invariant Borel pair $(T,B)$ of $G$ we have the based root datum $\tx{brd}(T,B)=(X,\Delta,Y,\Delta^\vee)$, where $X=X^*(T)$, $\Delta \subset X$ is the set of $B$-simple roots, $Y=X_*(T)$, and $\Delta^\vee \subset Y$ is the set of $B$-simple coroots. Any two $\Gamma$-invariant Borel pairs $(T_1,B_1)$ and $(T_2,B_2)$ are conjugate under $G(F)$, and any element $g \in G(F)$ with $\tx{Ad}(g)(T_1,B_1)=(T_2,B_2)$ induces the same $\Gamma$-equivariant isomorphism $\tx{brd}(T_1,B_1) \rw \tx{brd}(T_2,B_2)$. We define $\tx{brd}(G)$ to be the limit of the system $\{\tx{brd}(T,B)\})$ where $(T,B)$ runs over all Borel pairs of $G$. By construction, $\hat G$ is a complex reductive group, endowed with an algebraic action of $\Gamma$ subject to the condition that there exists a $\Gamma$-invariant splitting of $\hat G$. It carries the following additional datum that determines its relationship with $G$: We construct $\tx{brd}(\hat G)$ in the same way as for $G$. Here the existence of an element of $\hat G^\Gamma$ that conjugates two $\Gamma$-invariant Borel pairs follows from \cite[Cor 1.7]{Kot84}. The additional datum is then an isomorphism $\tx{brd}(G)^\vee \cong \tx{brd}(\hat G)$. We will write $\tx{brd}(G)=(X,\Delta,Y,\Delta^\vee)$, and consequently $\tx{brd}(\hat G)=(Y,\Delta^\vee,X,\Delta)$.

We now come to the construction of the pairing $\<-,-\>$. It will be useful and instructive to first consider the case of an equivalence class of discrete parameters $\{\varphi\} : W_\R \rw {^LG}$. Fix a $\Gamma$-invariant Borel pair $(\hat T,\hat B)$ of $\hat G$. We may choose a representative $\varphi$ of its equivalence class so that $\varphi(W_\R)$ normalizes $\hat T$. There exists $\mu \in X_*(\hat T) \otimes \C$ such that for $z \in \C^\times \subset W_\R$ and $\chi \in X^*(\hat T)$ we have
\[ \chi(\varphi(z)) = z^{\<\chi,\mu\>} \cdot \bar z^{\<\chi,\bar\mu\>}. \]
Here $\bar\mu=\varphi(\sigma)\mu$, with $\sigma \in \Gamma$ being complex conjugation. The element $\mu$ is regular and we may choose $\varphi$ so that it is $\hat B$-dominant. Then $\mu$ is uniquely specified by the equivalence class of $\varphi$, and the representative $\varphi$ is uniquely determined up to conjugation by $\hat T$. Now $\varphi(\sigma) = n \rtimes \sigma$, with $n \in N(\hat T)$. Choose $\lambda \in X_*(\hat T) \otimes \C$ with the property that for all $\chi \in X^*(\hat G) \subset X^*(\hat T)$ the equality $\chi(n)=\exp(2\pi i \<\chi,\lambda\>)$ holds. Then $\lambda$ is well defined up to an element of $X_*(\hat T) + \{x-\varphi(\sigma)x|x \in X_*(\hat T)\otimes \C\}$. If we chose a different Borel pair $(\hat T',\hat B')$ and performed these operation based on it, the resulting $\mu',\lambda'$ will coincide with the image of $\mu,\lambda$ under the canonical $\Gamma$-invariant isomorphism $(\hat T,\hat B) \rw (\hat T',\hat B')$. Thus we obtain a canonical pair of elements $\mu,\lambda \in X \otimes \C$, with $\lambda$ well-defined up to the ambiguity stated above. We furthermore obtain a new $\Gamma$-action on $X$, where $\sigma$ acts as $\tx{Ad}(\varphi(\sigma))$. There is a unique real torus $S$ with $X^*(S)=X$ with the new $\Gamma$-action. It is not not a-priori a maximal torus of $G$ defined over $\R$, but it does come with a set of simple roots, namely $\Delta$. It moreover comes with a set of embeddings into any inner form of $G$ as a maximal torus. To see this, fix any $\Gamma$-invariant Borel pair $(T,B)$ of $G$ and identify $\tx{brd}(G)$ with $\tx{brd}(T,B)$. This identifies $S$ with $T$ over $\C$ and thus gives an embedding $i : S \rw G$ defined over $\C$ whose $G(\C)$-conjugacy class is $\Gamma$-invariant. According to \cite[Cor 2.2]{Kot82}, there exists a $G(\C)$-conjugate of this embedding defined over $\R$. The set of embeddings $S \rw G$ defined over $\R$ and conjugate to $i$ over $G(\C)$ is independent of the choice of $(T,B)$. Moreover, for any inner twist $\psi : G \rw G'$ the embedding $\psi\circ i : S \rw G'$ defined over $\C$ also has a $G'(\C)$-conjugate defined over $\R$, \cite[Lemma 2.8]{She79b}, and the set of embeddings $S \rw G'$ satisfying this condition is also independent of the choice of $(T,B)$. We will call the embeddings $S \rw G'$ obtained in this way \emph{admissible}.

Now fix an inner twist $\psi : G \rw G'$. For each embedding $\eta : S \rw G'$ defined over $\R$ we obtain an essentially discrete series representation $\Theta(\varphi,\eta)$. Namely, the embedding $\eta$ provides us with the images of $\mu$, $\lambda$, and the positive Weyl chamber $\Psi$ corresponding to the set $\Delta$ of simple roots, and we take $\Theta(\varphi,\eta)$ to be the unique essentially discrete series representation whose character restricted to $\eta(S)(\R)_\tx{reg}$ is given by the function $\Theta(\eta(\mu),\eta(\lambda),\eta(\Psi))$ given in \cite[\S4.3]{She82}. The representations $\Theta(\varphi,\eta_1)$ and $\Theta(\varphi,\eta_2)$ are equivalent if and only if $\eta_1,\eta_2 : S \rw G'$ are conjugate under $G'(\R)$. Thus $\eta \mapsto \Theta(\varphi,\eta)$ is a bijection from the set of $G'(\R)$-conjugacy classes of admissible embeddings $\eta : S \rw G'$ to the $L$-packet $\Pi_\varphi^{G'}$.

We can now define $\<-,-\>$ in the case where $\varphi$ is a discrete parameter. By results of Kostant \cite{Kos78} and Vogan \cite{Vog78} there exists a unique $\mf{w}$-generic representation $\pi_\mf{w}$ in the $L$-packet $\Pi_\varphi^G$. Let $\dot\pi_\mf{w}=(G,\tx{id},1,\pi_\mf{w}) \in \Pi_\varphi$ and let $\dot\pi = (G',\psi,z,\pi')$ be a constituent of $\Pi_\varphi$. The representation $\pi_\mf{w}$  corresponds to a $G(\R)$-conjugacy class of admissible embeddings $\eta_\mf{w} : S \rw G$, and the representation $\pi'$ corresponds to a $G'(\R)$-conjugacy class of admissible embeddings $\eta_{\dot\pi} : S \rw G'$. Choose $g \in G(\C)$ such that $\eta_{\dot\pi} = \psi\circ\tx{Ad}(g)\circ\eta_\mf{w}$ and set $\tx{inv}(\dot\pi_\mf{w},\dot\pi)=\tx{inv}(\eta_\mf{w},\eta_{\dot\pi})$ to be the class of $w \mapsto g^{-1}z(w)w(g)$ in $H^1(u \rw W,Z \rw S)$. Then we obtain a bijection
\[ \Pi_\varphi \rw H^1(u \rw W,Z \rw S),\qquad \dot\pi \mapsto \tx{inv}(\dot\pi_\mf{w},\dot\pi). \]
On the other hand, the regularity of $\mu$ implies that $S_\varphi \subset \hat T$ and we obtain $S_\varphi = \hat S^\Gamma$, where $\hat S$ is the dual torus to $S$, which we identify with the complex torus $\hat T$ equipped with the Galois action given by $\tx{Ad}(\varphi(\sigma))$. For an element $\dot s \in S_\varphi^+ = [\hat{\bar S}]^+$ we set
\begin{equation} \label{eq:dspair} \<\dot s,\dot\pi\> = \<\dot s,\tx{inv}(\dot\pi_\mf{w},\dot\pi)\>, \end{equation}
where the pairing on the left is the one we are defining, and the pairing on the right is the perfect pairing of Corollary \ref{cor:tn+pair}. The validity of points 1-4 above is visible from the construction: Point 1 is obvious, point 2 is the uniqueness of the $\mf{w}$-generic constituent of $\Pi_\varphi^G$, and points 3-4 come from the bijection between $\Pi_\varphi$ and $H^1(u \rw W,Z \rw S)$, as well as the fact that the pairing of Corollary \ref{cor:tn+pair} is perfect.

We have thus constructed Diagram \eqref{eq:llc+} for a discrete parameter $\varphi$. If $Z \subset {_1Z} \subset G$ is a larger finite central subgroup, then, using the subscript $1$ for objects built with respect to $_1Z$ rather than $Z$, we have a diagram
\[ \xymatrix{
\Pi_\varphi\ar[r]\ar@{^(->}[d]&H^1(u \rw W,Z \rw S)\ar@{^(->}[d]\ar[r]&\tx{Irr}(S_\varphi^+)\ar@{^(->}[d]\\
_1\Pi_\varphi\ar[r]&H^1(u \rw W,{_1Z} \rw S)\ar[r]&\tx{Irr}({_1S_\varphi^+})
} \]
whose commutativity follows from the compatibility of $\tx{inv}$ with the inclusion $H^1(u \rw W,Z \rw S) \rw H^1(u \rw W,{_1Z} \rw S)$ discussed in Section \ref{sec:rigid}, and the functoriality of the pairing of Corollary \ref{cor:tn+pair}. We conclude that Diagram \eqref{eq:llc+} for $Z$ embeds into Diagram \eqref{eq:llc+} for $_1Z$, as expected.

We now turn to an equivalence class $\{\varphi\} : W_\R \rw {^LG}$ of tempered (but not necessarily discrete) parameters. We fix again $\Gamma$-invariant Borel pairs $(\hat T,\hat B)$ of $\hat G$ and $(T,B)$ of $G$. We have a 1-1 correspondence between standard parabolic subgroups of $G$ and $\hat G$, as well as between standard Levi subgroups. Fix a standard Levi subgroup $\hat M \subset \hat G$ such that some representative $\varphi$ within its equivalence class has image belonging to $^LM = \hat M \rtimes W_\R$ and such that $\hat M$ is minimal with this property. The corresponding standard Levi subgroup $M \subset G$ is cuspidal and $\varphi$ is a discrete parameter for it. We have
\[ R(\hat T,\hat M) = \{ \alpha^\vee \in R(\hat T,\hat G)| \varphi(\sigma)\alpha^\vee = -\alpha^\vee \} \]
Let $S_\varphi = \tx{Cent}(\varphi(W_\R),\hat G)$. This is an algebraic subgroup of $\hat G$ whose connected component is reductive. Using Langlands' formulation of the Knapp-Stein $R$-group, Shelstad argues in \cite[\S5.3]{She82}  as follows: $\hat B_\varphi := (\hat B \cap S_\varphi)^\circ$ is a Borel subgroup of $S_\varphi^\circ$ and $\hat T_\varphi := (\hat T \cap S_\varphi)^\circ$ is a maximal torus of $S_\varphi^\circ$ whose Lie-algebra $\mf{t}_\varphi$ is the set of fixed points of $\tx{Ad}(\varphi(W_\R))$ in the Lie algebra
$\mf{t} = X_*(\hat T)\otimes \C$ of $\hat T$. The subgroup
\[ \Omega_\varphi(\hat T,\hat G) = \frac{N(\hat T,\hat G) \cap S_\varphi}{\hat T \cap S_\varphi} \]
of the Weyl group $\Omega(\hat T,\hat G)$ acts on $\hat T_\varphi$ and decomposes as a semi-direct product
\[ \Omega_\varphi(\hat T,\hat G) = \Omega(\hat T_\varphi,S_\varphi^\circ) \cdot R_\varphi, \]
where $R_\varphi$ is the subgroup of $\Omega_\varphi(\hat T,\hat G)$ whose action on $\hat T_\varphi$ preserves the $\hat B_\varphi$-positive chamber.

The group $R_\varphi$ governs the reducibility of the representations induced parabolically from the constituents of the $L$-packet of discrete series representations of inner forms of $M$ corresponding to $\varphi$. To facilitate the study of this reducibility, Shelstad introduces a second Levi subgroup of $G$, which contains $M$, and a second representative of the equivalence class of $\varphi$. This construction starts with the root system
\[ \Delta_\varphi^\vee = \{ \alpha^\vee \in R(\hat T,\hat G)| \<\mu,\alpha^\vee\> = 0, \sum_{r \in R_\varphi} r\alpha^\vee = 0 \}. \]
Shelstad argues in \cite[Lemma 5.3.13]{She82} that this is a root system of type $A_1 \times \dots \times A_1$, that $\tx{Ad}(\varphi(\sigma))$ fixes each root in it and acts by $-1$ on the corresponding root space, and that $R_\varphi$ is contained in the Weyl group of that root system. For each $\alpha^\vee \in \Delta_\varphi^\vee$, Shelstad chooses root vectors $X_{\alpha^\vee},X_{-\alpha^\vee}$ with $[X_{\alpha^\vee},X_{-\alpha^\vee}]=H_{\alpha^\vee}$ and sets $s_{\alpha^\vee} = \exp(\frac{i\pi}{4}(X_{\alpha^\vee}+X_{-\alpha^\vee})) \in \hat G$. Then $s_{\alpha^\vee}\cdot[\tx{Ad}(\varphi(\sigma))s_{\alpha^\vee}]=s_{\alpha^\vee}^2=w_{\alpha^\vee}$, where $w_{\alpha^\vee}$ is the reflection with respect to $\alpha^\vee$. Since any two non-proportional roots $\alpha^\vee,\beta^\vee \in \Delta_\varphi^\vee$ are strongly orthogonal, the elements $s_{\alpha^\vee}$ and $s_{\beta^\vee}$ commute. Set $s=\prod_{\alpha^\vee\in \Delta_\varphi^\vee}s_{\alpha^\vee}$ and $w=s^2$. Then Shelstad defines a new Levi subgroup $\hat M_1$ of $\hat G$  (denoted by $^L\tilde M^0$ in loc. cit.) by demanding that its root system be given by
\[ \{ \alpha^\vee \in R(\hat T,\hat G)| \varphi(\sigma)\alpha^\vee = -w\alpha^\vee \}. \]
Since $R(\hat T,\hat M) \perp \Delta_\varphi^\vee$ we see that $\hat M \subset \hat M_1$. Moreover, she defines a new representative $\varphi_1$ of $\{\varphi\}$ (denoted by $\tilde\varphi$ in loc. cit.) determined by $\varphi_1 = \tx{Ad}(s)\varphi$. Its image belongs to $N(\hat T,\hat M_1)$ and $\varphi_1(\sigma)$ acts on all roots of $\hat M_1$ by $-1$ \cite[Prop. 5.4.3]{She82}. The parameter $\varphi_1$ is a limit of discrete series parameter for $M_1$.

With these constructions, the $L$-packet associated to $\{\varphi\}$ can be described as follows. Let $\psi : G \rw G'$ be an inner twist. Then $\{\varphi\}$ is relevant to $G'$ if and only if the Levi subgroup $M$ transfers to $G'$, i.e. if there exists an equivalent inner twist $\psi' : G \rw G'$ which restricts to an inner twist $\psi : M \rw M'$, where $M' \subset G'$ is a Levi subgroup. Assume that this is the case, for otherwise $\Pi_\varphi^{G'}=\emptyset$. Putting $M_1' = \psi(M_1)$, we obtain the inner twist $\psi : M_1 \rw M_1'$. Let $\Pi_\varphi^{M'}$ be the $L$-packet of essentially discrete series representations corresponding to $\varphi$. The $L$-packet $\Pi_\varphi^{M_1'}$, which by construction consists of the irreducible subrepresentations of the parabolic induction to $M_1'$ of each element of $\Pi_\varphi^{M'}$, has the following alternative description: We apply the arguments for the discrete case to the limit of discrete series parameter $\varphi_1$ for $M_1$ and obtain $\mu,\lambda \in X$. The element $\mu$ is still $\Psi$-dominant, where $\Psi$ is the positive chamber determined by the set $\Delta$ of simple roots, but $\mu$ is no longer necessarily regular. We also obtain the torus $S$ as before, together with a set of admissible embeddings into all inner forms of $M_1$. In addition to this data, we obtain further a subset $\Delta_\varphi$ of the root system contained in $X=X^*(S)$, namely the set $\{\alpha| \alpha^\vee \in \Delta_\varphi^\vee\}$. For every admissible embedding $\eta : S \rw M_1'$ there is a distribution $\Theta(\varphi_1,\eta)$ defined by coherent continuation, as explained in \cite[\S4.3]{She82}: Take a sufficiently regular and $\Psi$-positive $\nu \in X^*(S)$,  transport $\mu,\lambda,\nu,\Psi$ to the image of $\eta$ and consider the character $\Theta(\eta(\mu)+\eta(\nu),\eta(\lambda),\eta(\Psi))$ of the corresponding essentially discrete series representation of $M_1'(\R)$. By coherent continuation we obtain from this a distribution on $M_1'(\R)$, which we call $\Theta(\varphi,\eta)$. Shelstad shows \cite[Theorem 4.3.2]{She82} that this distribution is either zero or the character of an irreducible representation, and furthermore that the set of non-zero distributions $\Theta(\varphi,\eta)$ for all admissible embeddings $\eta : S \rw M_1'$ forms the $L$-packet on $M_1'$ corresponding to $\varphi$ (note that the group $G$ in the notation of that theorem is the group $M_1'$ here). The discussion after the statement of that theorem shows that if a distribution $\Theta(\varphi,\eta)$ is non-zero, then it equals $\Theta(\varphi,\eta')$ if and only if the embeddings $\eta, \eta' : S \rw M_1'$ are conjugate under $M_1'(\R)$. Using a counting argument, Shelstad shows in the discussion following \cite[Cor 5.4.13]{She82} that the parabolic induction from $M_1'(\R)$ to $G'(\R)$ of each non-zero distribution $\Theta(\varphi,\eta)$ is the character of an irreducible representation, and that the $L$-packet $\Pi_\varphi^{G'}$ consists of the set of these distributions.
Note that, since $S/Z(M_1')$ is anisotropic, the set of $G'(\R)$-conjugacy classes of admissible embeddings $S \rw G'$ coincides with the set of $M_1'(\R)$-conjugacy classes of admissible embeddings $S \rw M_1'$. Thus $\eta \mapsto \Theta(\varphi,\eta)$ is a bijection between the set of $G'(\R)$-conjugacy classes of admissible embeddings $\eta : S \rw G'$ for which $\Theta(\varphi,\eta) \neq 0$ and the $L$-packet $\Pi_\varphi^{G'}$. This can be made more precise by using results of Knapp-Zuckerman and Vogan, e.g. \cite[Lemma 7.3]{Vog79}. The root system $\Delta_\varphi^\vee$ coincides with the set of roots of $\hat M_1$ which annihilate $\mu$ \cite[Prop. 5.4.2]{She82}, and moreover the positive roots inside $\Delta_\varphi^\vee$ are $\Psi$-simple, from which it follows that $\Theta(\varphi,\eta)$ is zero if and only if for some $\alpha^\vee \in \Delta_\varphi^\vee$ the image of $\alpha$ under $\eta$, which is an imaginary root of $\eta(S)$ in $G'$, is compact. We can summarize the results of this discussion as follows:

\begin{thm}[Shelstad] For every inner twist $\psi : G \rw G'$, the map
\[ \eta \mapsto \Theta(\varphi,\eta) \]
sets up a bijection from the set of $G'(\R)$-conjugacy classes of admissible embeddings $\eta : S \rw G'$ having the property that $\eta(\Delta_\varphi)$ consists entirely of non-compact roots, to the $L$-packet $\Pi_\varphi^{G'}$.
\end{thm}
Note that this theorem also covers the case when $\varphi$ is not relevant for $G'$, because Shelstad shows \cite[Lemma 4.3.5]{She82} that in this case every admissible embedding $\eta$ carries some root of $\Delta_\varphi$ to a compact root.

Having recalled Shelstad's alternative description of the $L$-packets $\Pi_\varphi^{G'}$, we can now describe the pairing $\<-,-\>$. For this, choose again a sufficiently regular $\Psi$-dominant $\nu \in X$ and let $\Pi_{\varphi_1+\nu}^{M_1'}$ be the $L$-packet consisting of the essentially discrete series characters $\Theta(\eta(\mu)+\eta(\nu),\eta(\lambda),\eta(\Psi))$ for all admissible embeddings $\eta : S \rw M_1'$. By construction we have an injection
\[ \Pi_\varphi^{G'} \rw \Pi_{\varphi+\nu}^{M_1'}, \]
whose image consists of those essentially discrete series characters which give non-zero distributions by coherent continuation. These injections for different $G'$ can be put together to an injection
\[ \Pi_\varphi \rw \Pi_{\varphi_1+\nu}^{(M_1)}, \]
where the right hand side consists of the equivalence classes of tuples of the form $(G',\psi,z,\pi_{M_1'})$, where $(G',\psi,z)$ is a rigid inner twist of $G$ to which $M_1$ transfers, say with image $M_1'$, and $\pi_{M_1'} \in \Pi_{\varphi+\nu}^{M_1'}$.
Using the inflation-restriction sequence \eqref{eq:infres} one sees that the map
\[ H^1(u \rw W,Z \rw M_1) \rw H^1(u \rw W,Z \rw G) \]
is injective, which means that the equivalence classes of rigid inner twists of $G$ realized by $Z$ to which $M_1$ transfers are in bijection with the equivalence classes of rigid inner twists of $M_1$ realized by $Z$. Thus $\Pi_{\varphi_1+\nu}^{(M_1)}$ is simply the enlarged discrete series $L$-packet for $M_1$ constructed above and we have the bijection
\[ \Pi_{\varphi_1+\nu}^{(M_1)} \rw H^1(u \rw W,Z \rw S) \]
obtained from the already discussed parameterization of such $L$-packets.
This provides an injection
\begin{equation} \label{eq:tlpinj} \Pi_\varphi \rw H^1(u \rw W,Z \rw S) \end{equation}
whose image we can characterize as follows. If $\pi_{M_1,\mf{w}}$ is the unique $\mf{w}$-generic constituent of $\Pi_\varphi^{M_1}$, then its character is obtained by coherent continuation from the character of the unique $\mf{w}$-generic constituent of $\Pi_{\varphi_1+\nu}^{M_1}$ \cite[\S13]{SheTE3}
and hence $\dot\pi_\mf{w}=(G,\tx{id},1,\tx{Ind}_{P_1}^G(\pi_{M_1,\mf{w}}))$, which is the unique $\mf{w}$-generic constituent of $\Pi_\varphi^G$, maps to the trivial element of $H^1(u \rw W,Z \rw S)$. Moreover, the corresponding embedding $\eta_\mf{w} : S \rw G$ maps $\Delta_\varphi$ to a set of non-compact imaginary roots. The other elements $\dot\pi=(G',\psi,z,\pi') \in \Pi_\varphi$ correspond precisely to the embeddings $\eta : S \rw G'$ that have the property that $\eta(\Delta_\varphi)$ also consists of non-compact imaginary roots. Using \cite[Prop. 3.2.2]{KalEpi} we see that the image of the injection \eqref{eq:tlpinj} consists precisely of those elements $\tx{inv}(\dot\pi_\mf{w},\dot\pi) \in H^1(u \rw W,Z \rw S)$ whose images in $H^1(\R,S_\tx{ad})$ correspond under the Tate-Nakayama isomorphism to elements $x \in H^{-1}_\tx{Tate}(\R,X_*(S_\tx{ad}))$ with the property $\<x,\alpha\> \in 2\Z$ for all $\alpha \in \Delta_\varphi$.

Turning to the dual side, for each $\alpha \in \Delta_\varphi$ we consider the element $\alpha(-1) \in \hat T_\tx{sc}$. We obtain a homomorphism of abelian groups
\[ \Omega(\Delta_\varphi) \rw \hat T_\tx{sc},\qquad s_\alpha \mapsto \alpha(-1). \]
Recall that the complex dual $\hat S$ of the torus $S$ is identified with the complex torus $\hat T$ equipped with the $\Gamma$-action given by $\tx{Ad}(\varphi(\sigma))$. Recall also from the above discussion that $\varphi(\sigma)$ sends $\alpha$ to $-\alpha$. Thus $\alpha(-1) \in \hat S_\tx{sc}^\Gamma \rw [\hat{\bar S}]^+$. The group $\pi_0([\hat{\bar S}]^+)$ is in perfect duality with $H^1(u \rw W,Z \rw S)$, and the image of $\Omega(\Delta_\varphi)$ in $\pi_0([\hat{\bar S}]^+)$ is precisely the annihilator of the image of $\Pi_\varphi$ in $H^1(u \rw W, Z \rw S)$. We conclude that the pairing between $\pi_0([\hat{\bar S}]^+)$ and $H^1(u \rw W,Z \rw S)$ descends to a perfect pairing
\[ \<-,-\> : \Pi_\varphi \times \tx{coker}\left(\Omega(\Delta_\varphi) \rw \pi_0([\hat{\bar S}]^+)\right) \rw \C^\times. \]
This is the pairing we wanted to construct, as the following proposition shows.

\begin{pro} \label{pro:sphireal} We have the exact sequence
\[ \Omega(\Delta_\varphi) \rw \pi_0([\hat{\bar S}]^+) \rw \pi_0(S_{\varphi_1}^+) \rw 1. \]
\end{pro}
\begin{proof}
For exactness at the third spot we will use the result \cite[(5.4.5)]{She82} of Shelstad that $\hat T \cap S_{\varphi_1}$ meets each connected component of $S_{\varphi_1}$. Consider the diagram
\[ \xymatrix{
\hat{\bar T} \cap S_{\varphi_1}^+\ar@{^(->}[r]\ar[d]&S_{\varphi_1}^+\ar[d]\\
\hat T \cap S_{\varphi_1}\ar@{^(->}[r]&S_{\varphi_1}
}\]
Shelstad's result implies that for any $x \in S_{\varphi_1}^+$ there exists $y \in \hat{\bar T} \cap S_{\varphi_1}^+$ such that the images of $x$ and $y$ in $S_{\varphi_1}$ belong to the same connected component. The right vertical arrow is an isogeny of algebraic groups and thus restricts to a surjection on their neutral connected components. It follows that we may modify $x$ within its connected component to achieve that the images of $x$ and $y$ in $S_{\varphi_1}$ are equal. Thus $y$ differs from $x$ by an element in the kernel of $\hat {\bar G} \rw \hat G$. This kernel is a subgroup of $\hat{\bar T} \cap S_{\varphi_1}^+$ and the exactness at the third spot is proved.

Now we turn to exactness at the second spot. First, we must show that for any $\alpha^\vee \in \Delta_\varphi^\vee$ we have $\alpha(-1) \in S_{\varphi_1}^{+,\circ}$. Recalling that $\varphi_1 = \tx{Ad}(s)\varphi$, with $s=\prod_{\beta^\vee \in \Delta_\varphi^\vee} s_{\beta^\vee}$, this is equivalent to showing $\tx{Ad}(s)^{-1}\alpha(-1) \in S_{\varphi}^{+,\circ}$. Now all $s_{\beta^\vee}$ commute with $\alpha(-1)$, because any $\beta^\vee$ not equal to $\alpha^\vee$ is strongly orthogonal to it. So we need to show $\alpha(-1) \in S_{\varphi}^{+,\circ}$. For this we observe that $\alpha \in X_*(\hat T_\tx{sc})$ is an element fixed by $\varphi(\sigma)$ and thus $\alpha \otimes (-r) \in X_*(\hat T_\tx{sc}) \otimes \C$ belongs to $\mf{t}_\varphi$ for all $r \geq 0$. The map $r \mapsto \exp(\alpha \otimes (-r))$ is a curve in $\hat T_\tx{sc} \cap S_{\varphi}^+$ connecting $1$ to $\alpha(-1)$ and this implies $\alpha(-1) \in S_\varphi^{+,\circ}$.

At this point, we know that the cokernel of the first map surjects onto the image of the second. To show that this surjection is in fact injective, we will compare the cardinalities of these two finite groups. Our discussion above shows that $\tx{coker}(\Omega(\Delta_\varphi) \rw \pi_0([\hat{\bar S}]^+))$ has cardinality equal to that of the packet $\Pi_\varphi$. On the other hand, recall the map \eqref{eq:centinj}
\[ \pi_0(Z(\hat{\bar G})^+) \rw \pi_0(S_{\varphi_1}^+). \]
The cokernel of this map is equal to Shelstad's group $\mb{S}_{\varphi_1}^\tx{ad}$. In \cite[\S7]{SheTE3}, Shelstad defines a (non-canonical) pairing between this group and the disjoint union of the $L$-packets $\Pi_\varphi^{G'}$ where $G'$ runs over a set of representatives for any given $K$-group. One of her main results, Theorem 7.5, shows that this pairing is perfect. This disjoint union of $L$-packets can be non-canonically identified with a non-trivial fiber of the composition
\begin{equation} \label{eq:lpproj} \Pi_\varphi \rw H^1_\tx{sc}(u \rw W,Z\rw G) \rw \pi_0(Z(\hat{\bar G})^+)^* \end{equation}
where the first map is $(G',\psi,z,\pi') \mapsto z$ and the second map comes from Corollary \ref{cor:tn+pair}. We conclude that
\[ |\tx{coker}(\Omega(\Delta_\varphi) \rw \pi_0([\hat{\bar S}]^+))|=|\Pi_\varphi|=|\tx{coker}(\pi_0(Z(\hat{\bar G})^+) \rw \pi_0(S_{\varphi_1}^+))|\cdot|\tx{im}(\eqref{eq:lpproj})|. \]
The proof will be complete once we show that the image of \eqref{eq:lpproj} has the same cardinality as the image of \eqref{eq:centinj}. By construction, the image of \eqref{eq:lpproj} parameterizes the isomorphism classes of those rigid $K$-groups which contain at least one rigid inner twists of $G$ for which $\varphi$ is relevant. By Lemma \ref{lem:centinj} this image coincides with the group of characters of $\pi_0(Z(\hat{\bar G})^+)$ which are trivial on the kernel of \eqref{eq:centinj} and this shows that the cardinalities of the images of \eqref{eq:centinj} and \eqref{eq:lpproj} coincide.
\end{proof}

This finishes the construction of Diagram \eqref{eq:llc+}. The compatibility of this diagram with enlargement of $Z$ follows from the already discussed case of discrete series.

We now turn to the remaining statements in Section \ref{sec:lpack}. The stability of the virtual character \eqref{eq:stchar} is one of the main results of Shelstad's work, see \cite[\S5]{She79b} and (for the setting of $K$-groups) \cite[\S5]{SheTE3}. The endoscopic character identities for tempered $L$-packets are another fundamental result of Shelstad. In order to extract them in their current formulation \eqref{eq:charid} from Shelstad's work, we will need to review Shelstad's theory of spectral transfer factors, following \cite{SheTE2} and \cite{SheTE3}. Say $(G',\psi)$ is an inner twist of $G$, $\mf{e}=(H,\mc{H},s,\eta)$ is an endoscopic datum for $G$ and $\mf{z}=(H_\mf{z},\eta_\mf{z})$ is a z-pair for it. Let $\Delta_\tx{geom}$ be an arbitrary normalization of the absolute transfer factor for the data $(G',\psi),\mf{e},\mf{z}$. Shelstad calls such a transfer factor geometric and constructs in \cite{SheTE2} a dual notion, called a spectral transfer factor. It is a function $\Delta_\tx{spec}$ that assigns a complex number to a pair $(\varphi_\mf{z},\pi')$, where $\varphi_\mf{z} : W_\R \rw {^LH_\mf{z}}$ is a tempered Langlands parameter and $\pi'$ is a tempered representation of $G'(F)$. The pair $(\varphi_\mf{z},\pi')$ is called related if $\varphi_\mf{z}=\eta_\mf{z} \circ \varphi\circ\eta^{-1}$ for a tempered Langlands parameter $\varphi : W_\R \rw {^LG}$ whose image belongs to $\eta(\mc{H})$, and furthermore $\pi' \in \Pi_\varphi^{G'}$. The value $\Delta_\tx{spec}(\varphi_\mf{z},\pi)$ is zero unless the pair $(\varphi_\mf{z},\pi)$ is related. We remark that Shelstad uses the notation $\Delta_\tx{spec}(\pi_1,\pi')$, where $\pi_1$ is an element of the $L$-packet $\Pi_{\varphi_\mf{z}}^{H_\mf{z}}$. Since the value of $\Delta_\tx{spec}(\pi_1,\pi')$ doesn't change when we vary $\pi_1$ within its packet, we have chosen the slightly different notation $\Delta_\tx{spec}(\varphi_\mf{z},\pi')$. In general both functions $\Delta_\tx{geom}$ and $\Delta_\tx{spec}$ are defined up to constant multiples, but there is a way to normalize them compatibly. For this, recall that the pair $(\varphi_\mf{z},\pi')$ is called $G$-regular if $\tx{Cent}(\varphi(\C^\times),\hat G)$ is torus. If $(\varphi_\mf{z},\pi')$ is a $G$-regular related pair, and $(\gamma_\mf{z},\delta')$ is a strongly $G$-regular related pair of elements $\gamma_\mf{z} \in H_\mf{z}(F)$ and $\delta' \in G'(F)$, then Shelstad constructs in \cite[\S12]{SheTE2} a canonical compatibility factor
\[ \Delta_\tx{comp}(\varphi_\mf{z},\pi';\gamma_\mf{z},\delta') \]
and defines that $\Delta_\tx{geom}$ and $\Delta_\tx{spec}$ are normalized compatibly if
\begin{equation} \label{eq:tfcomp} \Delta_\tx{spec}(\varphi_\mf{z},\pi') = \Delta_\tx{comp}(\varphi_\mf{z},\pi';\gamma_\mf{z},\delta')\Delta_\tx{geom}(\gamma_\mf{z},\delta') \end{equation}
for one, hence any choice of related pairs $(\varphi_\mf{z},\pi')$ and $(\gamma_\mf{z},\delta')$ as above. Thus if $\Delta_\tx{geom}$ is chosen, this determines $\Delta_\tx{spec}$ -- a-priori for all $G$-regular related pairs $(\varphi_\mf{z},\pi')$, but then in fact for all related pairs since the values of $\Delta_\tx{spec}$ on general related pairs are determined from its values on $G$-regular related pairs by means of coherent continuation of characters. Shelstad's theorem on spectral transfer \cite[Thm 5.1]{SheTE2} states that if $f^\mf{e}$ is a function on $H_\mf{z}(\R)$ and $f$ is a function on $G'(\R)$ such that $f^\mf{e}$ and $f$ have $\Delta_\tx{geom}$-matching orbital integrals, then
\[ S\Theta_{\varphi_\mf{z},\tx{id},1}(f^\mf{e}) = \sum_{\pi \in \Pi_\varphi^{G'}} \Delta_\tx{spec}(\varphi_\mf{z},\pi) \Theta_\pi(f). \]
In order to obtain equation \eqref{eq:charid} from this theorem we need to prove the following proposition.

\begin{pro} If $\Delta_\tx{geom}$ is the geometric transfer factor $\Delta'[\mf{\dot e},\mf{z},\mf{w},\psi,z]$, then the compatibly normalized spectral transfer factor $\Delta_\tx{spec}$ satisfies
\[ \Delta_\tx{spec}(\varphi_\mf{z},\pi') = e(G')\<\dot s,\dot\pi\> \]
for any $\pi' \in \Pi_\varphi^{G'}$, where $\dot\pi = (G',\psi,z,\pi')$.
\end{pro}
The appearance of the factor $e(G')$ in the right hand side of this equation comes from the fact that this factor also appears in the definition \eqref{eq:ustchar} of the right hand side of Equation \eqref{eq:charid}, while it does not appear in right hand side of Shelstad's character identity recalled above.
\begin{proof}
We need a slight change in notation in order to quote Shelstad's results more easily. Recall from \cite{KS12} that there are four different normalizations of the relative transfer factor, denoted by $\Delta$, $\Delta'$, $\Delta_D$, $\Delta_D'$. In Section \ref{sec:lpack} we use $\Delta'$, while Shelstad's results use $\Delta$. The passage from one to the other is very simple: If $\mf{\dot e} = (H,\mc{H},\dot s,\eta)$ is our refined endoscopic datum, and we set $\mf{\dot e'}=(H,\mc{H},\dot s^{-1},\eta)$, then $\mf{\dot e'}$ is another refined endoscopic datum, and we have $\Delta'[\mf{\dot e},\mf{z},\mf{w},\psi,z]=\Delta[\mf{\dot e'},\mf{z},\mf{w},\psi,z]$. In this proof we will work with $\Delta$ and $\mf{\dot e'}$.

The structure of $\Delta_\tx{spec}(\varphi_\mf{z},\pi)$ is interwoven with Shelstad's arguments that were recalled during the construction of the pairing $\<-,-\>$. Returning to these arguments, we take again the Levi subgroup $M_1$ of $G$ constructed from the parameter $\varphi$. We are assuming that $\Pi_\varphi^{G'}$ is non-empty and thus that $\varphi$ is relevant for $G'$, which implies that $M_1$ transfers to $M_1'$. Shelstad explains in \cite[\S7a]{SheTE2} that the data $\mf{e'}$ and $\mf{z}$ descend to data $\mf{e'}_{M_1}$ and $\mf{z}_{M_1}$ for the group $M_1'$. Moreover, she explains in \cite[\S12]{SheTE2} that the transfer factor $\Delta[\mf{\dot e'},\mf{z},\mf{w},\psi,z]$ gives rise to a transfer factor $\Delta_{M_1}[\mf{\dot e'},\mf{z},\mf{w},\psi,z]$ for the reductive group $M_1'$ and its endoscopic datum $\mf{e'}_{M_1}$ and $z$-pair $\mf{z}_{M_1}$ by the rule
\[ \Delta_{M_1'}[\mf{\dot e'},\mf{z},\mf{w},\psi,z](\gamma_\mf{z},\delta') = \Delta[\mf{\dot e'},\mf{z},\mf{w},\psi,z](\gamma_\mf{z},\delta') \cdot |\tx{det}(\tx{Ad}(\delta')-I)_{\mf{g'}/\mf{m_1'}}|. \]
We can apply the same construction to the Levi subgroup $M_1$ of $G$ and obtain
\[ \Delta_{M_1}[\mf{e'},\mf{z},\mf{w}](\gamma_\mf{z},\delta) = \Delta[\mf{e'},\mf{z},\mf{w}](\gamma_\mf{z},\delta) \cdot |\tx{det}(\tx{Ad}(\delta)-I)_{\mf{g}/\mf{m_1}}|. \]
Noting that the second factors on the right of these two equations coincide when $\delta \in M_1(\R)$ and $\delta' \in M_1'(\R)$ are stably conjugate, and recalling the definition \eqref{eq:tf} of $\Delta[\mf{\dot e'},\mf{z},\mf{w},\psi,z]$, we see that
\[ \Delta_{M_1'}[\mf{\dot e'},\mf{z},\mf{w},\psi,z](\gamma_\mf{z},\delta') = \Delta_{M_1}[\mf{e'},\mf{z},\mf{w}](\gamma_\mf{z},\delta) \cdot \<\tx{inv}(\delta,\dot\delta'),\dot s^{-1}_{\gamma,\delta}\>^{-1}. \]
The Whittaker datum $\mf{w}$ for $G$ serves $M_1$ as well, and a short argument, contained in the proof of \cite[Lemma 14.1]{SheTE1}, reveals that the factor $\Delta_{M_1}[\mf{e'},\mf{z},\mf{w}]$ is equal to $\Delta[\mf{e'}_{M_1},\mf{z}_{M_1},\mf{w}]$ -- the Whittaker normalization corresponding to $\mf{w}$ of the absolute transfer factor for $M_1$ and its endoscopic datum $\mf{e'}_{M_1}$ and $z$-pair $\mf{z}_{M_1}$. We conclude
\begin{equation} \label{eq:tflevid} \Delta_{M_1'}[\mf{\dot e'},\mf{z},\mf{w},\psi,z](\gamma_\mf{z},\delta') = \Delta[\mf{e'}_{M_1},\mf{z}_{M_1},\mf{w}](\gamma_\mf{z},\delta) \cdot \<\tx{inv}(\delta,\dot\delta'),\dot s_{\gamma,\delta}^{-1}\>^{-1}. \end{equation}

Recall the essentially discrete series $L$-packet $\Pi_{\varphi_1+\nu}^{M_1'}$ and the injection
\[ \Pi_\varphi^{G'} \rw \Pi_{\varphi_1+\nu}^{M_1'}. \]
If $\pi' \in \Pi_\varphi^{G'}$ maps to $\pi^{M_1'} \in \Pi_{\varphi_1+\nu}^{M_1'}$ then Shelstad defines in \cite[\S14]{SheTE2}
\[ \Delta_\tx{spec}(\varphi_\mf{z},\pi') = \Delta_\tx{spec}^{M_1'}(\varphi_{1,\mf{z}}+\nu,\pi^{M_1'}) \]
where the spectral transfer factor on the right is for the reductive group $M_1'$ and its endoscopic datum $\mf{e}_{M_1}$ and $z$-pair $\mf{z}_{M_1}$ and is compatibly normalized with the geometric transfer factor $\Delta_{M_1'}[\mf{\dot e},\mf{z},\mf{w},\psi,z]$. Let $\pi_\mf{w} \in \Pi_\varphi^G$ be the unique $\mf{w}$-generic constituent and recall that its image $\pi_\mf{w}^{M_1} \in \Pi_{\varphi_1+\nu}^{M_1}$ is the unique $\mf{w}$-generic constituent there. If we denote by $\Delta_\tx{spec}^{M_1}$ the spectral transfer factor for $M_1$ and its endoscopic datum $\mf{e}_{M_1}$ and $z$-pair $\mf{z}_{M_1}$ that is compatibly normalized with the geometric transfer factor $\Delta[\mf{e}_{M_1},\mf{z}_{M_1},\mf{w}]$, then Theorem 11.5 of \cite{SheTE3} asserts that
\[ 1 = \Delta_\tx{spec}^{M_1}(\varphi_\mf{1,z}+\nu,\pi_\mf{w}^{M_1}) = \Delta_\tx{comp}^{M_1}(\varphi_\mf{z},\pi_\mf{w}^{M_1},\gamma_\mf{z},\delta)\Delta[\mf{e'}_{M_1},\mf{z}_{M_1},\mf{w}](\gamma_\mf{z},\delta). \]
Combining this with \eqref{eq:tfcomp} and \eqref{eq:tflevid} we arrive at
\[ e(G')\Delta_\tx{spec}(\varphi_\mf{z},\pi')  =
\<\tx{inv}(\delta,\dot\delta'),\dot s_{\gamma,\delta}^{-1}\>^{-1}\frac{e(G')\Delta_\tx{comp}^{M_1'}(\varphi_{1,\mf{z}}+\nu,\pi^{M_1'};\gamma_\mf{z},\delta')}{ \Delta_\tx{comp}^{M_1}(\varphi_{1,\mf{z}}+\nu,\pi_\mf{w}^{M_1};\gamma_\mf{z},\delta)}. \]
Recalling the internal structure of $\Delta_\tx{comp}$ from \cite[\S12]{SheTE2} and noting that the spectral $\Delta_{II}$-terms for $M_1$ and $M_1'$ differ only by the sign $e(M_1')$, which by \cite{Kot83} is equal to $e(G')$, we see that the right hand side of the above expression reduces to
\[ \<\tx{inv}(\delta,\dot\delta'),\dot s_{\gamma,\delta}^{-1}\>^{-1}\frac{\Delta_\tx{comp,III}^{M_1'}(\varphi_{1,\mf{z}}+\nu,\pi^{M_1'};\gamma_\mf{z},\delta')}{ \Delta_\tx{comp,III}^{M_1}(\varphi_{1,\mf{z}}+\nu,\pi_\mf{w}^{M_1};\gamma_\mf{z},\delta)}. \]
By choosing Shelstad's ``toral data'' in the construction of $\Delta_\tx{comp,III}$ appropriately we may arrange the denominator to be equal to $1$. Doing so makes the above expression equal to
\[ \<\tx{inv}(\delta,\dot\delta'),\dot s_{\gamma,\delta}^{-1}\>^{-1}\left\<s_U^{-1},\tx{inv}\left(\frac{\pi_\mf{w}^{M_1},\pi^{M_1'}}{\delta,\delta'}\right)\right\> \]
where the second factor is constructed in \cite[\S12]{SheTE2}. It's construction is very similar to that of the analogous factor of \cite[\S3.4]{LS87} that was recalled in the proof of Proposition \ref{pro:tf}. The argument in the proof of that proposition that was used to prove Equation \eqref{eq:d1} can be applied to prove that
\[ \left\<s_U^{-1},\tx{inv}\left(\frac{\pi_\mf{w}^{M_1},\pi^{M_1'}}{\delta,\delta'}\right)\right\> = \frac{\<\dot s^{-1},\tx{inv}(\pi_\mf{w}^{M_1},\pi^{M_1'})\>^{-1}}{\<\dot s^{-1}_{\gamma,\delta},\tx{inv}(\delta,\delta')\>^{-1}}.\]
Recall now that $\<\dot s,\dot\pi\>$ was defined to be the pairing of $\dot s$ with the element of $H^1(u \rw W,Z \rw S)$ determined by $\dot\pi$ under the injection \eqref{eq:tlpinj}, and that this element is precisely $\tx{inv}(\pi_\mf{w}^{M_1},\pi^{M_1'})$.
\end{proof}

tkaletha@math.princeton.edu\\
Princeton University, Princeton, NJ 08544

\end{document}